\documentclass[11pt,reqno,a4paper]{amsart}
\RequirePackage{amsmath}
\RequirePackage{amssymb}
\usepackage{mathrsfs}
\usepackage{bm}
\usepackage[colorlinks,linkcolor=blue,anchorcolor=red,citecolor=blue]{hyperref}
\usepackage{geometry}
\newtheorem{Thm}{Theorem}[section]

\newtheorem{Lem}[Thm]{Lemma}

\numberwithin{equation}{section}

\newcommand{\1}{\mathbf{1}}
\newcommand{\R}{\mathbb{R}}
\newcommand{\Rd}{{\mathbb{R}^3}}

\renewcommand{\P}{\mathbf{P}}

\newcommand{\N}{\mathbb{N}}

\newcommand{\E}{\mathcal{E}}
\newcommand{\D}{\mathcal{D}}

\newcommand{\ve}{\varepsilon}
\newcommand{\vt}{\vartheta}

\newcommand{\dis}{\displaystyle}
\newcommand{\pa}{\partial}
\newcommand{\na}{\nabla}

\newcommand{\al}{\alpha}

\newcommand{\vertiii}[1]{{\left\vert\kern-0.25ex\left\vert\kern-0.25ex\left\vert #1 \right\vert\kern-0.25ex\right\vert\kern-0.25ex\right\vert}}

\usepackage{cite}

\allowdisplaybreaks

\newcommand{\<}{\langle}
\renewcommand{\>}{\rangle}

\newcommand{\I}{\mathbf{I}}
\newcommand{\II}{\mathbf{I}_{\pm}}
\renewcommand{\P}{\mathbf{P}}
\newcommand{\PP}{\mathbf{P}_{\pm}}

\title[The VPB and VPL system in union of cubes]{The Non-cutoff Vlasov-Poisson-Boltzmann and Vlasov-Poisson-Landau Systems in Union of Cubes}

\begin{document}
	
	\author{Dingqun Deng}
	\address{Yanqi Lake Beijing Institute of Mathematical Sciences and Applications, Tsinghua University, Beijing, People's Republic of China}
	\curraddr{}
	\email{dingqun.deng@gmail.com}
	\thanks{}

	\date{}
	
	
	\begin{abstract}
		This work concerns the Vlasov-Poisson-Boltzmann system without angular cutoff and Vlasov-Poisson-Landau system including Coulomb interaction in bounded domain, namely union of cubes. We establish the global stability, exponential large-time decay with specular-reflection boundary condition when an initial datum is near Maxwellian equilibrium. We provide the compatible specular boundary condition for high-order derivatives and a velocity weighted energy estimate.  
	\end{abstract}
	
	\maketitle

	\tableofcontents

	\section{Introduction}

	\subsection{Equation}
	
	We consider the Vlasov-Poisson-Boltzmann (VPB) and Vlasov-Poisson-Landau (VPL) systems describing the motion of plasma particles of two species in domain $\Omega$:
	\begin{equation}\label{F1}\left\{
		\begin{aligned}
			&\pa_tF_+ + v\cdot \na_xF_+ - \na_x\phi\cdot\na_vF_+ = Q(F_+,F_+)+Q(F_-,F_+),\\
			&\pa_tF_- + v\cdot \na_xF_- + \na_x\phi\cdot\na_vF_- = Q(F_+,F_-)+Q(F_-,F_-),\\
			&-\Delta_x\phi = \int_{\R^3}(F_+-F_-)\,dv,\\ 
			& F_\pm(0,x,v)=F_{0,\pm}(x,v),\quad E(0,x)=E_0(x).
		\end{aligned}\right. 
	\end{equation}
	Here, the unknown $F=[F_+,F_-]$ is the velocity distribution functions for the particles of ions $(+)$ and electrons $(-)$, respectively, at position $x\in\Omega$ and velocity $v\in\R^3$ and time $t\ge 0$. 
	The self-consistent electrostatic field takes the form $E(t,x)= -\na_x\phi(t,x)$. 
	The boundary condition for $(f,E)$ will be given in \eqref{specular} and \eqref{Neumann}. Next we introduce the collision operator $Q$ first.

	\medskip 
	For Vlasov-Poisson-Landau system, the collision operator $Q$ is given by 
	\begin{align*}
		Q(G,F)&=\nabla_v\cdot\int_{\R^3}\phi(v-v')\big[G(v')\nabla_vF(v)-F(v)\nabla_vG(v')\big]\,dv'.
	\end{align*}
	The non-negative definite matrix-valued function $\phi=[\phi^{ij}(v)]_{1\leq i,j\leq 3}$ takes the form of 
	\begin{align}\label{phiij}
		\phi^{ij}(v) = \Big\{\delta_{ij}-\frac{v_iv_j}{|v|^2}\Big\}|v|^{\gamma+2},
	\end{align}
	with $\gamma\ge -3$. It is convenient to call it {\em hard potential} when $\gamma\ge -2$ and {\em soft potential} when $-3\le\gamma<-2$. The case $\gamma=-3$ corresponds to the physically realistic Coulomb interactions; cf. \cite{Guo2002a}.

	\smallskip
	For Vlasov-Poisson-Boltzmann system, the collision operator $Q$ is defined by 
	\begin{align*}
		Q(G,F) = \int_{\R^3}\int_{\mathbb{S}^{2}} B(v-v_*,\sigma)\big[G(v'_*)F(v')-G(v_*)F(v)\big]\,d\sigma dv_*.
	\end{align*}  
	In this expression $v,v_*$ and $v',v'_*$ are velocity pairs given in terms of the $\sigma$-representation by
	\begin{align*}
		v'=\frac{v+v_*}{2}+\frac{|v-v_*|}{2}\sigma,\quad v'_*=\frac{v+v_*}{2}-\frac{|v-v_*|}{2}\sigma,\quad \sigma\in\mathbb{S}^2,
	\end{align*}
	that satisfy 
	conservation laws of momentum and energy:
	\begin{equation*}
			v+v_*=v'+v'_*,\quad 
		|v|^2+|v_*|^2=|v'|^2+|v'_*|^2.
	\end{equation*}
	The Boltzmann collision kernel $B(v-v_*,\sigma)$ depends only on $|v-v_*|$ and the deviation angle $\theta$ through $\cos\theta=\frac{v-v_*}{|v-v_*|}\cdot\sigma$. Without loss of generality we can assume $B(v-v_*,\sigma)$ is supported on $0\le\theta\le\pi/2$, since one can reduce the situation with {\it symmetrization}: $\overline{B}(v-v_*,\sigma)={B}(v-v_*,\sigma)+{B}(v-v_*,-\sigma)$. Moreover, we assume 
	\begin{equation*}
		B(v-v_*,\sigma) = |v-v_*|^\gamma b(\cos\theta),
	\end{equation*}
	and there exist $C_b>0$ and $0<s<1$ such that  
	\begin{align*}
		\frac{1}{C_b\theta^{1+2s}}\le \sin\theta b(\cos\theta)\le \frac{C_b}{\theta^{1+2s}}, \quad\forall\,\theta\in (0,\frac{\pi}{2}].
	\end{align*} 
	It is convenient call it {\em hard potential} when $\gamma+2s\ge 0$ and {\em soft potential} when $-3<\gamma+2s<0$. 
	Throughout the paper, we will assume 
	\begin{align*}
		&-3\le\gamma\le 1 \text{ for Landau case,}\\ &\max\{-3,-2s-\frac{3}{2}\}<\gamma<1-2s,\ 0<s<1\  \text{ for Boltzmann case}. 
	\end{align*} 
	Note that we consider the full range of $0<s<1$ for Boltzmann case. 
	
	
We reformulate problem \eqref{F1} near a global Maxwellian as the following. 
	Let $\mu$ be the global Maxwellian equilibrium state:
	\begin{equation*}
		\mu = \mu(v) = (2\pi)^{-3/2}e^{-|v|^2/2}. 
	\end{equation*}
	We construct a solution to \eqref{F1} of the form 
	\begin{equation*}
		F(t,x,v) = \mu + \mu^{1/2}f(t,x,v).
	\end{equation*}
	Then $f=[f_+,f_-]$ satisfies 
	\begin{equation}\label{1}\left\{
		\begin{aligned}
			&\partial_tf_\pm + v\cdot\nabla_xf_\pm \pm \frac{1}{2}\nabla_x\phi\cdot vf_\pm  \mp\nabla_x\phi\cdot\nabla_vf_\pm \pm \nabla_x\phi\cdot v\mu^{1/2} - L_\pm f = \Gamma_{\pm}(f,f),\\
			&
			-\Delta_x\phi = \int_{\R^3}(f_+-f_-)\mu^{1/2}\,dv,\\
			&f(0,x,v) = f_0(x,v),\quad E(0,x)=E_0(x), 
		\end{aligned}\right.
	\end{equation}
	where the linearized collision operator $L=[L_+,L_-]$ and nonlinear collision operator $\Gamma=[\Gamma_+,\Gamma_-]$ are given respectively by 
	\begin{equation*}
		L_\pm f = \mu^{-1/2}\Big\{2Q(\mu,\mu^{1/2}f_\pm)+Q(\mu^{1/2}(f_\pm+f_\mp),\mu)\Big\},
	\end{equation*}
	and 
	\begin{equation*}
		\Gamma_\pm(f,g) =  \mu^{-1/2}\Big\{Q(\mu^{1/2}f_\pm,\mu^{1/2}g_\pm)+Q(\mu^{1/2}f_\mp,\mu^{1/2}g_\pm)\Big\}.
	\end{equation*}
	The kernel of $L$ on $L^2_v\times L^2_v$ is the span of $\{[1,0]\mu^{1/2},[0,1]\mu^{1/2},[1,1]v\mu^{1/2},[1,1]|v|^2\mu^{1/2}\}$ and we define the projection $\P=[\P_+,\P_-]$ from $L^2_v\times L^2_v$ onto $\ker L$ to be 
	\begin{equation*}
		\P f = \Big(a_+(t,x)[1,0]+a_-(t,x)[0,1]+v\cdot b(t,x)[1,1]+(|v|^2-3)c(t,x)[1,1]\Big)\mu^{1/2},
	\end{equation*}
where functions $a_\pm,b,c$ are given by 
\begin{equation}
\begin{aligned}\label{abc}
	a_\pm &= (\mu^{1/2},f_\pm)_{L^2_v},\\
	b_j&= \frac{1}{2}(v_j\mu^{1/2},f_++f_-)_{L^2_v} ,\\
	c&=\frac{1}{12}((|v|^2-3)\mu^{1/2},f_++f_-)_{L^2_v}. 
\end{aligned}
\end{equation}
	Then for given $f$, one can decompose $f$ uniquely as the macroscopic part microscopic part:
	\begin{equation*}
		f = \P f+ (\I-\P)f. 
	\end{equation*}
	It's well-known that the solution to \eqref{1} satisfies the conservation laws on mass and energy. That is, the solution $f$ to \eqref{1} satisfies the following identities whenever it's satisfied initially at $t=0$: 
	\begin{equation}\label{conservation_bounded}\left\{\begin{aligned}
			&\int_{\Omega\times\R^3}f_+(t)\mu^{1/2}\,dvdx = \int_{\Omega\times\R^3}f_-(t)\mu^{1/2}\,dvdx =0,\\
			&\int_{\Omega\times\R^3}(f_+(t)+f_-(t))|v|^2\mu^{1/2}\,dvdx +\int_{\Omega}|E(t)|^2\,dx= 0.
		\end{aligned}\right.
	\end{equation}
	
	\subsection{Spatial Domain}
	In this paper, we consider a domain $\Omega$ that is the union of finitely many cubes:
	\begin{align}\label{Omega}
		\Omega=\cup_{i=1}^N\Omega_i,
	\end{align}  where $\Omega_i = (a_{i,1},b_{i,1})\times(a_{i,2},b_{i,2})\times(a_{i,3},b_{i,3})$ with $a_{i,j},b_{i,j}\in\R$ such that $a_{i,j}<b_{i,j}$. Then
	$\partial\Omega=\cup_{i=1}^3\Gamma_i$ is the union of three kinds of boundary $\Gamma_i$ $(i=1,2,3)$, where $\Gamma_i$ is orthogonal to axis $x_i$ and is the union of finitely many connected sets. We further assume that $\Gamma_i$ is of {\em non-zero} spherical measure. 
	Since the boundary of $\Gamma_i$'s are of {\em zero} spherical measure, we don't distinguish $\Gamma_i$ and the interior of $\Gamma_i$. Note that $\Omega$ could be {\it non-convex} and be closed to general bounded domains arbitrarily.
	
	The unit normal outer vector $n(x)$ exists on $\partial\Omega$ almost everywhere with respect to spherical measure. On the interior of $\Gamma_i$$(i=1,2,3)$, we have $n(x) = e_i$ or $-e_i$, where $e_i$ is the unit vector with $i$th-component being $1$. We will denote vectors $\tau_1(x), \tau_2(x)$ on boundary $\partial\Omega$ such that $(n(x),\tau_1(x), \tau_2(x))$ forms an unit orthonormal basis for $\R^3$ such that for $j=1,2$, $\tau_j=e_k$ or $-e_k$ for some $k$. This implies that $\pa_{\tau_j}$ is the tangent derivative on $\pa\Omega$ for $j=1,2$.

	
	The boundary of the phase space is 
	\begin{align*}
		\gamma:=\{(x,v)\in\partial\Omega\times\R^3\}.
	\end{align*}
	Denoting $n=n(x)$ to be the outward normal direction at $x\in\partial\Omega$, we decompose $\gamma$ as 
	\begin{align*}
		\gamma_- &= \{(x,v)\in\partial\Omega\times\R^3 : n(x)\cdot v<0\},\quad\text{(the incoming set),}\\
		\gamma_+ &= \{(x,v)\in\partial\Omega\times\R^3 : n(x)\cdot v>0\},\quad\text{(the outgoing set),}\\
		\gamma_0 &= \{(x,v)\in\partial\Omega\times\R^3 : n(x)\cdot v=0\},\quad\text{(the grazing set).}
	\end{align*}
	Correspondingly, we assume that $F(t,x,v)$ satisfies the {\it specular-reflection} boundary condition:
	\begin{equation*}
			F(t,x,R_xv) = F(t,x,v), \ \text{ on } \gamma_-, 
	\end{equation*}
where for $(x,v)\in\gamma$, 
\begin{equation*}
	R_xv = v - 2n(x)(n(x)\cdot v). 
\end{equation*}
This is equivalent to the {\it specular reflection} boundary condition for perturbation $f$:
\begin{equation}\label{specular}\begin{aligned}
	f(t,x,R_xv) = f(t,x,v), \ \text{ on } \gamma_-. 
\end{aligned}
\end{equation}
For the boundary condition of electric potential $\phi$, we further assume that
	\begin{align}\label{Neumann}
		\partial_{n}\phi=0, \ \text{ on }x\in \partial\Omega.  
	\end{align}
	In particular, the Poisson equation for potential $\phi$ is a pure Neumann boundary problem and we require zero-mean condition
	\begin{align*}
		\int_\Omega\int_{\R^3}(f_+-f_-)\mu^{1/2}\,dvdx=0,\quad \text{ for } t\ge 0,
	\end{align*}
	to ensure its existence, which follows from \eqref{conservation_bounded}. Also, the zero-mean condition 
	\begin{align*}
		\int_\Omega\phi(t,x)\,dx=0, \ \text{ for } t\ge 0
	\end{align*}
	ensures the uniqueness of solutions.


	%

	For the general theory of Vlasov-Poisson-Boltzmann and Vlasov-Poisson-Landau systems, we refer to \cite{Mischler2000, Guo2002, Guo2012, Duan2013} and reference therein. 
	Mischler \cite{Mischler2000} generalized the existence theory of Diperna-Lions renormalized solutions (cf. \cite{Diperna1989}) to Vlasov-Poisson-Boltzmann system for the initial boundary value problem. 
	Guo \cite{Guo2002} gives the global solution of VPB system near a global Maxwellian for the cutoff case. Guo \cite{Guo2012} establishes the global existence of VPL system with Coulomb potential by introducing a weight $e^{\pm\phi}$. Using this method, Duan-Liu \cite{Duan2013} proves the global existence for VPB system without angular cutoff. 
	   
	For the boundary theory of collisional kinetic problem such as Landau and Boltzmann equations, we refer to \cite{Cercignani1992,Hamdache1992, Mischler2000,Yang2005, Liu2006,Guo2009, Esposito2013, Guo2016,Kim2017,Cao2019, Guo2020, Dong2020}. 
	In the framework of perturbation near a global Maxwellian, initiating by Guo \cite{Guo2009}, which established the $L^2-L^\infty$ method, many results are developed for Boltzmann equation and Landau equation. For instance, Guo, Kim, Tonon and Trescases \cite{Guo2016} give regularity of cutoff Boltzmann equation with several physical boundary conditions in short time. Esposito, Guo, Kim and Marra \cite{Esposito2013} construct a non-equilibrium stationary solution. Kim and Lee \cite{Kim2017} study cutoff Boltzmann equation with specular boundary condition with external potential in $C^3$ bounded domain. Liu and Yang \cite{Liu2016} extend the result in \cite{Guo2009} to cutoff soft potential case. Cao, Kim and Lee \cite{Cao2019} prove the global existence for Vlasov-Poisson-Boltzmann with diffuse boundary condition. Guo, Hwang, Jang and Ouyang \cite{Guo2020} give the global stability of Landau equation with specular reflection boundary. Duan, Liu, Sakamoto and Strain \cite{Duan2020} prove the low regularity solution for Landau and non-cutoff Boltzmann equation in finite channel. 
Dong, Guo and Ouyang \cite{Dong2020} find the global existence for VPL system in general bounded domain with specular boundary condition. 

Unfortunately, the boundary theory for non-cutoff VPB system remains open since many tools for cutoff Boltzmann theory are not applicable to non-cutoff case. 
Our main target is to consider the global stability of VPB system and VPL system in union of cubes. Compare to \cite{Deng2021}, the boundary value for electric potential $\phi$ creates new difficulties and we introduce Neumann boundary condition for $\phi$ to overcome it in an elegant way. This work gives the first existence result on non-cutoff VPB system in bounded domain with specular boundary condition.

	\subsection{Notations}
	Now we give some notations throughout the paper. 
	Let $\<v\>=\sqrt{1+|v|^2}$ and $\1_{S}$ be the indicator function on a set $S$. 
	 Let
	$\partial^\alpha_\beta = \partial^{\alpha_1}_{x_1}\partial^{\alpha_2}_{x_2}\partial^{\alpha_3}_{x_3}\partial^{\beta_1}_{v_1}\partial^{\beta_2}_{v_2}\partial^{\beta_3}_{v_3}$,
	where $\alpha=(\alpha_1,\alpha_2,\alpha_3)$ and $\beta=(\beta_1,\beta_2,\beta_3)$ are multi-indices. If each component of $\beta'$ is not greater than that of $\beta$'s, we denote by $\beta'\le\beta$. 
	The notation $a\approx b$ (resp. $a\gtrsim b$, $a\lesssim b$) for positive real function $a$, $b$ means there exists $C>0$ not depending on possible free parameters such that $C^{-1}a\le b\le Ca$ (resp. $a\ge C^{-1}b$, $a\le Cb$) on their domain.
	We will write $C>0$ (large) to be a generic constant, which may change from line to line. Denote  spaces $L^2_v$, $L^r_{x}L^2_x$ and $L^s_TL^r_xL^2_v$ for $1\le r,s\le\infty$, respectively, as 
	\begin{align*}
		|f|^2_{L^2_v} = \int_{\R^3}|f|^2\,dv,\quad \|f\|_{L^r_{x}L^2_x} = \Big(\int_{\Omega}|f|^r_{L^2_v}\,dx\Big)^{\frac{1}{r}},
		\quad \|f\|_{L^s_TL^r_{x}L^2_x} = \big\|\|f(t)\|_{L^r_{x}L^2_x}\big\|_{L^s([0,T])}.
	\end{align*}
Also, for velocity weighted space, we write 
\begin{align*}
	|f|^2_{L^2_{k}} = \int_{\R^3}\<v\>^{2k}|f|^2\,dv. 
\end{align*}
We will use some tools from pseudo-differential calculus. One may refer to \cite[Chapter 2]{Lerner2010} for more details. Set $\Gamma=|dv|^2+|d\eta|^2$ and let $M$ be an $\Gamma$-admissible weight function. That is, $M:\R^{2d}\to (0,+\infty)$ satisfies the following conditions:
(a) (slowly varying) there exists $\delta>0$ such that, for any $X,Y\in\R^{2d}$, $|X-Y|\le \delta$ implies
\begin{align*}
	M(X)\approx M(Y);
\end{align*}
(b) (temperance) there exists $C>0$, $N\in\R$, such that for $X,Y\in \R^{2d}$,
\begin{align*}
	\frac{M(X)}{M(Y)}\le C\<X-Y\>^N.
\end{align*}
We say that a symbol $a\in S(M)=S(M,\Gamma)$, if for $\alpha,\beta\in \N^d$, $v,\eta\in\Rd$,
\begin{align*}
	|\partial^\alpha_v\partial^\beta_\eta a(v,\eta)|\le C_{\alpha,\beta}M,
\end{align*}with $C_{\alpha,\beta}$ being a constant depending only on $\alpha$ and $\beta$.
We formally define the Weyl quantization by
\begin{align*}
	a^wu(v)=\int_\Rd\int_\Rd e^{2\pi i (v-u)\cdot\eta}a(\frac{v+u}{2},\eta)u(u)\,dud\eta,
\end{align*}for $f\in\mathscr{S}$. A Weyl quantization $a^w$ is said to be in $Op(M)$ if $a\in S(M)$.

		To study the global well-posedness of problem \eqref{1} in union of cubes, we will consider the following function spaces and energy functionals. Firstly We let $\nu\ge 0$ and consider weight function 
	\begin{equation}\label{w22}
		w_{l,\nu}(\alpha,\beta)=
		\left\{\begin{aligned}
			&\<v\>^{l-r|\alpha|-q|\beta|},\qquad \text{ for hard potential,}\\
			&\<v\>^{l-r|\alpha|-q|\beta|}\exp\big({\nu\<v\>}\big)  , \quad \text{ for soft potential,}
		\end{aligned}\right.
	\end{equation}
where 
\begin{equation}
	\begin{aligned}&r = 1,\quad q=1, \ \text{ for hard potential,}\\
		\label{pq}	&r = -\gamma-\frac{2\gamma(1-s)}{s}+1,\quad q=-\frac{2\gamma}{s}+1, \ \text{ for soft potential,}
	\end{aligned}
\end{equation}
and we let $s=1$ for Landau case. 
	
	\smallskip
	For Landau case, we denote
	\begin{align*}
		\sigma^{ij}(v) &= \phi^{ij}*\mu = \int_{\R^3}\phi^{ij}(v-v')\mu(v')\,dv',\\
		\sigma^i(v)&=\sigma^{ij}\frac{v_j}{2}=\phi^{ij}*\big\{\frac{v_j}{2}\mu\big\}.
	\end{align*}
	Here and after repeated indices are implicitly summed over. Define
	\begin{align*}
		|f|^2_{L^2_{D,w}}=\sum^3_{i,j=1}\int_{\R^3}w^2\big(\sigma^{ij}\partial_{v_i}f\partial_{v_j}f+\sigma^{ij}\frac{v_i}{2}\frac{v_j}{2}|f|^2\big)dv,\quad \|f\|^2_{L^2_xL^2_{D,w}}=\int_\Omega|f|_{L^2_{D,w}}^2\,dx,
	\end{align*}
	and $|{f}|^2_{L^2_{D}}=|{f}|^2_{L^2_{D,1}}$.
	Then by \cite[Corollary 1]{Strain2007} and \cite[Lemma 5]{Guo2002a}, we have 
	\begin{align*}\notag
		|{f}|^2_{L^2_{D,w}} &= |w\<v\>^{\frac{\gamma}{2}}P_v\partial_{v_j}f|_{L^2_v}^2 + |w\<v\>^{\frac{\gamma+2}{2}}\{I-P_v\}\partial_{v_j}f|_{L^2_v}^2 + |w\<v\>^{\frac{\gamma+2}{2}}{f}|_{L^2_v}^2,
	\end{align*}
	where $P_v\xi=\frac{\xi\cdot v}{|v|}\frac{v}{|v|}$.

	For Boltzmann case, as in \cite{Gressman2011}, we denote
	\begin{equation*}
		|f|^2_{L^2_D}:=|\<v\>^{\frac{\gamma+2s}{2}}f|^2_{L^2_v}+ \int_{\R^3}dv\,\<v\>^{\gamma+2s+1}\int_{\R^3}dv'\,\frac{(f'-f)^2}{d(v,v')^{3+2s}}\1_{d(v,v')\le 1},
	\end{equation*}
	and 
	\begin{equation*}
		|f|^2_{L^2_{D,w}}=|wf|^2_{L^2_D},\quad \|f\|^2_{L^2_xL^2_{D,w}} = \int_{\Omega}|f|_{L^2_{D,w}}^2\,dx.
	\end{equation*}
	The fractional differentiation effects are measured using the anisotropic metric on the {\it lifted} paraboloid
	$d(v,v'):=\{|v-v'|^2+\frac{1}{4}(|v|^2-|v'|^2)^2\}^{1/2}$.
	%
	Then by \cite[eq. (2.15)]{Gressman2011}, we have 
	\begin{align*}
		|\<v\>^{\frac{\gamma}{2}}\<D_v\>^sf|^2_{L^2_v}+|\<v\>^{\frac{\gamma+2s}{2}}f|^2_{L^2_v}\lesssim |f|^2_{L^2_D}\lesssim |\<v\>^{\frac{\gamma+2s}{2}}\<D_v\>^sf|^2_{L^2_v}
	\end{align*}

	We will consider function space $
		H^3_{x,v}$ for our analysis. 
Correspondingly, we define the ``instant energy functional" $\E_\nu(t)$ and ``dissipation rate functional" $\D_\nu(t)$ respectively by 	
	\begin{equation}\label{E1}
		\E_\nu(t) \approx
			\sum_{\substack{|\alpha|+|\beta|\le 3}}\Big(\| w_{l,\nu}(\alpha,\beta)\pa^\alpha_\beta f\|^2_{L^2_xL^2_v} + \|\pa^\alpha E\|^2_{L^2_x}\Big),
	\end{equation}
	and
	\begin{equation}\
		\label{D}\D_\nu(t) =
			\sum_{|\alpha|+|\beta|\le 3}\Big(\|w_{l,\nu}(\alpha,\beta)\pa^\alpha_\beta f\|^2_{L^2_xL^2_D} + \|\pa^\alpha E\|^2_{L^2_x}\Big).
	\end{equation}
	We will let $\E(t)=\E_0(t)$ and $\D(t)=\D_0(t)$ to be the energy functional without exponential weight. 
	
	\smallskip
	To obtain the rate of convergence, associated with weight function $w_{l,\nu}(\alpha,\beta)$ given by \eqref{w22}, we let $p\in(0,1]$ be defined by
	\begin{equation}\label{p}
		p = \left\{
		\begin{aligned}
			&1,\qquad\qquad\qquad \text{ for hard potential in both Boltzmann and Landau case, }\\
			&\frac{1}{-\gamma-2s+1},\quad\text{ for soft potential in Boltzmann case,}\\
			&\frac{1}{-\gamma-2+1},\qquad \text{for soft potential in Landau case.}
		\end{aligned}
		\right.
	\end{equation}
	We then are able to show that the obtained solutions decay in time as 
	\begin{align*}
		\E(t)\lesssim e^{-\delta t^p}\E(0), 
	\end{align*}
	with some $\delta>0$.

	\subsection{Main Results}
	
	In this section, we state our main results on global well-posedness of Vlasov-Poisson-Landau systems and Vlasov-Poisson-Boltzmann systems.

	\begin{Thm}\label{Theorem_bounded}
		Let $\Omega$ be defined by \eqref{Omega} and $w_{l,\nu}(\alpha,\beta)$ be given by \eqref{w22}.
		Let $\gamma\ge -3$ for Landau case and $(\gamma,s)\in\{-\frac{3}{2}<\gamma+2s\le 1,
		\ 0<s<1\}$ for Boltzmann case. 
		Let $l\ge 3q$ with $q$ is given in \eqref{pq}.  
		There exists $\varepsilon_0,\nu>0$ such that if $F_0(x,v)=\mu+\mu^{1/2}f_0(x,v)\ge 0$ satisfying \eqref{conservation_bounded} and 
		\begin{align}\label{small}
			\sum_{|\alpha|+|\beta|\le 3}\big(\|{w_{l,\nu}(\alpha,\beta)\partial^\alpha_\beta f_0}\|^2_{L^2_{x,v}}+\|{\partial^\alpha E_0}\|^2_{L^2_{x}}\big)\le \varepsilon_0,
		\end{align}
		then there exists a unique solution $f(t,x,v)$ to the specular reflection boundary problem \eqref{1}, \eqref{specular} and \eqref{Neumann}, satisfying that $F(t,x,v)=\mu+\mu^{1/2}f(t,x,v)\ge 0$ and for any $T>0$,
		\begin{align*}
			\sup_{0\le t\le T}e^{\delta t^p}\E(t) + \sup_{0\le t\le T}\E_\nu(t)\lesssim \ve_0,
		\end{align*}
		where $\E_\nu(t)$ and $\E(t)=\E_0(t)$ are defined by \eqref{E1}.
	\end{Thm}

	As in \cite{Deng2021}, we consider the bounded domain $\Omega$ as union of cubes. In this case, normal derivatives $\pa_n$ on $\pa\Omega$ are also derivatives along axis. By using 
	\begin{equation*}
		v\cdot\nabla_xf = v\cdot n(x)\partial_{n}f + v\cdot \tau_1(x)\partial_{\tau_1}f +v\cdot \tau_2(x)\partial_{\tau_2}f, 
	\end{equation*}  
and equation \eqref{1}, 
one can obtain the compatible high-order specular boundary condition in Lemma \ref{Lem24}, \ref{Lem25} and \ref{Lem26}. On the other hand, $\pa^\alpha$ can be rewritten into normal derivative $\pa_n$ and tangent derivative $\pa_{\tau_1}$ and $\pa_{\tau_2}$ on the boundary. Hence, the boundary term generated from $(\pa^\alpha_\beta(v\cdot\na_x f),\pa^\alpha_\beta f)_{L^2_{x,v}}$ vanishes by using high-order specular-reflection boundary condition. 

In this work, we use space $H^3_{x,v}$ up to third derivatives. In order to obtain the specular boundary condition, we need to assume the Neumann boundary condition for potential $\phi$. With the Poisson equation, we can also derive the third order boundary values for $\phi$; see \eqref{4.16}:
\begin{equation*}
	\pa_{x_ix_ix_i}\phi = 0,\ \ \text{ on }\Gamma_i. 
\end{equation*}
 Correspondingly, we can obtain the boundary values for macroscopic parts $\pa^\alpha[a_\pm,b,c]$ up to third derivatives:
 \begin{multline*}
	 	\partial_{x_i}c(x) =\partial_{x_i}a_\pm(x)=\partial_{x_i}b_j(x)=	b_i(x) \\=
	 	\pa_{x_ix_ix_i}c(x) =\partial_{x_ix_ix_i}a_\pm(x)=\partial_{x_ix_ix_i}b_j(x)=\pa_{x_ix_i}b_i(x)=0.
 \end{multline*} 
	These boundary values enable us to estimate the boundary terms and take integration by parts suitably with respect to $x$. 
	
	For the dissipation rate of $(a_\pm,b,c)$ in Section \ref{Sec_Marco}, we use the solutions to Poisson equation $-\Delta_x\phi_h =h$ with mixed Dirichlet-Neumann boundary condition or pure Neumann boundary condition. One should be careful when dealing with pure Neumann boundary condition. In this case, we need to assume the function $h$ on the right hand side has {\em zero} mean:
	\begin{equation*}
		\int_{\Omega}h\,dx = 0.
	\end{equation*}
	Correspondingly, we need to assume the {\em zero} mean condition for $\phi_h$ to ensure the uniqueness for pure Neumann boundary problem. By using Poincar\'{e}'s inequality, one can obtain the elliptic estimate for Poisson equation with Neumann boundary. The case of mixed Dirichlet-Neumann boundary problem is much easier since there's {\em zero} condition on the boundary and one can apply Sobolev embedding. We will illustrate these calculations in Theorem \ref{Thm41} in details.

	With the nice property of macroscopic parts, we use $e^{\pm\phi}$ as in \cite{Guo2012} to derive the energy estimates. 
	The exponential weight in \eqref{w22} is designed to generate velocity decay after taking derivative. The $2|\beta|$ in \eqref{w22} for soft potential is designed for obtaining velocity decay when estimating $w_{l,\nu}(\alpha,\beta)\pa_{\beta_1}v\cdot\na_x \pa_{\beta-\beta_1}f$. That is, when $|\beta_1|=1$, one can generate velocity decay by using $w_{l,\nu}(\alpha,\beta)\lesssim \<v\>^{\gamma+2s}w_{l,\nu}(\alpha+e_i,\beta-e_i)$. Then we can control it by using dissipation rate. 
	
	Using the boundary values carefully, we will use the integration by parts with respect to spatial variable $x$ again and again in our analysis. Moreover, we will derive that 
	\begin{equation*}
		\sum_{i,j=1}^3\|\partial_{x_ix_j}f\|_{L^2_xL^2_v}^2 = \|\Delta_xf\|_{L^2_xL^2_v}^2.
	\end{equation*}		
	Similarly, 
	\begin{equation*}
		\sum_{i,j=1}^3\|\partial_{x_ix_j}E\|_{L^2_xL^2_v}^2 = \|\Delta_xE\|_{L^2_xL^2_v}^2.
	\end{equation*}	
	Then we only need to estimate $\Delta_xf$ in our proof, which is one of the key points.

	The paper is organized as follows. In Section \ref{Sec_pre}, we give some basic estimates on collision operators. In Section \ref{Sec_Marco}, we give the dissipation macroscopic estimates for VPL and VPB systems. In Section \ref{Sec_Main}, we prove the global existence with large-time behavior via estimates on instant energy \eqref{E1} and dissipation rate \eqref{D}. In Section \ref{Sec_loc_bounded}, we give the proof of local-in-time existence to close the {\it a priori} estimate.

	\section{Preliminary}\label{Sec_pre}
	In this section, we give some basic estimate on collision operator $L$ and $\Gamma(\cdot,\cdot)$. 
	%
	We begin with splitting $L_\pm$. For the Landau case, let $\varepsilon>0$ small and choose a smooth cutoff function $\chi(|v|)\in[0,1]$ such that 
	$
	\chi(|v|)=1\text{ if } |v|<\varepsilon;\  \chi(|v|)=0 \text{ if } |v|>2\varepsilon.
	$ Then we split $L_\pm f = -A_\pm f + K_\pm f$ as in \cite[Section 4.2]{Yang2016}, where  
	\begin{align*}
		-A_\pm f &= 2\partial_{v_i}(\sigma^{ij}\partial_{v_j}f_\pm)-2\sigma^{ij}\frac{v_i}{2}\frac{v_j}{2}f_\pm+2\partial_{v_i}\sigma^i\1_{|v|> R}f_\pm+A_1f\notag\\
		&\qquad+(K_1-\1_{|v|\le R}K_1\1_{|v|\le R})f,\\
		K_\pm f &= 2\partial_{v_i}\sigma^i\1_{|v|\le R}f_\pm + \1_{|v|\le R}K_1\1_{|v|\le R}f,
	\end{align*}
	and $R>0$ is to be chosen large, $\varepsilon>0$ is to be chosen small, and $A_1$ and $K_1$ are given respectively by 
	\begin{align*}
		A_1f &= -\sum_\pm\mu^{-1/2}\partial_{v_i}\Big\{\mu\Big[\Big(\phi^{ij}\chi\Big)*\Big(\mu\partial_{v_j}\big[\mu^{-1/2}f_\pm\big]\Big)\Big]\Big\},\\
		K_1f &=  -\sum_\pm\mu^{-1/2}\partial_{v_i}\Big\{\mu\Big[\Big(\phi^{ij}\big(1-\chi\big)\Big)*\Big(\mu\partial_{v_j}\big[\mu^{-1/2}f_\pm\big]\Big)\Big]\Big\},
	\end{align*}
	with the convolution taken with respect to the velocity variable $v$. 
	Then \cite[eq. (4.33), (4.32)]{Yang2016} shows that 
	\begin{equation*}
		\sum_\pm(A_\pm f,f_\pm)_{L^2_{v}}
		\ge c_0|f|^2_{L^2_{D}},
	\end{equation*}
for some $c_0>0$, 
	and 
	\begin{equation}\label{20aa}
		|(K_1g,h)_{L^2_v}|\lesssim |\mu^{1/10}g|_{L^2_v}|\mu^{1/10}h|_{L^2_v}. 
	\end{equation}
	From \cite[Lemma 3]{Guo2002a}, we know that  
	\begin{align}\label{65aa}
		|\partial_\beta\sigma^{ij}(v)|+|\partial_\beta\sigma^i(v)|\le C_\beta(1+|v|)^{\gamma+2-|\beta|}.
	\end{align}
	Thus, \eqref{20aa} and \eqref{65aa} implies that $K$ is a bounded operator on $L^2_v$ with estimate 
	\begin{align}\label{K1}
		|Kf|_{L^2_v}\lesssim |\mu^{1/10}f|_{L^2_v}. 
	\end{align}
	
	For Boltzmann case, we split $L_\pm f = -A_\pm f + K f$ with 
	\begin{align*}
		-A_\pm f &= 2\mu^{-1/2}Q(\mu,\mu^{1/2}f_\pm),\\
		K f &= \mu^{-1/2}Q(\mu^{1/2}(f_++f_-),\mu).
	\end{align*}
	Then by \cite[Lemma 2.15]{Alexandre2012}, we have 
	\begin{align}\label{K2}
		|Kf|_{L^2_v}\lesssim |\mu^{1/10^3}f|_{L^2_v}. 
	\end{align}

	\begin{Lem}\label{Lem21}Let $w=w_{l,\nu}(\alpha,\beta)$ be given by \eqref{w22}. Let $\gamma>\max\{-3,-2s-\frac{3}{2}\}$ for Boltzmann case and $\gamma\ge-3$ for Landau case. 
		Then 
		\begin{align}\label{36b}
			\sum_{\pm}(L_\pm f,f_\pm)_{L^2_v} \gtrsim |\{\I-\P\}f|_{L^2_D}^2, 
		\end{align}
		\begin{equation}\label{36c}\begin{aligned}
				\sum_\pm(w^2L_\pm g,g_\pm)_{L^2_v}&\ge c_0|g|_{L^2_{D,w}}^2-C|g|_{L^2(B_C)}^2,
			\end{aligned}
		\end{equation}
		and for $|\beta|\ge 1$, 
		\begin{multline}\label{36cw}
				\sum_\pm(w^2_{l,\nu}(\alpha,\beta)\pa^\alpha_{\beta}L_\pm g,\pa^\alpha_\beta g_\pm)_{L^2_v}\ge c_0|\pa^\alpha_\beta g|_{L^2_{D,w_{l,\nu}(\alpha,\beta)}}^2
				-C\sum_{|\beta_1|<|\beta|}|\pa^\alpha_{\beta_1} g|_{L^2_{D,w_{l,\nu}(\alpha,\beta_1)}}^2
				 \\-C|g|_{L^2(B_C)}^2.
		\end{multline}for some generic constant $c_0,C>0$. 
		There exists decomposition $L_\pm = -A_\pm + K_\pm$ such that $K_\pm$ is a bounded linear operator on $L^2_v$ and $A_\pm$ satisfies
		\begin{equation}
			\label{36a}
			\sum_\pm(w^2A_\pm g,g_\pm)\ge c_0|g|_{L^2_{D,w}}^2-C|g|_{L^2(B_C)}^2,
		\end{equation}
		and for $|\beta|\ge 1$, 
		\begin{equation}\label{36aw}\begin{aligned}
				\sum_\pm(w^2_{l,\nu}(\alpha,\beta)\pa^\alpha_{\beta}A_\pm g,\pa^\alpha_{\beta} g_\pm)&\ge c_0|\pa^\alpha_{\beta} g|_{L^2_{D,w_{l,\nu}(\alpha,\beta)}}^2
				-C\sum_{|\beta_1|<|\beta|}|\pa^\alpha_{\beta_1} g|_{L^2_{D,w_{l,\nu}(\alpha,\beta_1)}}^2
				\\&\qquad\qquad\qquad\qquad\qquad\qquad\qquad -C|\pa^\alpha g|_{L^2(B_C)}^2.
			\end{aligned}
		\end{equation}
		Moreover, for any $|\alpha|+|\beta|\le 3$, we have 
		\begin{multline}\label{35a}
			\big(w^2\pa^\alpha_\beta\Gamma_\pm(g_1,g_2),\pa^\alpha_\beta g_3\big)_{L^2_v}\\\lesssim \sum\big(|w\pa^{\alpha_1}_{\beta_1}g_1|_{L^2_v}|\pa^{\alpha_2}_{\beta_2}g_2|_{L^2_{D,w}}+|\pa^{\alpha_1}_{\beta_1}g_1|_{L^2_{D,w}}|w\pa^{\alpha_2}_{\beta_2}g_2|_{L^2_{v}}\big)|\pa^\alpha_\beta g_3|_{L^2_{D,w}},
		\end{multline}
		where the summation is taken over $\alpha_1+\alpha_2=\alpha$, $\beta_1+\beta_2\le \beta$. 
		Consequently, taking integration over $x$, 
		\begin{multline}\label{37a}
			\big(w^2\pa^\alpha_\beta\Gamma_\pm(g_1,g_2),\pa^\alpha_\beta g_3\big)_{L^2_{x,v}}\lesssim \Big(\sum_{|\alpha_1|+|\beta_1|\le 3} \|w\pa^{\alpha_1}_{\beta_1}g_1\|_{L^2_xL^2_v}\sum_{|\alpha_1|+|\beta_1|\le 3}\|\pa^{\alpha_1}_{\beta_1}g_2\|_{L^2_xL^2_{D,w}}\\+\sum_{|\alpha_1|+|\beta_1|\le 3}\|\pa^{\alpha_1}_{\beta_1}g_1\|_{L^2_xL^2_{D,w}}\sum_{|\alpha_1|+|\beta_1|\le 3}\|w\pa^{\alpha_1}_{\beta_1}g_2\|_{L^2_xL^2_v}\Big)\|\pa^\alpha_\beta g_3\|_{L^2_xL^2_{D,w}}.
		\end{multline}
	\end{Lem}
	\begin{proof}
		The proof of \eqref{36b}, \eqref{36c} and \eqref{36cw} can be found in \cite[Lemma 5, Lemma 9]{Strain2007} for Landau case and \cite[eq. (2.13)]{Gressman2011} as well as
		\cite[Lemma 2.7]{Duan2013a} for Boltzmann case. 
		The proof of \eqref{36a} can be found in \cite[Lemma 7 and Lemma 8]{Strain2007} for Landau case and \cite[Lemma 2.6]{Duan2013a} for Boltzmann case. Using \eqref{36c} and boundedness of $K$ from \eqref{K1}, \eqref{K2}, we can obtain \eqref{36aw}. 
		The proof of \eqref{35a} is given by \cite[Lemma 10]{Strain2007} for Landau case and \cite[Lemma 2.4 and Lemma 2.4]{Duan2013a, Fan2017} for Boltzmann case. Finally, we will give the proof of \eqref{37a}. 
		For $|\alpha|+|\beta|\le 3$, from \eqref{35a}, we know that 
		\begin{align*}
			\big(w^2\pa^\alpha_\beta\Gamma_\pm(g_1,g_2),\pa^\alpha_\beta g_3\big)_{L^2_{x,v}}\lesssim \sum_{\alpha_1+\alpha_2=\alpha,\ \beta_1+\beta_2\le \beta}\Big\{ \int_{\Omega}|w\pa^{\alpha_1}_{\beta_1}g_1|^2_{L^2_v}|\pa^{\alpha_2}_{\beta_2}g_2|^2_{L^2_{D,w}}\,dx\\+\int_{\Omega}|\pa^{\alpha_1}_{\beta_1}g_1|^2_{L^2_{D,w}}|w\pa^{\alpha_2}_{\beta_2}g_2|^2_{L^2_v}\,dx\Big\}^{\frac{1}{2}}\|\pa^\alpha_\beta g_3\|_{L^2_xL^2_{D,w}}.
		\end{align*}
		We apply $L^\infty-L^2$ and $L^3-L^6$ H\"{o}lder's inequality to the first term inside the brace: 
		\begin{align*}
			\Big(\int_{\Omega}|w\pa^{\alpha_1}_{\beta_1}g_1|^2_{L^2_v}|\pa^{\alpha_2}_{\beta_2}g_2|^2_{L^2_{D,w}}\,dx\Big)^{\frac{1}{2}}
			&\le \sum_{|\alpha_1|+|\beta_1|=0}\|w\pa^{\alpha_1}_{\beta_1}g_1\|_{L^\infty_xL^2_v}\|\pa^{\alpha_2}_{\beta_2}g_2\|_{L^2_xL^2_{D,w}}\\
			&\qquad+\sum_{|\alpha_1|+|\beta_1|=1}\|w\pa^{\alpha_1}_{\beta_1}g_1\|_{L^6_xL^2_v}\|\pa^{\alpha_2}_{\beta_2}g_2\|_{L^3_xL^2_{D,w}}\\
			&\qquad+\sum_{2\le|\alpha_1|+|\beta_1|\le3}\|w\pa^{\alpha_1}_{\beta_1}g_1\|_{L^2_xL^2_v}\|\pa^{\alpha_2}_{\beta_2}g_2\|_{L^\infty_xL^2_{D,w}}\\
			&\lesssim \sum_{|\alpha_1|+|\beta_1|\le 3}\|w\pa^{\alpha_1}_{\beta_1}g_1\|_{L^2_xL^2_v}
			\sum_{|\alpha_1|+|\beta_1|\le 3}\|\pa^{\alpha_1}_{\beta_1}g_2\|_{L^2_xL^2_{D,w}}. 
		\end{align*}where we used embedding $\|f\|_{L^3_x(\Omega)}\lesssim\|f\|_{H^1_x(\Omega)}$, $\|f\|_{L^6_x(\Omega)}\lesssim\|f\|_{H^1_x(\Omega)}$ and $\|f\|_{L^\infty_x(\Omega)}\lesssim\|f\|_{H^2_x(\Omega)}$ from \cite[Section V and (V.21)]{Adams2003}. 
		Similarly, 
		\begin{align*}
			\Big(\int_{\Omega}|\pa^{\alpha_1}_{\beta_1}g_1|^2_{L^2_{D,w}}|w\pa^{\alpha_2}_{\beta_2}g_2|^2_{L^2_v}\,dx\Big)^{\frac{1}{2}}
			\lesssim \sum_{|\alpha_1|+|\beta_1|\le 3}\|\pa^{\alpha_1}_{\beta_1}g_1\|_{L^2_xL^2_{D,w}}
			\sum_{|\alpha_1|+|\beta_1|\le 3}\|w\pa^{\alpha_1}_{\beta_1}g_2\|_{L^2_xL^2_v}.
		\end{align*}
		Combining the above estimate, we obtain \eqref{37a}. 
		This completes the proof of Lemma \ref{Lem21}.

		\end{proof}

	\section{Macroscopic Estimates}\label{Sec_Marco}

	In this section, we consider the macroscopic estimates for union of cubes.  
	Let $\Omega$ be given by \eqref{Omega} and consider the following problem  
	\begin{align}\label{4.1}
		\partial_t{f_\pm}+v\cdot\na_x{f_\pm} \pm {\nabla_x\phi}\cdot v\mu^{1/2} - L_\pm {f} = {g_\pm},
	\end{align}
	with initial data $(f_0,E_0)$ and boundary condition 
	\begin{align}\label{4.2}
		f(t,x,R_xv) = f(t,x,v),\ \text{ on }x\in\partial\Omega,
	\end{align}
	where $g_\pm$ is chosen to be zero or given by 
	\begin{align}\label{g}
		g_\pm = \pm\nabla_x\phi\cdot\nabla_vf_\pm\mp\frac{1}{2}\nabla_x\phi\cdot vf_\pm+\Gamma_\pm(f,f),
	\end{align}
	and the potential is determined by Poisson equation:
	\begin{align}
		-\Delta_x\phi = a_+-a_-,\label{4.3}
	\end{align}
	with the $zero$ Neumann boundary condition 
	\begin{align}\label{4.4}
		\partial_n\phi = 0, \ \text{ on }\partial\Omega. 
	\end{align}
	We denote $\zeta(v)$ to be a smooth function satisfying $
	\zeta(v)\lesssim e^{-\lambda|v|^2},$ for some $\lambda >0$, which may change from line to line.

		In order to discover the macroscopic dissipation, we take the following velocity moments
	\begin{equation*}
		\mu^{\frac{1}{2}}, v_j\mu^{\frac{1}{2}}, \frac{1}{6}(|v|^2-3)\mu^{\frac{1}{2}},
		(v_j{v_m}-1)\mu^{\frac{1}{2}}, \frac{1}{10}(|v|^2-5)v_j \mu^{\frac{1}{2}}
	\end{equation*}
	with {$1\leq j,m\leq 3$} for the equation \eqref{4.1}. By taking the average and difference on $\pm$ of the resultant equations, one sees that  
	the coefficient functions $[a_\pm,b,c]=[a_\pm,b,c](t,x)$ satisfy
	the fluid-type system; see \cite[Section 3]{Deng2021b}: 
	\begin{equation}\label{19}\left\{
		\begin{aligned}
			&\partial_t\Big(\frac{a_++a_-}{2}\Big)+\nabla_x\cdot b = 0,\\
			&\partial_tb_j+\partial_{x_j}\Big(\frac{a_++a_-}{2}+2c\Big)+\frac{1}{2}\sum_{m=1}^3\partial_{x_m}\Theta_{jm}(\{\I-\P\}f\cdot[1,1])
			= \frac{1}{2}\sum_\pm(g_\pm,v_j\mu^{1/2})_{L^2_v},\\
			&\partial_tc+\frac{1}{3}\nabla_x\cdot b + \frac{5}{6}\sum^3_{j=1}\partial_{x_j}\Lambda_j(\{\I-\P\}f\cdot[1,1]) = \frac{1}{12}\sum_\pm(g_\pm,(|v|^2-3)\mu^{1/2})_{L^2_v},\\
			&\partial_t\Big(\frac{1}{2}\Theta_{jm}((\I-\P)f\cdot[1,1])+2c\delta_{jm}\Big) + \partial_{x_j}b_m+\partial_{x_m}b_j = \frac{1}{2}\sum_\pm\Theta_{jm}(g_\pm+h_\pm),\\
			&\frac{1}{2}\partial_t\Lambda_j((\I-\P)f\cdot[1,1])+\partial_{x_j}c = \frac{1}{2}\Lambda_j(g_++g_-+h_++h_-),
		\end{aligned}\right.
	\end{equation}for $1\le j,m\le3$, where
	\begin{align*}
		h_\pm &= -v\cdot\nabla_x(\II-\PP)f+L_\pm f,\\
		\Theta_{jm}(f_\pm) = (f_\pm,(v_jv_m&-1)\mu^{1/2})_{L^2_v},\ \ \Lambda_j(f_\pm) =\frac{1}{10}(f_\pm,(|v|^2-5)v_j\mu^{1/2})_{L^2_v},
	\end{align*}
	and 
	\begin{equation}\label{98a}\left\{
		\begin{aligned}
			&\partial_t({a_+}-{a_-})+\nabla_{x}\cdot {G}=0,\\
			&\partial_t{G}+\nabla_x({a_+}-{a_-})-2{E}+\nabla_x\cdot\Theta(\{\I-\P\}{f}\cdot[1,-1])=(({g}+L{f})\cdot[1,-1],v\mu^{1/2})_{L^2_v},
		\end{aligned}\right.
	\end{equation}
	where
	\begin{align}\label{93a}
		G = (\{\I-\P\}f\cdot[1,-1],v\mu^{1/2})_{L^2_v}.
	\end{align}

	Here we first write the Lemma for high-order specular reflection boundary conditions. These conditions can be regarded as compatible condition. 
	\begin{Lem}\label{Lem24}
		Let $(f,E)$ be the solution to \eqref{4.1}, \eqref{4.2}, \eqref{4.3} and \eqref{4.4}. Fix $i\in\{1,2,3\}$. Then we have the following identities on boundary $\big\{(x,v) : v\cdot n(x)\neq 0$ and $x$ belongs to the interior of $\Gamma_i\big\}$:
		\begin{equation}\label{33}
			f(x,v)=f(x,R_xv),
		\end{equation}
		and
		\begin{equation}\label{34}\begin{aligned}
				\partial_{\tau_j}f(x,R_xv) &= \partial_{\tau_j}f(x,v),\\
				\partial_{\tau_{j}\tau_k}f(x,R_xv) &= \partial_{\tau_j\tau_k}f(x,v),\\
				\partial_{\tau_{j}\tau_k\tau_m}f(x,R_xv) &= \partial_{\tau_j\tau_k\tau_m}f(x,v),
			\end{aligned}
		\end{equation}for $j,k,m=1,2$, 
		where $(n,\tau_1,\tau_2)$ forms an unit normal basis in $\R^3$.
		For the normal derivatives, we have that on $\big\{(x,v) : v\cdot n(x)\neq 0$ and $x$ belongs to the interior of $\Gamma_i\big\}$,  
		\begin{equation}\label{35}
			\begin{aligned}
				\partial_{n}f(x,R_xv) &= -\partial_{n}f(x,v),\\
				\partial_{\tau_j}\partial_{n}f(x,R_xv) &= -\partial_{\tau_j}\partial_{n}f(x,v),\\
				\partial_{\tau_j\tau_k}\partial_{n}f(x,R_xv) &= -\partial_{\tau_j\tau_k}\partial_{n}f(x,v),
			\end{aligned}
		\end{equation}
		for $j,k=1,2$, 
		and
		\begin{equation}\label{37}\begin{aligned}
				\partial^2_{n}f(x,R_xv) &= \partial^2_{n}f(x,v),\\
				\pa_{\tau_j}\partial^2_{n}f(x,R_xv) &= \pa_{\tau_j}\partial^2_{n}f(x,v),
			\end{aligned}
		\end{equation}for $j=1,2$. 
	\end{Lem}
	\begin{proof}
		Note that $R_xv$ maps $\gamma_-$ onto $\gamma_+$. Then it's direct to obtain \eqref{33} from \eqref{4.2}. On $\Gamma_i$, $\partial_{\tau_j(x)}$ $(j=1,2)$ is the derivative with direction lies in $\Gamma_i$, where $\tau_1(x)$, $\tau_2(x)$ are tangent vector such that $(n,\tau_1,\tau_2)$ forms a unit normal basis in $\R^3$. Then we can obtain \eqref{34} by taking tangent derivatives on \eqref{33}. For normal derivatives, we will apply the equation \eqref{4.1}.
		We claim that 	
		\begin{equation}\label{2.10}
			Lf(x,v)=Lf(x,R_xv)\text{ and }g(x,v)=g(x,R_xv),\ \text{ on } n(x)\cdot v\neq 0,
		\end{equation}for any $x$ belongs to the interior of $\Gamma_i$. 
		Indeed, it suffices to show that $$\nabla_x\phi\cdot\nabla_vf_\pm(R_xv) = \nabla_x\phi\cdot\nabla_vf_\pm(v),\quad \nabla_x\phi\cdot R_xvf_\pm(R_xv)=\nabla_x\phi\cdot vf_\pm(v)$$ and 
		\begin{equation}\label{39}
			\mu^{-1/2}Q(\mu^{1/2}f,\mu^{1/2}g)(R_xv)=\mu^{-1/2}Q(\mu^{1/2}f(R_xv),\mu^{1/2}g(R_xv)),
		\end{equation} on $n(x)\cdot v\neq 0$. By \eqref{4.4}, we have $\partial_{x_i}\phi=0$ on $\Gamma_i$. 
		Notice that $R_xv$ sends $v_i$ to $-v_i$ and preserve the other components on $\Gamma_i$. Then for $j=1,2,3$ such that $j\neq i$, we have $\partial_{v_j}f_\pm(R_xv)=\partial_{v_j}f_\pm(v)$ on $\Gamma_i$. Thus, on $\Gamma_i$, we have 
		\begin{align*}
			\nabla_x\phi\cdot\nabla_vf_\pm(R_xv) = \sum_{j\neq i}\partial_{x_j}\phi\,\partial_{v_j}f_\pm(R_xv)
			= \nabla_x\phi\cdot\nabla_vf_\pm(v),
		\end{align*} 
		and 
		\begin{align*}
			\nabla_x\phi\cdot R_xvf_\pm(R_xv) = \sum_{j\neq i}\partial_{x_j}\phi\, v_jf_\pm(v) = \nabla_x\phi\cdot vf_\pm(v).
		\end{align*}
	Next we prove \eqref{39}. 
		For the Boltzmann case, we apply the Carleman representation as in \cite[Appendix]{Global2019} to find that
		\begin{align*}
			&\quad\,\mu^{-1/2}Q(\mu^{1/2}f,\mu^{1/2}g)(R_xv)\\
			&= \int_{\R^3_h}\int_{E_{0,h}}\tilde{b}(\alpha,h)\1_{|\alpha|\ge|h|}\frac{|\alpha+h|^{\gamma+1+2s}}{|h|^{3+2s}}\mu^{1/2}(R_xv+\alpha-h)\\&\qquad\qquad\qquad\times\big(f(R_xv+\alpha)g(R_xv-h)-f(R_xv+\alpha-h)g(R_xv)\big)\,d\alpha dh \\
			&= \mu^{-1/2}Q(\mu^{1/2}f(R_xv),\mu^{1/2}g(R_xv)),
		\end{align*} where we use change of variable $(\alpha,h)\mapsto (R_x\alpha,R_xh)$. 
		
		For Landau case, we will use representation from \cite[Lemma 1]{Guo2002a}:
		\begin{align}
			\label{Landau}\notag
			\mu^{-1/2}Q(\mu^{1/2}f,\mu^{1/2}g) = \partial_{v_j}\Big[\Big\{\phi^{jk}*[\mu^{1/2}f]\Big\}\partial_{v_k}g\Big]
			-\Big\{\phi^{jk}*\Big[\frac{v_j}{2}\mu^{1/2}f\Big]\Big\}\partial_{v_k}g\\
			-\partial_{v_j}\Big[\Big\{\phi^{jk}*[\mu^{1/2}\partial_{v_k}f]\Big\}g\Big]
			-\Big\{\phi^{jk}*\Big[\frac{v_j}{2}\mu^{1/2}\partial_kf\Big]\Big\}g.
		\end{align}
		Notice that $\partial_{v_i}f(R_xv)=-\partial_{v_i}(f(R_xv))$  and $\partial_{v_j}f(R_xv)=\partial_{v_j}(f(R_xv))$ on $\Gamma_i$, for $j\neq i$. Then on $\Gamma_i$, 
		\begin{multline*}
			\sum_{j,k=1}^3\partial_{v_j}\Big[\Big\{\phi^{jk}*[\mu^{1/2}f]\Big\}\partial_{v_k}g\Big](R_xv)\\
			= \sum_{k=1}^3\partial_{v_i}\Big[-\Big\{\phi^{ik}*[\mu^{1/2}f]\Big\}(R_xv)\partial_{v_k}g(R_xv)\Big]+ \sum_{j\neq i}\sum_{k=1}^3\partial_{v_j}\Big[\Big\{\phi^{jk}*[\mu^{1/2}f]\Big\}(R_xv)\partial_{v_k}g(R_xv)\Big]\\
		= \partial_{v_i}\Big[\Big\{\phi^{ii}*[\mu^{1/2}f(R_xv)]\Big\}\partial_{v_i}(g(R_xv))\Big]+
			\sum_{k\neq i}\partial_{v_i}\Big[\Big\{\phi^{ik}*[\mu^{1/2}f(R_xv)]\Big\}\partial_{v_k}(g(R_xv))\Big]\qquad\ \\
			\ \ +
			\sum_{j\neq i}\partial_{v_j}\Big[\Big\{\phi^{ji}*[\mu^{1/2}f(R_xv)]\Big\}\partial_{v_i}(g(R_xv))\Big]+
			\sum_{j\neq i, k\neq i}\partial_{v_j}\Big[\Big\{\phi^{jk}*[\mu^{1/2}f(R_xv)]\Big\}\partial_{v_k}g(R_xv)\Big]\\
			= \sum_{j,k=1}^3\partial_{v_j}\Big[\Big\{\phi^{jk}*[\mu^{1/2}f(R_xv)]\Big\}\partial_{v_k}(g(R_xv))\Big],\qquad\qquad\qquad\qquad\qquad\qquad\qquad\qquad\qquad
		\end{multline*}where we apply \eqref{phiij} to deduce that $\phi^{ik}(R_xv) = -\phi^{ik}(v)$, $\phi^{ji}(R_xv)=-\phi^{ji}(v)$ when $k\neq i$, $j\neq i$. Similar calculation can be applied to the second to forth terms of \eqref{Landau} and we obtain \eqref{39} for Landau case. This completes the claim \eqref{2.10}.

		Noticing that 
		\begin{equation*}
			v\cdot\nabla_xf = v\cdot n(x)\partial_{n}f + v\cdot \tau_1(x)\partial_{\tau_1}f +v\cdot \tau_2(x)\partial_{\tau_2}f, 
		\end{equation*}  
		we can rewrite \eqref{4.1} as 
		\begin{align*}
			v\cdot n(x)\partial_n{f_\pm} =-v\cdot \tau_1(x)\partial_{\tau_1}f_\pm -v\cdot \tau_2(x)\partial_{\tau_2}f_\pm -\partial_t{f_\pm}\mp {\nabla_x\phi}\cdot v\mu^{1/2} + L_\pm {f} + {g_\pm}.
		\end{align*}
		Applying \eqref{33} and \eqref{2.10} to the right hand side, we can obtain that on $\partial\Omega$, 
		\begin{align*}
			R_xv\cdot n(x)\partial_n{f_\pm}(R_xv) = v\cdot n(x)\partial_nf_\pm(v).
		\end{align*}
		Since $R_xv\cdot n(x)= -v\cdot n(x)$, this implies \eqref{35} by taking tangent derivative. 	
		Apply $\partial_n$ to \eqref{4.1} twice and rewrite it to be 
		\begin{multline}\label{4.6}
			v\cdot n\partial_n\partial_nf_\pm = -v\cdot \tau_1(x)\partial_{\tau_1}\partial_nf_\pm -v\cdot \tau_2(x)\partial_{\tau_2}\partial_nf_\pm -\partial_t\partial_nf_\pm \\+L_\pm\partial_nf \mp {\partial_n\nabla_x\phi}\cdot v\mu^{1/2}+\partial_ng_\pm.
		\end{multline} 
		Here, on $\Gamma_i$, by taking tangent derivatives on \eqref{4.4}, we have $\pa_n\pa_{x_j}\phi=0$ for $j\neq i$ and hence 
		\begin{align*}
			\partial_n\nabla_x\phi\cdot R_xv\mu^{1/2}(R_xv) = \partial_n\partial_{x_i}\phi (R_xv)_i\mu^{1/2}(v) = -\partial_n\partial_{x_i}\phi v_i\mu^{1/2}.
		\end{align*}
		When $g_\pm$ is given by \eqref{g}, we have on $\Gamma_i$ that 
		\begin{align*}
			\partial_ng_\pm &= \pm\partial_n\nabla_x\phi\cdot\nabla_vf_\pm\pm\nabla_x\phi\cdot\partial_n\nabla_vf_\pm\mp\frac{1}{2}\partial_n\nabla_x\phi\cdot vf_\pm
			\\&\quad\mp\frac{1}{2}\nabla_x\phi\cdot v\partial_nf_\pm+\Gamma_\pm(\partial_nf,f)+\Gamma_\pm(f,\partial_nf)\\
			&=\pm\partial_n\pa_{x_i}\phi\,\pa_{v_i}f_\pm \pm\sum_{j\neq i}\pa_{x_j}\phi\,\partial_n\pa_{v_j}f_\pm \mp\frac{1}{2}\partial_n\pa_{v_i}\phi\, v_if_\pm
			\\&\quad\mp\frac{1}{2}\sum_{j\neq i}\pa_{x_j}\phi\, v_j\partial_nf_\pm +\Gamma_\pm(\partial_nf,f)+\Gamma_\pm(f,\partial_nf). 
		\end{align*}
		Together with \eqref{35} and \eqref{39}, we know that on $\Gamma_i$, 
		\begin{align*}
			\partial_ng_\pm(R_xv) = - \partial_ng_\pm(v). 
		\end{align*}
		Combining the above identities and \eqref{4.6}, we have 
		\begin{align*}
			R_xv\cdot n(x)\partial^2_{n}f(x,R_xv) = -v\cdot n(x)\partial^2_{n}f(x,v).
		\end{align*}This gives \eqref{37}$_1$ and \eqref{37}$_2$ follows by taking tangent derivatives. This completes the proof of Lemma \ref{Lem24}. 
		\end{proof}

	As a corollary, by definition \eqref{abc}, we have the following boundary values for $[a_\pm,b,c]$.
	\begin{Lem}\label{Lem25}
		Let $(f,E)$ be the solution to \eqref{4.1}, \eqref{4.2}, \eqref{4.3} and \eqref{4.4}. Define $[a_\pm,b,c]$ by \eqref{abc}. 
		For $i=1,2,3$ and any $x\in\Gamma_i$, we have 
		\begin{align}\label{Lem25a}
			\partial_{x_i}c(x) =\partial_{x_i}a_\pm(x)=\partial_{x_i}b_j(x)=\pa_{x_ix_i}b_i(x)=	b_i(x) = 0,
		\end{align}for $j\neq i$. 
		As a consequence, 
		\begin{equation}\label{2.14}
			\begin{aligned}
				\sum_{i,j=1}^3\|\partial_{x_ix_j}a_\pm\|^2_{L^2_x} = \|\Delta_xa_\pm\|^2_{L^2_x},\\ 
				\sum_{i,j=1}^3\|\partial_{x_ix_j}b\|^2_{L^2_x} = \|\Delta_xb\|^2_{L^2_x}, \\
				\sum_{i,j=1}^3\|\partial_{x_ix_j}c\|^2_{L^2_x} = \|\Delta_xc\|^2_{L^2_x}. 
			\end{aligned}
		\end{equation}
		Moreover, on $\Gamma_i$, we have 
		\begin{equation}\label{4.16}
			\pa_{x_ix_ix_i}\phi = 0,
		\end{equation}
	\end{Lem}
	\begin{proof}
		Fix $x\in\Gamma_i$. Notice that on the boundary of union of cubes, we have $\partial_nf = \partial_{x_i}f$ or $-\partial_{x_i}f$. Then by \eqref{35} and change of variable $v\mapsto R_xv$, we have on $\Gamma_i$ that 
		\begin{align*}
			\partial_{x_i}c &= \frac{1}{12}\int_{\R^3}\partial_{x_i}\big(f_+(x,R_xv)+f_-(x,R_xv)\big)|R_xv|^2\mu^{1/2}(R_xv)\,dv\\
			&= -\frac{1}{12}\int_{\R^3}\partial_{x_i}\big(f_+(x,v)+f_-(x,v)\big)|v|^2\mu^{1/2}(v)\,dv = 0. 
		\end{align*}
		Similarly, on interior of $\Gamma_i$, we have 
		\begin{equation*}
			\partial_{x_i}a_\pm = \int_{\R^3}\partial_{x_i}f_\pm(x,R_xv)\mu^{1/2}(R_xv)\,dv
			= -\int_{\R^3}\partial_{x_i}f_\pm(x,v)\mu^{1/2}(v)\,dv = 0,  
		\end{equation*}
		and for $j\neq i$, 
		\begin{align*}
			\partial_{x_i}b_j &= \frac{1}{2}\int_{\R^3}\partial_{x_i}\big(f_+(x,R_xv)+f_-(x,R_xv)\big)(R_xv)_j\mu^{1/2}(R_xv)\,dv\\
			&= -\frac{1}{2}\int_{\R^3}\partial_{x_i}\big(f_+(x,v)+f_-(x,v)\big)v_j\mu^{1/2}(v)\,dv = 0.  
		\end{align*}
		On $\Gamma_i$, we have $(R_xv)_i=-v_i$ and hence by \eqref{33} and \eqref{37}, we have 
		\begin{align*}
			b_i(x) &= \frac{1}{2}\int_{\R^3}\big(f_+(x,R_xv)+f_-(x,R_xv)\big)(R_xv)_i\mu^{1/2}(R_xv)\,dv\\
			&= -\frac{1}{2}\int_{\R^3}\big(f_+(x,v)+f_-(x,v)\big)v_i\mu^{1/2}(v)\,dv = 0, 
		\end{align*}
	and 
	\begin{align*}
		\pa_{x_ix_i}b_i &= \frac{1}{2}\int_{\R^3}\big(\pa_{x_ix_i}f_+(x,R_xv)+\pa_{x_ix_i}f_-(x,R_xv)\big)(R_xv)_i\mu^{1/2}(R_xv)\,dv\\
		&= -\frac{1}{2}\int_{\R^3}\big(\pa_{x_ix_i}f_+(x,v)+\pa_{x_ix_i}f_-(x,v)\big)v_i\mu^{1/2}(v)\,dv = 0. 
	\end{align*}
		For any $\varphi=\varphi(x)$ satisfying that $\partial_{x_k}\varphi=0$ or $\partial_{x_kx_k}\varphi=0$ on $\Gamma_k$ for any $k=1,2,3$. We have 
		\begin{align*}
			\int_{\Omega}|\partial_{x_ix_j}\varphi|^2\,dx
			&= \int_{\Gamma_i}\partial_{x_ix_j}\varphi\,\partial_{x_j}\varphi\,dS(x) - \int_{\Gamma_j}\partial_{x_ix_i}\varphi\,\partial_{x_j}\varphi\,dx + \int_{\Omega}\partial_{x_ix_i}\varphi\,\partial_{x_jx_j}\varphi\,dx
			\\&= \int_{\Omega}\partial_{x_ix_i}\varphi\partial_{x_jx_j}\varphi\,dx, 
		\end{align*}
		where $dS$ is the spherical measure. Then we have $\sum_{i,j}\|\partial_{x_ix_j}\varphi\|^2_{L^2_x} = \|\Delta_x\varphi\|^2_{L^2_x}$. Replacing $\varphi$ to be $a_\pm$, $b_j$ and $c$, we obtain \eqref{2.14}.
		
		For the proof of \eqref{4.16}, by taking tangent derivatives on \eqref{4.4}, we have $\pa_{x_ix_jx_j}\phi=0$ on $\Gamma_i$ for $j\neq i$. Then by \eqref{4.3}, we have on $\Gamma_i$ that 
		\begin{align*}
			\pa_{x_ix_ix_i}\phi = -\sum_{j\neq i}\pa_{x_ix_jx_j}\phi - \pa_{x_i}(a_+-a_-) = 0,
		\end{align*}
		where we used \eqref{Lem25a} for $\pa_{x_i}a_\pm=0$. 
		This completes Lemma \ref{Lem25}. 
		
		\end{proof}

	With the help of \eqref{4.16}, we are able to obtain the third derivative version for Lemma \ref{Lem24} and \ref{Lem25}. 
	\begin{Lem}\label{Lem26}
		Assuming the same conditions in Lemma \ref{Lem24}. 
		Then we have for $x\in\pa\Omega$ that 
		\begin{align}\label{38}
			\partial_n^3f_\pm(R_xv) = \partial_n^3f_\pm(v),\text{ on }v\cdot n(x)\neq 0. 
		\end{align}
		Consequently, we have on $\Gamma_i$$(i=1,2,3)$ that 
		\begin{align}\label{38a}
			\pa_{x_ix_ix_i}c(x) =\partial_{x_ix_ix_i}a_\pm(x)=\partial_{x_ix_ix_i}b_j(x) = 0,
		\end{align}for $j\neq i$. 
	\end{Lem}
	\begin{proof}
		By taking normal derivative of \eqref{4.6}, we have 
		\begin{multline}\label{4.17}
			v\cdot n\,\partial_n\partial_n\partial_nf_\pm = -v\cdot \tau_1(x)\partial_{\tau_1}\partial_n\partial_nf_\pm -v\cdot \tau_2(x)\partial_{\tau_2}\partial_n\partial_nf_\pm -\partial_t\partial_n\partial_nf_\pm \\+L_\pm\partial_n\partial_nf \mp {\partial_n\partial_n\nabla_x\phi}\cdot v\mu^{1/2}+\partial_n\partial_ng_\pm.
		\end{multline}
		Notice that on $\Gamma_i$, by \eqref{4.16}, we have 
		\begin{align*}
			\partial_n\partial_n\nabla_x\phi\cdot v\mu^{1/2} = \sum_{j\neq i}\pa_{x_ix_ix_j}\phi\ v_j\mu^{1/2},
		\end{align*}
		and when $g$ is given by \eqref{g}, one has 
		\begin{align*}
			\partial_n\partial_ng_\pm
			&=\pm\partial_n\partial_n\nabla_x\phi\cdot\nabla_vf_\pm
			\pm2\partial_n\nabla_x\phi\cdot\nabla_v\partial_nf_\pm
			\pm\nabla_x\phi\cdot\partial_n\partial_n\nabla_vf_\pm\\
			&\quad\mp\frac{1}{2}\partial_n\partial_n\nabla_x\phi\cdot vf_\pm
			\mp\partial_n\nabla_x\phi\cdot v\partial_nf_\pm
			\mp\frac{1}{2}\nabla_x\phi\cdot v\partial_n\partial_nf_\pm\\ &\quad+\Gamma_\pm(\partial_n\partial_nf,f)+2\Gamma_\pm(\partial_nf,\partial_nf)+\Gamma_\pm(f,\partial_n\partial_nf)\\
			&=\pm\sum_{j\neq i}\pa_{x_ix_ix_j}\phi\,\pa_{v_j}f_\pm
			\pm2\pa_n\pa_{x_i}\phi\,\pa_{v_i}\partial_nf_\pm
			\pm\sum_{j\neq i}\pa_{v_j}\phi\,\partial_{x_ix_i}\pa_{v_j}f_\pm\\
			&\quad\mp\frac{1}{2}\sum_{j\neq i}\partial_{x_ix_ix_j}\phi\, v_jf_\pm
			\mp\partial_n\pa_{x_i}\phi\, v_i\partial_nf_\pm
			\mp\frac{1}{2}\sum_{j\neq i}\pa_{v_j}\phi\, v_j\partial_n\partial_nf_\pm\\ &\quad+\Gamma_\pm(\partial_n\partial_nf,f)+2\Gamma_\pm(\partial_nf,\partial_nf)+\Gamma_\pm(f,\partial_n\partial_nf).  
		\end{align*}
		Applying Lemma \ref{Lem24} and identity \eqref{39}, we have on $\Gamma_i$ that 
		\begin{equation*}
			\sum_{j\neq i}\pa_{x_ix_ix_j}\phi \, (R_xv)_j\mu^{1/2}(R_xv) = \sum_{j\neq i}\pa_{x_ix_ix_j}\phi\, v_j\mu^{1/2}(v),
		\end{equation*}
		\begin{equation*}
			L_\pm\partial_n\partial_nf(R_xv) = L_\pm\partial_n\partial_nf(v),
		\end{equation*}
		and
		\begin{equation*}
			\partial_n\partial_ng_\pm(R_xv) = \partial_n\partial_ng_\pm(v).
		\end{equation*}
		Combining the above identities and Lemma \ref{Lem24}, we have from \eqref{4.17} that on $\Gamma_i$ 
		\begin{align*}
			R_xv\cdot n\,\partial_n\partial_n\partial_nf_\pm(R_xv) = v\cdot n\,\partial_n\partial_n\partial_nf_\pm(v).
		\end{align*}
		Note that $R_xv\cdot n(x)=-v\cdot n(x)$ on $\Gamma_i$, we obtain \eqref{38}. 
		The proof of \eqref{38a} is similar to \eqref{Lem25a} by using specular reflection boundary condition \eqref{37} and \eqref{38} for high-order derivative and we omit the details for brevity. Then we complete the proof of Lemma \ref{Lem26}.
	\end{proof}

	Next we give the estimates on macroscopic parts $[a_\pm,b,c]$. The idea is similar to \cite{Deng2021}. However, with the electric potential $\phi$, we need more careful calculations. 
	
	\begin{Thm}\label{Thm41}
		Let $K=2,3$ be the total order of derivative. 
		Let $\gamma\ge -3$ for Landau case, $\gamma>\max\{-3,-2s-3/2\}$ for Boltzmann case and $T>0$. 
		Let $(f,E)$ be the solution of \eqref{4.1}, \eqref{4.2}, \eqref{4.3} and \eqref{4.4} in bounded domain $\Omega$ with initial data $(f_0,E_0)$. Then there exists an instant energy functional $\E_{int}(t)$ satisfying
		\begin{align*}
			\E_{int}(t)\lesssim \sum_{|\alpha|\le K} \|\partial^\alpha f\|^2_{L^2_{x}},
		\end{align*} such that 
		\begin{multline*}
			\partial_t\E_{int}(t) + \lambda \sum_{|\alpha|\le K} \|\partial^\alpha[{a_+},{a_-},{b},{c}]\|^2_{L^2_{x}}+\lambda \sum_{|\alpha|\le K}\|{\partial^\alpha E}\|^2_{L^2_{x}}\\
			\lesssim \sum_{|\alpha|\le K}\|\{\I-\P\}{\partial^\alpha f}\|^2_{L^2_{x}L^2_{\gamma/2}} + \sum_{|\alpha|\le K}\|(\partial^\alpha g,\zeta(v))_{L^2_v}\|_{L^2_x}^2+ \|E\|_{L^2_x}^4, 
		\end{multline*}for some constant $\lambda>0$, where $g=[g_+,g_-]$ is zero or given by \eqref{g}.  
	\end{Thm}
	\begin{proof}
		Let $|\alpha|\le K$ and we restrict
		\begin{equation}\label{323}
			\partial^\alpha=\partial_{x_ix_i} \text{ for some }i=1,2,3 \text{ when } |\alpha|=2.
		\end{equation} 
	Using Lemma \ref{Lem25}, we only need to consider $\|\Delta_x[{a_+},{a_-},{b},{c}]\|^2_{L^2_{x}}$ when estimating the second order derivatives of $[{a_+},{a_-},{b},{c}]$. Applying $\partial^\alpha$ to \eqref{4.1}, we have 
		\begin{equation}\label{40d}
			\partial_t{\partial^\alpha f_\pm}+v\cdot \na_x{\partial^\alpha f_\pm} \pm \partial^\alpha {\nabla_x\phi}\cdot v\mu^{1/2} - L_\pm {\partial^\alpha f} = {\partial^\alpha g_\pm}.
		\end{equation}
		To state the proof in a unified way, we let ${\Phi}(t,x,v)\in C^1((0,+\infty)\times\Omega\times\R^3)$ be a test function. Taking the inner product of \eqref{40d} with ${\Phi}(t,x,v)$ with respect to $(x,v)$, we obtain 
		\begin{multline*}
			\partial_t({\partial^\alpha f_\pm},{\Phi})_{L^2_{x,v}}- ({\partial^\alpha f_\pm},\partial_t{\Phi})_{L^2_{x,v}}-({\partial^\alpha f_\pm},v\cdot{\nabla_{x}\Phi})_{L^2_{x,v}} 
			+\int_{\pa\Omega}(v\cdot n(x){\partial^\alpha f_\pm},{\Phi})_{L^2_v}\,dS(x)\\
			\pm (\partial^\alpha {\nabla_x\phi}\cdot v\mu^{1/2},{\Phi})_{L^2_{x,v}} - (L_\pm {\partial^\alpha f},{\Phi})_{L^2_{x,v}} = ({\partial^\alpha g_\pm},{\Phi})_{L^2_{x,v}}.
		\end{multline*}
		Using the decomposition ${f_\pm}=\P{f_\pm}+\{\I-\P\}{f_\pm}$, we rewrite the above equation to be  
		\begin{align}\label{100}
			\partial_t({\partial^\alpha f_\pm},{\Phi})_{L^2_{x,v}}-({\partial^\alpha \P_\pm f},v\cdot{\nabla_{x}\Phi})_{L^2_{x,v}}  = \sum_{j=1}^5S_j,
		\end{align}
		where $S_j$'s are defined by 
		\begin{align*}
			S_1 &= ({\partial^\alpha f_\pm},\partial_t{\Phi})_{L^2_{x,v}},\\
			S_2 &= ({\partial^\alpha (\II-\PP)f},v\cdot{\nabla_{x}\Phi})_{L^2_{x,v}} ,\\
			S_3&= (L_\pm {\partial^\alpha f},{\Phi})_{L^2_{x,v}}+({\partial^\alpha g_\pm},{\Phi})_{L^2_{x,v}},\\
			S_4 &= \mp(\partial^\alpha {\nabla_x\phi}\cdot v\mu^{1/2},{\Phi})_{L^2_{x,v}},\\
			S_5 &= -\int_{\pa\Omega}(v\cdot n(x){\partial^\alpha f_\pm},{\Phi})_{L^2_v}\,dS(x).
		\end{align*}

		\medskip \noindent{\bf Step 1. Estimate on ${c}(t,x)$:} In this step, we will let $|\alpha|\ge 1$. Choose test function 
		\begin{align*}
			{\Phi} = {\Phi_c} = (|v|^2-5)\big(v\cdot{\nabla_{x}\phi_c}(t,x)\big)\mu^{1/2},
		\end{align*}
		where $\phi_c$ solves 		
		\begin{equation}\label{120}\left\{\begin{aligned}
				&-\Delta_x \phi_c = {\partial^\alpha c}\ \text{ in }\Omega,\\
				&{\phi_c}(x)= 0 \ \text{ on }\ x\in \Gamma_i,\ \text{ if }\alpha_i = 1\text{ or }3,\\
				&\partial_n\phi_c(x)= 0\ \text{ on }\ x\in \Gamma_i,\ \text{ if }\alpha_i = 0\text{ or } 2.
			\end{aligned}\right.
		\end{equation}
		The existence of solution to \eqref{120} is given by \cite[Lamma 4.4.3.1]{Grisvard1985}. In particular, when $|\alpha|=2$, we deduce from \eqref{323} that $\alpha_i=2$ for some $i$ and $\alpha_k=0$ for $k\neq i$. Thus, \eqref{120} is pure Neumann problem and we need $\int_{\Omega}\partial_{x_ix_i}c\,dx=\int_{\Gamma_i}\partial_{x_i}c\,dS(x)=0$ from Lemma \ref{Lem25} to ensure the existence of \eqref{120}. 
		Similar to the proof for \eqref{2.14}, by using boundary value of $\phi_c$, we have 
		\begin{align}\label{3.8}
			\sum_{i,j=1}^3\|\partial_{x_ix_j}{\phi_c}\|_{L^2_{x}}^2 = \|\Delta_x\phi_c\|_{L^2_x}^2 \lesssim \|\partial^\alpha c\|^2_{L^2_x}. 
		\end{align}
		We discuss the value of $\alpha$ in the following cases. 
%
%
		If $|\alpha| = 1$, then $\alpha_i=1$ for some $1\le i\le 3$. Hence, $\phi_c(x)=0$ on $\Gamma_i$ and $\partial_n\phi_c(x)=0$ on $\Gamma_j$ for $j\neq i$. It follows that 
		\begin{align}\label{3.9}\notag
			\|\na_x\phi_c\|^2_{L^2_x} &= \sum_{j=1}^3\int_{\Gamma_j}\partial_{x_j}\phi_c\,\phi_c\,dx - \int_{\Omega}\Delta_x\phi_c\,\phi_c\,dx\\
			&= \int_{\Omega}\partial^\alpha c\,\phi_c\,dx \le \|\partial^\alpha c\|_{L^2_x}\|\phi_c\|_{L^2_x}.
		\end{align}
		Since $\phi_c =0$ on $\Gamma_i$, by Sobolev embedding \cite[Theorem 6.7-5]{Ciarlet2013}, we have $\|\phi_c\|_{L^2_x}\lesssim \|\na_x\phi_c\|_{L^2_x}$. Then from \eqref{3.9}, we have 
		\begin{align}\label{3.10}
			\|\na_x\phi_c\|_{L^2_x}\lesssim \|\partial^\alpha c\|_{L^2_x}\lesssim \sum_{|\alpha|=1}\|\partial^\alpha c\|_{L^2_x}.
		\end{align}
		Similarly, since derivative $\pa_t$ doesn't affect the boundary value for $\phi_c$, we have 
		\begin{equation}\label{3.11}
			\|\partial_t\na_x\phi_c\|_{L^2_x}\lesssim \sum_{|\alpha|=1}\|\partial_t\partial^\alpha c\|_{L^2_x}.
		\end{equation}
		
		If $|\alpha|=2$, at stated before, we only consider $\alpha_i=2$ for some $1\le i\le 3$. Then for this $i$, similar to \eqref{3.9}, by using boundary values $\partial_{x_i}c=0$ on $\Gamma_i$ from \eqref{Lem25}, we have 
		\begin{align*}
			\|\na_x\phi_c\|^2_{L^2_x}
			&= \int_{\Omega}\partial_{x_ix_i} c\,\phi_c\,dx = \int_{\Omega}\partial_{x_i}c\,\partial_{x_i}\phi_c\,dx
			\le \|\partial_{x_i}c\|_{L^2_x}\|\partial_{x_i}\phi_c\|_{L^2_x}.
		\end{align*}
		This implies that 
		\begin{align}\label{3.12}
			\|\na_x\phi_c\|_{L^2_x}\le \|\partial_{x_i}c\|_{L^2_x}.
		\end{align}
		Similarly, since $\pa_t$ doesn't affect the boundary value for $\phi_c$, we have 
		\begin{align}\label{3.13}
			\|\partial_t\na_x\phi_c\|_{L^2_x}\le \|\partial_t\partial_{x_i}c\|_{L^2_x}.
		\end{align}
		
		If $|\alpha|=3$, then there exists $1\le i\le 3$ such that $\alpha_i=1$ or $3$.
		Then the boundary value for $\phi_c$ gives that $\phi_c=0$ on $\Gamma_i$. 
		Denote $\pa^\alpha=\pa_{x_ix_jx_k}$ for some $1\le j,k\le 3$. 
		Then taking inner product of \eqref{120} with $\phi_c$, we have 
		\begin{align*}
			\|\na_x\phi_c\|^2_{L^2_x} &= \int_{\Omega}\pa^\alpha c\,\phi_c\,dx
			= \int_{\Gamma_i}\pa_{x_jx_k}c\, \phi_c\,dS(x)
			- \int_{\Omega}\pa_{x_jx_k}c\,\pa_{x_i}\phi_c\,dx\\
			&\lesssim \|\na_x^2 c\|_{L^2_x}\|\na_x\phi\|_{L^2_x}.
		\end{align*}
		Thus, 
		\begin{align}\label{4.28}
			\|\na_x\phi_c\|_{L^2_x}\lesssim \|\na_x^2 c\|_{L^2_x}. 
		\end{align}
		Similarly, 
		\begin{align}\label{4.29}
			\|\pa_t\na_x\phi_c\|_{L^2_x}\lesssim \|\pa_t\na_x^2 c\|_{L^2_x}. 
		\end{align}

		Now we can compute \eqref{100}. For the second term on left hand side of \eqref{100}, we have 
		\begin{align*}
			&\quad\,-({\partial^\alpha \PP f},v\cdot{\nabla_{x}\Phi_{c}})_{L^2_{x,v}} \\
			&= -\sum_{j,m=1}^3\big(({\partial^\alpha a_\pm}+{\partial^\alpha b}\cdot v+(|v|^2-3){\partial^\alpha c} )\mu^{1/2},v_jv_m(|v|^2-5)\mu^{1/2}\partial_{x_j}\partial_{x_m}\phi_{c})_{L^2_{x,v}} \\
			&= 10\sum_{j=1}^3({\partial^\alpha c} ,{-\partial^2_j\phi_{c}}\big)_{L^2_{x,v}}  = 10\|{\partial^\alpha c}\|^2_{L^2_{x}} .
		\end{align*}
		Note that $\int_{\R^3}|v|^4v_j^2\mu\,dv = 35$, $\int_{\R^3}|v|^2v_j^2\mu\,dv = 5$ and $\int_{\R^3}v_j^2\mu\,dv = 1$. 
		For $S_1$, we obtain from \eqref{3.11}, \eqref{3.13} and \eqref{4.29} that for any $\eta>0$, 
		\begin{align*}
			|S_1|&\le|({\partial^\alpha f},\partial_t{\Phi_c})_{L^2_{x,v}}| = |(\{\I-\P\}{\partial^\alpha f},\partial_t{\Phi_c})_{L^2_{x,v}}|\\
			&\lesssim \eta\|\partial_t{\nabla_x\phi_c}\|^2_{L^2_{x}}+C_\eta\|\{\I-\P\}{\partial^\alpha f}\|^2_{L^2_{x}L^2_{\gamma/2}}\\
			&\lesssim \eta
			\sum_{1\le|\alpha|\le K}\|\partial^\alpha b\|^2_{L^2_{x}}+\eta\sum_{1\le|\alpha|\le K} \|(\partial^\alpha g,\zeta)_{L^2_v}\|_{L^2_x}^2+C_\eta\sum_{|\alpha|\le K}\|\{\I-\P\}{\partial^\alpha f}\|^2_{L^2_{x}L^2_{\gamma/2}}.
		\end{align*}where we used \eqref{19}$_3$ in the last inequality. 
		By \eqref{3.8}, $S_2$ can be estimated as 
		\begin{align*}
			|S_2|\lesssim \eta\sum_{i,j=1}^3\|\partial_{x_ix_j}\phi_c\|_{L^2_x}^2+C_\eta\|\partial^\alpha\{\I-\P\}f\|^2_{L^2_xL^2_{\gamma/2}}\lesssim \eta\|{\partial^\alpha c}\|^2_{L^2_{x}}+C_\eta\|\partial^\alpha \{\I-\P\}{f}\|^2_{L^2_{x}L^2_{\gamma/2}},
		\end{align*}
	for any $\eta>0$. 
		For $S_3$, applying \eqref{3.10}, \eqref{3.12} and \eqref{4.28}, we have 
		\begin{align*}
			|S_3|\le \eta\sum_{|\alpha|\le K}\|{\partial^\alpha c}\|^2_{L^2_{x}}+C_\eta\|\{\I-\P\}{\partial^\alpha f}\|^2_{L^2_{x}L^2_{\gamma/2}}+C_\eta\|({\partial^\alpha g},\zeta)_{L^2_v}\|^2_{L^2_{x}}.
		\end{align*}
		For $S_4$, we obtain from \eqref{3.10}, \eqref{3.12} and \eqref{4.28} that 
		\begin{align*}
			|S_4|\lesssim C_\eta\|\pa^\alpha \na_x\phi\|_{L^2_x}^2 + \eta\sum_{|\alpha|\le K}\|\partial^\alpha c\|_{L^2_x}^2. 
		\end{align*}
		For $S_5$, we need to use the boundary condition from Lemma \ref{Lem24} and \eqref{38}: 
		\begin{align*}
			S_5 &= -\int_{\partial\Omega}(v\cdot n(x){\partial^\alpha f}(x),{\Phi_c}(x))_{L^2_v}\,dS(x).
		\end{align*}
		Divide the integral on $\partial\Omega$ into three parts, $\Gamma_i$ $(i=1,2,3)$, and consider each component $\Gamma_i$ separately. Fix $i=1,2,3$. Then on $\Gamma_i$, we have $\partial_{n} = \partial_{x_i}$ or $-\pa_{x_i}$. 
Then 
		\begin{multline}
			\int_{\Gamma_i}(v\cdot n(x){\partial^\alpha f}(x),{\Phi_c}(x))_{L^2_v}\,dS(x)\\
			= \int_{\Gamma_i}\int_{\R^3}v\cdot n(x){\partial^\alpha f}(t,x,v)(|v|^2-5)\big(v\cdot{\nabla_{x}\phi_c}(t,x)\big)\mu^{1/2}\,dvdS(x).\label{20}
		\end{multline}
		If $\alpha_i=0$ or $2$, then we deduce from  \eqref{33}, \eqref{34} and \eqref{37} that $\pa^\al f(R_xv)=\pa^\al f(v)$ and from \eqref{120} that $\partial_{x_i}\phi_c=0$. Applying change of variable $v\mapsto R_xv$, \eqref{20} becomes 
		\begin{align*}
			&\quad\,\int_{\Gamma_i}\int_{\R^3}v\cdot n(x){\partial^\alpha f}(t,x,v)(|v|^2-5)\sum_{j\neq i}\big(v_j\pa_{x_j}\phi_c(t,x)\big)\mu^{1/2}\,dvdS(x)\\
			&=\int_{\Gamma_i}\int_{\R^3}R_xv\cdot n(x){\partial^\alpha f}(t,x,R_xv)(|R_xv|^2-5)\sum_{j\neq i}\big(R_xv_j\pa_{x_j}\phi_c(t,x)\big)\mu^{1/2}\,dvdS(x)\\
			&=\int_{\Gamma_i}\int_{\R^3}(-v\cdot n(x)){\partial^\alpha f}(t,x,v)(|v|^2-5)\sum_{j\neq i}\big(v_j\pa_{x_j}\phi_c(t,x)\big)\mu^{1/2}\,dvdS(x) = 0.
		\end{align*}
		
		If $\alpha_i = 1$ or $3$, then from boundary conditions, we have $\pa^\al f(R_xv)=-\pa^\al f(v)$ and $\partial_{x_j}\phi_c=0$ on $\Gamma_i$ for any $j\neq i$. Applying change of variable $v\mapsto R_xv$ to \eqref{20} and using \eqref{35}, we obtain 
		\begin{align*}
			&\quad\,\int_{\Gamma_i}\int_{\R^3}v\cdot n(x){\partial^\alpha f}(t,x,v)(|v|^2-5)v_i\pa_{x_i}\phi_c(t,x)\mu^{1/2}\,dvdS(x)\\
			&=\int_{\Gamma_i}\int_{\R^3}R_xv\cdot n(x){\partial^\alpha f}(t,x,R_xv)(|R_xv|^2-5)R_xv_i\pa_{x_i}\phi_c(t,x)\mu^{1/2}\,dvdS(x)\\
			&=\int_{\Gamma_i}\int_{\R^3}v\cdot n(x){\partial^\alpha f}(t,x,v)(|v|^2-5)(-v_i)\pa_{x_i}\phi_c(t,x)\mu^{1/2}\,dvdS(x) = 0. 
		\end{align*}
		Since the above estimates are valid for $i=1,2,3$, we obtain 
		\begin{align}\label{s5}
			S_5= 0. 
		\end{align}
		Combining the above estimates for $S_j$ $(1\le j\le 5)$, taking summation over $1\le|\alpha|\le K$ and letting $\eta$ suitably small, we obtain
		\begin{multline}\label{122ab}
			\partial_t\sum_{1\le|\alpha|\le K}(\partial^\alpha f,\Phi_c)_{L^2_{x,v}} + \lambda\sum_{1\le|\alpha|\le K}\|{\partial^\alpha c}\|^2_{L^2_{x}} 
			\lesssim \eta \sum_{1\le|\alpha|\le K}\|{\partial^\alpha b}\|^2_{L^2_{x}}+\sum_{|\alpha|\le K}\|\pa^\alpha \na_x\phi\|_{L^2_x}^2\\
			+C_\eta\sum_{|\alpha|\le K}\|\{\I-\P\}{\partial^\alpha f}\|^2_{L^2_{x}L^2_{\gamma/2}} 
			+C_\eta\sum_{|\alpha|\le K}\|({\partial^\alpha g},\zeta)_{L^2_v}\|^2_{L^2_{x}},
		\end{multline}for some $\lambda>0$ and any $\eta>0$. Note that we have applied \eqref{2.14}. 
	The estimate \eqref{122ab} gives derivatives estimate on $c$. For the zeroth derivative of $c$, we apply the Poincar\'{e}'s inequality and \eqref{conservation_bounded} to obtain that 
	\begin{align*}
		\|c\|_{L^2_x}\lesssim \|\na_xc\|_{L^2_x} + \Big|\int_\Omega c\,dx\Big|
		&\lesssim \|\na_xc\|_{L^2_x} + \|E\|_{L^2_x}^2.
	\end{align*}
	Plugging this estimate into \eqref{122ab}, we have 	
	\begin{multline}\label{122a}
		\partial_t\sum_{1\le|\alpha|\le K}(\partial^\alpha f,\Phi_c)_{L^2_{x,v}} + \lambda\sum_{|\alpha|\le K}\|{\partial^\alpha c}\|^2_{L^2_{x}} 
		\lesssim \eta \sum_{1\le|\alpha|\le K}\|{\partial^\alpha b}\|^2_{L^2_{x}}+\sum_{|\alpha|\le K}\|\pa^\alpha \na_x\phi\|_{L^2_x}^2\\
		+C_\eta\sum_{|\alpha|\le K}\|\{\I-\P\}{\partial^\alpha f}\|^2_{L^2_{x}L^2_{\gamma/2}} 
		+C_\eta\sum_{|\alpha|\le K}\|({\partial^\alpha g},\zeta)_{L^2_v}\|^2_{L^2_{x}}+ \|E\|_{L^2_x}^4,
	\end{multline}

		\medskip \noindent{\bf Step 2. Estimate of ${b}(t,x)$.}
		Next we consider the estimate of ${b}$. For this purpose we choose 
		\begin{align*}
			{\Phi}={\Phi_b}=\sum^3_{m=1}{\Phi^{j,m}_b},\ j=1,2,3,
		\end{align*}
		where 
		\begin{equation*}
			{\Phi^{j,m}_b}=\left\{\begin{aligned}
				\big(|v|^2v_mv_j{\partial_{x_m}\phi_j}-\frac{7}{2}(v_m^2-1){\partial_{x_j}\phi_j}\big)\mu^{1/2},\ m\neq j,\\
				\frac{7}{2}(v_j^2-1){\partial_{x_j}\phi_j}\mu^{1/2},\qquad\qquad\qquad m=j,
			\end{aligned}\right.
		\end{equation*}
		and $\phi_j$($1\le j\le 3$) solves  
		\begin{equation}\label{120a}\left\{\begin{aligned}
				&-\Delta_x \phi_j = {\partial^\alpha b_j}\ \text{ in }\Omega,\\
				&{\phi_k}(x) =\partial_n\phi_m(x)= 0 \ \text{ on }\ x\in \Gamma_m,\ \text{ for }k\neq m, \text{ if }\alpha_m = 1\text{ or }3,\\
				&\phi_m(x)={\partial_n\phi_k}(x)= 0\ \text{ on }\ x\in \Gamma_m,\ \text{ for }k\neq m, \text{ if }\alpha_m = 0\text{ or } 2,\ \forall\, m=1,2,3.
			\end{aligned}\right.
		\end{equation}
		The existence of solutions to \eqref{120a} is given by \cite[Lamma 4.4.3.1]{Grisvard1985}. We will explain the sufficient and necessary conditions for existence to pure Neumann type problem later. 
		By using the boundary value of $\phi_j$, similar to \eqref{3.8}, we have 
		\begin{equation}\label{3.18}\begin{aligned}
				\sum_{i,k=1}^3\|\partial_{x_ix_k}{\phi_j}\|_{L^2_{x}}^2& = \|\Delta_x\phi_j\|_{L^2_x}^2 \lesssim \|\partial^\alpha b_j\|^2_{L^2_x}.
			\end{aligned} 	
		\end{equation}
		Then $S_2$ can be estimated as 
		\begin{align}\label{3.19a}
			|S_2|\lesssim \|\partial^\alpha\{\I-\P\}f\|_{L^2_xL^2_{\gamma/2}}\sum_{i,j,m=1}^3\|\partial_{x_ix_m}\phi_j\|_{L^2_x}
			\lesssim C_\eta\|\partial^\alpha\{\I-\P\}f\|_{L^2_xL^2_{\gamma/2}}+\eta \|\partial^\alpha b\|^2_{L^2_x}.
		\end{align}
		Next we fix $1\le j\le 3$ and discuss the value of $|\alpha|$ in the following cases. 
		If $|\alpha|=0$, then $\eqref{120a}$ is mixed Neumann-Dirichlet boundary problem. Then by standard elliptic estimates, we have 
		\begin{align}\label{5.21c}
			\|\na_x\phi_j\|_{L^2_x}\lesssim \|b_j\|_{L^2_x},\qquad
			\|\partial_t\na_x\phi_j\|_{L^2_x}\lesssim \|\partial_t b_j\|_{L^2_x},
		\end{align}

		If  $|\alpha|=1$. Then $\alpha_i=1$ for some $1\le i\le 3$ and $\alpha_k=0$ for $k\neq i$.
		In particular, if $j=i$, then $\partial_{x_i}\phi_i=0$ on $\Gamma_i$ and $\partial_{x_k}\phi_i=0$ on $\Gamma_k$ for $k\neq i$. In this case, \eqref{120a} is a pure Neumann boundary problem and we need $\int_{\Omega}\partial_{x_i}b_i\,dx=\int_{\Gamma_i}b_i\,dS(x)=0$ to ensure the existence for \eqref{120a}, which follows from \eqref{Lem25a}. In this case, $\pa_{x_m}\phi_i=0$ on a subset of boundary $\partial\Omega$ with non-zero spherical measure for any $m=1,2,3$. 
		By Sobolev embedding \cite[Theorem 6.7-5]{Ciarlet2013}, we have from \eqref{3.18} that 
		\begin{align}\label{3.21}
			\|\partial_{x_m}\phi_i\|_{L^2_x}\lesssim \|\na_x\partial_{x_m}\phi_i\|_{L^2_x}\lesssim \|\partial^\alpha b_i\|_{L^2_x},
		\end{align}
	 and 
		\begin{align}\label{3.21a}
			\|\partial_t\partial_{x_m}\phi_i\|_{L^2_x}\lesssim \|\partial_t\na_x\partial_{x_m}\phi_i\|_{L^2_x}
			\lesssim \|\partial_t\partial^\alpha b_i\|_{L^2_x},
		\end{align}
	for any $m=1,2,3$.

		If $j\neq i$, then $\phi_j=0$ on $\Gamma_i$ and $\Gamma_j$ while $\partial_{x_k}\phi_j=0$ on $\Gamma_k$ for $k\neq j,i$. \eqref{120a} is a mixed Dirichlet-Neumann boundary problem. By Sobolev embedding \cite[Theorem 6.7-5]{Ciarlet2013}, we have $\|\partial_t\phi_j\|_{L^2_x}\lesssim \|\partial_t\na_x\phi_j\|_{L^2_x}$ and $\|\phi_j\|_{L^2_x}\lesssim \|\na_x\phi_j\|_{L^2_x}$. Thus, by standard elliptic estimates for \eqref{120a}, we have 
		\begin{align}\label{3.21c}
			\|\na_x\phi_j\|_{L^2_x}\lesssim \|\partial^\alpha b_j\|_{L^2_x},\qquad
			\|\partial_t\na_x\phi_j\|_{L^2_x}\lesssim \|\partial_t\partial^\alpha b_j\|_{L^2_x}.
		\end{align}

		Next we assume $|\alpha|=2$ and $\partial^\alpha = \partial_{x_ix_i}$ for some $1\le i\le 3$.
		Then for $j=1,2,3$, $\phi_j=0$ on $\Gamma_j$ and $\partial_{x_k}\phi_j=0$ on $\Gamma_k$ for $k\neq j$.
		Thus \eqref{120a} is a mixed Dirichlet-Neumann boundary problem and by Sobolev embedding \cite[Theorem 6.7-5]{Ciarlet2013}, we know that $\|\phi_j\|_{L^2_x}\lesssim \|\nabla_x\phi_j\|_{L^2_x}$. Then by standard elliptic estimates for \eqref{120a}, we have 
		\begin{align*}
			\|\na_x\phi_j\|_{L^2_x}^2= \int_\Omega\partial_{x_ix_i}b_j\,\phi_j\,dx
			=  - \int_\Omega\partial_{x_i}b_j\,\partial_{x_i}\phi_j\,dx,
		\end{align*}
		where we used $\pa_{x_i}b_j=0$ on $\Gamma_i$ from \eqref{Lem25a} for $j\neq i$ and $\phi_j=0$ on $\Gamma_i$ if $j=i$. Then we have 
		\begin{align}\label{3.27}
			\|\na_x\phi_j\|_{L^2_x}\le \|\partial_{x_i}b_j\|_{L^2_x},\qquad
			\|\partial_t\na_x\phi_j\|_{L^2_x}\lesssim\|\partial_t\partial_{x_i}b_j\|_{L^2_x}. 
		\end{align}
		
		If $|\alpha|=3$, then $\alpha_i=1$ or $3$ for some $1\le i\le 3$. If further $\alpha_k=2$ or $0$ and $\alpha_m=0$ for some $k\neq i$ and $m\neq k,i$, then \eqref{120a} is pure Neumann problem when $i=j$. Here we need  $\int_{\Omega}\pa_{x_ix_kx_k}b_i\,dx=\int_{\Gamma_k}\pa_{x_ix_k}b_i\,dS(x)=0$
		or  
		$\int_{\Omega}\pa_{x_ix_ix_i}b_i\,dx=\int_{\Gamma_i}\pa_{x_ix_i}b_i\,dS(x)=0$
		to ensure the existence of \eqref{120a}, which follows from \eqref{Lem25a} and \eqref{38a}. In any cases, we write $\pa^\alpha = \pa_{x_ix_kx_m}$ for some $1\le k,m\le 3$. Here either $k=m=i$ or $k,m\neq i$. Then taking inner product of \eqref{120a} with $\phi_j$, we have 
		\begin{multline*}
			\|\na_x\phi_j\|^2_{L^2_x} = \int_{\Omega}\pa_{x_ix_kx_m}b_j\,\phi_j\,dx
			= \int_{\Gamma_i}\pa_{x_kx_m}b_j\,\phi_j\,dx - \int_{\Omega}\pa_{x_kx_m}b_j\,\pa_{x_i}\phi_j\,dx\\
			\le \|\pa_{x_kx_m}b_j\|_{L^2_x}\|\pa_{x_i}\phi_j\|_{L^2_x},
		\end{multline*}
		where we used the fact that $\pa_{x_kx_m}b_j=0$ on $\Gamma_i$ when $i=j$ and $\phi_j=0$ on $\Gamma_i$ when $i\neq j$, which is from \eqref{Lem25a} and boundary condition \eqref{120a}. 
		This implies that 
		\begin{align}\label{4.42}
			\|\na_x\phi_j\|_{L^2_x}\le \sum_{|\alpha|=2}\|\pa^\alpha b_j\|_{L^2_x},\qquad 
			\|\pa_t\na_x\phi_j\|_{L^2_x}\le \sum_{|\alpha|=2}\|\pa_t\pa^\alpha b_j\|_{L^2_x}.
		\end{align}
		As a summary, for $|\al|\le K$, we have from \eqref{5.21c}, \eqref{3.21}, \eqref{3.21a}, \eqref{3.21c}, \eqref{3.27} and \eqref{4.42} that 
		\begin{align}\label{b1}
			\|\na_x\phi_j\|_{L^2_x}\le \sum_{|\alpha|\le K-1}\|\pa^\alpha b_j\|_{L^2_x},\quad
			\|\pa_t\na_x\phi_j\|_{L^2_x}\le \sum_{|\alpha|\le K-1}\|\pa_t\pa^\alpha b_j\|_{L^2_x}.
		\end{align}

		Now we let $|\alpha|\le K$. 
		For $S_1$, we have from \eqref{b1} that 
		\begin{align}\label{4.31}\notag
			\dis|S_1| &\le  \big(\P \partial^\alpha f,\partial_t\Phi_b\big)_{L^2_{x,v}}+\big(\{\I-\P\} \partial^\alpha f,\partial_t\Phi_b\big)_{L^2_{x,v}}\\
			&\notag\lesssim C_\eta\|\partial^\alpha c\|_{L^2_x}^2 + C_\eta\|\{\I-\P\} \partial^\alpha f\|_{L^2_xL^2_{\gamma/2}}^2 + \eta\sum_{|\alpha|\le K-1}\|\partial_t\pa^\alpha b_j\|_{L^2_x}^2\\
			&\notag\lesssim  \sum_{|\alpha|\le K}\|\partial^\alpha c\|^2_{L^2_x} +C_\eta\sum_{|\alpha|\le K}\|\partial^\alpha \{\I-\P\}f\|^2_{L^2_xL^2_{\gamma/2}} \\&\qquad+\eta\sum_{1\le|\alpha|\le K}\|\partial^\alpha (a_++a_-)\|^2_{L^2_x}+\eta\sum_{|\alpha|\le K}\|(\partial^\alpha g,\zeta)_{L^2_v}\|^2_{L^2_x},
		\end{align}
		where we used \eqref{19}$_2$.
		For $S_3$, by \eqref{b1}, we have 
		\begin{align}\label{3.22a}\notag
			|S_3|
			&\lesssim\notag C_\eta\|\partial^\alpha \{\I-\P\}f\|^2_{L^2_xL^2_{\gamma/2}} +C_\eta\|(\partial^\alpha g,\zeta)_{L^2_v}\|^2_{L^2_x}+ \eta\|\nabla_x\phi_j\|_{L^2_x}\\
			&\lesssim C_\eta\|\partial^\alpha \{\I-\P\}f\|_{L^2_xL^2_{\gamma/2}}^2 + C_\eta\|(\pa^\alpha g,\zeta(v))_{L^2_v}\|_{L^2_x}^2 + \eta\sum_{|\alpha|\le K-1}\|\partial^\alpha b\|_{L^2_x},
		\end{align}
	for any $\eta>0$. 
		For $S_4$, we apply \eqref{b1} to obtain 
		\begin{align}\label{3.28a}
			|S_4|\le C_\eta \|\partial^\alpha\na_x\phi\|^2_{L^2_x}+\eta\sum_{|\alpha|\le K-1}\|\partial^\alpha b\|_{L^2_x},
		\end{align}
	for any $\eta>0$. 
		For the second term on left hand side of \eqref{100}, we have  
		\begin{align}\label{3.30}
			&\notag\quad\,-\sum^3_{m=1}(\PP {\partial^\alpha f},v\cdot\na_x{\Phi^{j,m}_b})_{L^2_{x,v}}\\
			&\notag=-\sum^{3}_{m=1,m\neq j}(v_mv_j\mu^{1/2}{\partial^\alpha b_j},|v|^2v_mv_j\mu^{1/2}{\partial_{x_m}^2\phi_j})_{L^2_{x,v}}\\
			&\notag\qquad-\sum^{3}_{m=1,m\neq j}(v_mv_j\mu^{1/2}{\partial^\alpha b_m},|v|^2v_mv_j\mu^{1/2}{\partial_{x_m}\partial_{x_j}\phi_j})_{L^2_{x,v}}\\
			&\notag\qquad+7\sum^{3}_{m=1,m\neq j}({\partial^\alpha b_m},{\partial_{x_m}\partial_{x_j}\phi_j})_{L^2_{x}}-7({\partial^\alpha b_m},{\partial^2_{x_j}\phi_j})_{L^2_{x}}\\
			&= -7 \sum^3_{m=1}({\partial^\alpha b_j},{\partial_{x_m}^2\phi_j})_{L^2_{x}}=7\|{\partial^\alpha b_j}\|^2_{L^2_{x}}.
		\end{align}Note that $\int_{\R^3}v_m^2(v_m^2-1)\mu\,dv=2$, $\int_{\R^3}v_m^2(v_j^2-1)\mu\,dv=0$, $\int_{\R^3}v_m^2v_j^2|v|^2\mu\,dv=7$ and $\int_{\R^3}(v_j^2-1)\mu\,dv=0$, when $m\neq j$. 
		%
		%
		%
		%
		Now we consider the boundary term $S_5$. As in the estimate on $c(t,x)$, we consider $\Gamma_i$ for fixed $i=1,2,3$: 
		\begin{multline}\label{119}
			\int_{\Gamma_i}(v\cdot n(x){\partial^\alpha f}(x),{\Phi_b}(x))_{L^2_v}\,dS(x)\\
			= \sum_{m=1}^3\int_{\Gamma_i}\int_{\R^3}v\cdot n(x){\partial^\alpha f}(t,x,v){\Phi^{j,m}_b}(x,v)\,dvdS(x).
		\end{multline}
		If $\alpha_i=0$ or $2$, then applying boundary condition \eqref{120a}, we have that for $x\in\Gamma_i$, 
		\begin{align*}
			\partial_{x_i}\phi_j(x) = \partial_{x_j}\phi_i(x) = 0,\quad\text{ for }j\neq i. 
		\end{align*}
		This shows that $\Phi^{j,m}_b(x,v)$ is even with respect to $v_i$ when $x\in\Gamma_i$. 
		Noticing $R_xv = v-2v\cdot e_je_j$ maps $v_i$ to $-v_i$ on $\Gamma_i$, we know that 
		\begin{align*}
			{\Phi^{j,m}_b}(x,R_xv) = {\Phi^{j,m}_b}(x,v),\ \text{ for }m=1,2,3.
		\end{align*} 		
		Applying change of variable $v\mapsto R_xv$ and using identities \eqref{33}, \eqref{34} and \eqref{37}, \eqref{119} becomes 
		\begin{align*}
			&\quad\, \sum_{m=1}^3\int_{\Gamma_i}\int_{\R^3}v\cdot n(x){\partial^\alpha f}(t,x,v){\Phi^{j,m}_b}(x,v)\,dvdS(x)\\
			&=\sum_{m=1}^3\int_{\Gamma_i}\int_{\R^3}R_xv\cdot n(x){\partial^\alpha f}(t,x,R_xv){\Phi^{j,m}_b}(x,R_xv)\,dvdS(x)\\
			&=\sum_{m=1}^3\int_{\Gamma_i}\int_{\R^3}-v\cdot n(x){\partial^\alpha f}(t,x,v){\Phi^{j,m}_b}(x,v)\,dvdS(x)=0.
		\end{align*}
		If $\alpha_i=1$ or $3$, then boundary condition \eqref{120a} shows that on $x\in\Gamma_i$, 
		\begin{align*}
			\partial_{x_j}\phi_j(x) &= 0, \quad \text{ for }j=1,2,3,\\
			\partial_{x_m}\phi_j(x) &= 0,\quad \text{ for } j,m\neq i.
		\end{align*}
		Note that $\pa_{x_m}$ is tangent derivative on $\Gamma_i$ when $m\neq i$. 
		Then we know that $\Phi^{j,m}_b(x,v)$ is odd with respect to $v_i$ when $x\in\Gamma_i$ and hence, 
		\begin{align*}
			{\Phi^{j,m}_b}(x,R_xv) = -{\Phi^{j,m}_b}(x,v).
		\end{align*}
		Now applying change of variable $v\mapsto R_xv$ and using identities \eqref{35}, \eqref{119} becomes 
		\begin{align*}
			&\quad\, \sum_{m=1}^3\int_{\Gamma_i}\int_{\R^3}v\cdot n(x){\partial^\alpha f}(t,x,v){\Phi^{j,m}_b}(x,v)\,dvdS(x)\\
			&=\sum_{m=1}^3\int_{\Gamma_i}\int_{\R^3}R_xv\cdot n(x){\partial^\alpha f}(t,x,R_xv){\Phi^{j,m}_b}(x,R_xv)\,dvdS(x)\\
			&=\sum_{m=1}^3\int_{\Gamma_i}\int_{\R^3}v\cdot n(x){\partial^\alpha f}(t,x,v){(-\Phi^{j,m}_b)}(x,v)\,dvdS(x)=0.
		\end{align*}
		Therefore, \begin{equation}\label{3.31}
			S_5=0.
		\end{equation} Combining estimates \eqref{3.19a}, \eqref{4.31}, \eqref{3.22a}, \eqref{3.28a}, \eqref{3.30} and \eqref{3.31}, taking summation over $|\alpha|\le K$ of \eqref{100} and letting $\eta$ sufficiently small, we have 
		\begin{multline}\label{122b}
			\partial_t\sum_{|\alpha|\le K}(\partial^\alpha f,\Phi_b)_{L^2_{x,v}} + \lambda\sum_{|\alpha|\le K}\|{\partial^\alpha b}\|^2_{L^2_{x}}\lesssim \eta\sum_{1\le|\alpha|\le K}\|\partial^\alpha (a_++a_-)\|^2_{L^2_{x}}+ C_\eta\sum_{|\alpha|\le K}\|\pa^\alpha{c}\|^2_{L^2_{x}}\\+\sum_{|\alpha|\le K}\|\partial^\alpha\na_x\phi\|^2_{L^2_x}
			+C_\eta\sum_{|\alpha|\le K}\|\pa^\alpha\{\I-\P\}{f}\|^2_{L^2_{x}L^2_{\gamma/2}}  +C_\eta\sum_{|\alpha|\le K}\|({\partial^\alpha g},\zeta)_{L^2_v}\|^2_{L^2_{x}},
		\end{multline}for some $\lambda>0$ and any $\eta>0$. Note that we have applied \eqref{2.14}.

		\medskip\noindent{\bf Step 3. Estimate on ${a_+}(t,x)+a_-(t,x)$:} We choose the following two test functions
		\begin{align*}
			{\Phi} = {\Phi_{a\pm}} = (|v|^2-10)\big(v\cdot{\nabla_{x}\phi_{a\pm}}(t,x)\big)\mu^{1/2},
		\end{align*}
		where $\phi_a = (\phi_{a+}(x),\phi_{a_-}(x))$ solves 
		\begin{equation}\label{122}\left\{\begin{aligned}
				&-\Delta_x \phi_{a+} =-\Delta_x \phi_{a-} = {\partial^\alpha (a_++a_-)}\ \text{ in }\Omega,\\
				&{\phi_a}(x)= 0 \ \text{ on }\ x\in \Gamma_i,\ \text{ if }\alpha_i = 1\text{ or }3,\\
				&\frac{\partial\phi_a}{\partial n}(x)= 0\ \text{ on }\ x\in \Gamma_i,\ \text{ if }\alpha_i = 0\text{ or } 2.
			\end{aligned}\right.
		\end{equation}
		%
		The existence and uniqueness of solution to \eqref{122} is guaranteed by \cite[Lamma 4.4.3.1]{Grisvard1985}.
		When $|\alpha|=0$, \eqref{122} is pure Neumann problem and we need $\int_{\Omega}(a_++a_-)\,dx=0$ from conservation laws \eqref{conservation_bounded} to ensure the existence of \eqref{122}. 
		When $|\al|=2$ and $\alpha_i=2$ for some $i$, \eqref{122} is pure Neumann problem and we need $\int_{\Omega}\partial_{x_ix_i}(a_++a_-)\,dx=\int_{\Gamma_i}\partial_{x_i}(a_++a_-)\,dS(x)=0$ from Lemma \ref{Lem25} to ensure the existence of \eqref{122}. 
		Now we compute \eqref{100}. For the second term on left hand side of \eqref{100}, taking summation on $\pm$, we have 
		\begin{align*}
			&\quad\,-\sum_{\pm}({\partial^\alpha \PP f},v\cdot{\nabla_{x}\Phi_{a\pm}})_{L^2_{x,v}} \\
			&= -\sum_{\pm}\sum_{j,m=1}^3({\partial^\alpha a_\pm}+{\partial^\alpha b}\cdot v+(|v|^2-3){\partial^\alpha c} ,v_jv_m(|v|^2-10)\mu{(\partial_{x_j}\partial_{x_m}\phi_{a\pm})^\wedge})_{L^2_{x,v}} \\
			&= \sum_{\pm}\sum_{j=1}^3({\partial^\alpha a_\pm} ,{-\partial^2_j\phi_{a\pm}})_{L^2_{x}}  = \|{\partial^\alpha a_+}+{\partial^\alpha a_-}\|^2_{L^2_{x}} .
		\end{align*}
		The estimates for $S_j$ $(1\le j\le 5)$ are similar to the case of $c(t,x)$ from \eqref{3.8} to \eqref{122a}, since $\Phi_a$ and $\Phi_c$ has similar structure. Then following the calculation from \eqref{3.8}
		to \eqref{4.29}, we have that for $|\alpha|\le K$, 
		\begin{equation}\label{3.32}
			\sum_{i,j=1}^3\|\partial_{x_ix_j}\phi_a\|^2_{L^2_x}=\|\Delta_x\phi_a\|^2_{L^2_x}\lesssim \|\partial^\alpha (a_++a_-)\|^2_{L^2_x},
		\end{equation}
		\begin{equation}\label{3.33}
			\|\nabla_x\phi_a\|_{L^2_x}\lesssim \sum_{|\alpha|\le K-1}\|\pa^\alpha(a_++a_-)\|_{L^2_x},
		\end{equation}
		and 
		\begin{equation}\label{3.34}
			\|\pa_t\nabla_x\phi_a\|_{L^2_x}\lesssim \sum_{|\alpha|\le K-1} \|\pa_t\pa^\alpha(a_++a_-)\|_{L^2_x}\lesssim \sum_{1\le|\alpha|\le K}\|\pa^\alpha b\|_{L^2_x},
		\end{equation}
		where the last inequality follows from \eqref{19}$_1$. 
		Then for $S_1$, we apply \eqref{3.34} to obtain 
		\begin{align*}
			|S_1|&\le \big|\big(\{\I-\P\}\partial^\alpha f, \pa_t\Phi_a\big)_{L^2_{x,v}}\big|
			+\big|\big(\P\partial^\alpha f, \pa_t\Phi_a\big)_{L^2_{x,v}}\big|\\
			&\lesssim \|\{\I-\P\}\partial^\alpha f\|^2_{L^2_xL^2_{\gamma/2}} + \|\partial^\alpha b\|_{L^2_x}^2 + \|\pa_t\na_x\phi_a\|^2_{L^2_x}\\
			&\lesssim \|\{\I-\P\}\partial^\alpha f\|^2_{L^2_xL^2_{\gamma/2}} + \sum_{|\alpha|\le K}\|\partial^\alpha b\|_{L^2_x}^2. 
		\end{align*}	
		For $S_2$, by \eqref{3.32}, we have 
		\begin{align*}
			|S_2|\lesssim C_\eta\|\{\I-\P\}\partial^\alpha f\|^2_{L^2_xL^2_{\gamma/2}} + \eta\|\partial^\alpha (a_++a_-)\|^2_{L^2_x}. 
		\end{align*}
		For $S_3$, by \eqref{3.33}, we have 
		\begin{align*}
			|S_3|\lesssim C_\eta \|\{\I-\P\}\partial^\alpha f\|^2_{L^2_xL^2_{\gamma/2}} + C_\eta \|(\partial^\alpha g,\zeta)_{L^2_v}\|^2_{L^2_x} + \eta\sum_{|\alpha|\le K-1} \|\pa^\alpha(a_++a_-)\|_{L^2_x}^2. 
		\end{align*}
		For $S_4$, we apply \eqref{3.33} to obtain 
		\begin{align*}
			|S_4|\le C_\eta \|\partial^\alpha\na_x\phi\|^2_{L^2_x}+\eta\sum_{|\alpha|\le K-1}\|\partial^\alpha (a_++a_-)\|_{L^2_x}.
		\end{align*}
		For $S_5$, since $\Phi_a$ and $\Phi_c$ has the same structure, following the arguments deriving \eqref{s5}, we have 
			$S_5=0.$ 
		Combining the above estimates, taking summation $|\alpha|\le K$ and $\pm$ of \eqref{100} and letting $\eta>0$ small enough, we have 
		\begin{multline}\label{122d}
			\partial_t\sum_{|\alpha|\le K}(\partial^\alpha f,\Phi_a)_{L^2_{x,v}} + \lambda\sum_{|\alpha|\le K}\|{\partial^\alpha (a_++a_-)}\|^2_{L^2_{x}}
			\lesssim\sum_{|\alpha|\le K} \|\partial^\alpha\{\I-\P\}{f}\|^2_{L^2_{x}L^2_{\gamma/2}}\\
			+\sum_{|\alpha|\le K}\|\pa^\alpha\na_x\phi\|^2_{L^2_x}+\sum_{|\alpha|\le K}\|{\partial^\alpha b}\|^2_{L^2_{x}} +\sum_{|\alpha|\le K}\|({\partial^\alpha g},\zeta)_{L^2_v}\|^2_{L^2_{x}}.
		\end{multline}
		Note that we have applied \eqref{2.14}.

		\medskip\noindent{\bf Step 4. Estimate on ${a_+}(t,x)-{a_-}(t,x)$ and ${E}(t,x)$:} We choose the following two test functions
		\begin{align*}
			{\Phi} = \tilde{\Phi}_{a\pm} = (|v|^2-10)\big(v\cdot{\nabla_{x}\phi_{a\pm}}(t,x)\big)\mu^{1/2},
		\end{align*}
		where $\phi_a = (\phi_{a+}(x),\phi_{a_-}(x))$ solves 
		\begin{equation}\label{102}\left\{\begin{aligned}
				&-\Delta_x \phi_{a+} = {\partial^\alpha (a_+-a_-)}\ \text{ in }\Omega,\\
				&-\Delta_x \phi_{a-} = {\partial^\alpha (a_--a_+)}\ \text{ in }\Omega,\\
				&{\phi_a}(x)= 0 \ \text{ on }\ x\in \Gamma_i,\ \text{ if }\alpha_i = 1\text{ or }3,\\
				&\partial_n\phi_a(x)= 0\ \text{ on }\ x\in \Gamma_i,\ \text{ if }\alpha_i = 0\text{ or } 2.
			\end{aligned}\right.
		\end{equation}
		The existence and uniqueness of solution to \eqref{122} is guaranteed by \cite[Lamma 4.4.3.1]{Grisvard1985}. 
		When $|\alpha|=0$, \eqref{102} is a pure Neumann problem and we need $\int_{\Omega}(a_+-a_-)\,dx=0$ from conservation laws \eqref{conservation_bounded} to ensure the existence of \eqref{102}. 
		When $|\al|=2$ and $\alpha_i=2$ for some $i$, \eqref{102} is also pure Neumann problem and we need $\int_{\Omega}\partial_{x_ix_i}(a_+-a_-)\,dx=\int_{\Gamma_i}\partial_{x_i}(a_+-a_-)\,dS(x)=0$ from Lemma \ref{Lem25} to ensure the existence of \eqref{102}. 
		For the second term on left hand side of \eqref{100}, taking summation on $\pm$, we have 
		\begin{align*}
			&\quad\,-\sum_{\pm}({\partial^\alpha \PP f},v\cdot{\nabla_{x}\tilde{\Phi}_{a\pm}})_{L^2_{x,v}} \\
			&= -\sum_{\pm}\sum_{j,m=1}^3({\partial^\alpha a_\pm}+{\partial^\alpha b}\cdot v+(|v|^2-3){\partial^\alpha c} ,v_jv_m(|v|^2-10)\mu{(\partial_{x_j}\partial_{x_m}\phi_{a\pm})^\wedge})_{L^2_{x,v}} \\
			&= \sum_{\pm}\sum_{j=1}^3({\partial^\alpha a_\pm} ,{-\partial^2_j\phi_{a\pm}})_{L^2_{x,v}}  = \|{\partial^\alpha a_+}-{\partial^\alpha a_-}\|^2_{L^2_{x,v}} .
		\end{align*}
		Then we estimate the $S_j$ $(1\le j\le 5)$.
		The same as the case of $a_++a_-$, the case of $a_+-a_-$ is also similar to estimate of $c(t,x)$. 
		Following the calculation from \eqref{3.8} to \eqref{4.29}, we have that for $|\alpha|\le K$, 
		\begin{equation}\label{3.32q}
			\sum_{i,j=1}^3\|\partial_{x_ix_j}\phi_a\|^2_{L^2_x}=\|\Delta_x\phi_a\|^2_{L^2_x}\lesssim \|\partial^\alpha (a_+-a_-)\|^2_{L^2_x},
		\end{equation}
		\begin{equation}\label{3.33q}
			\|\nabla_x\phi_a\|_{L^2_x}\lesssim \sum_{|\alpha|\le K-1}\|\pa^\alpha(a_+-a_-)\|_{L^2_x},
		\end{equation}
		and 
		\begin{equation}\label{3.34q}
			\|\pa_t\nabla_x\phi_a\|_{L^2_x}\lesssim \sum_{|\alpha|\le K-1}\|\pa_t\pa^\alpha(a_+-a_-)\|_{L^2_x}\lesssim \sum_{1\le|\alpha|\le K}\|\pa^\alpha \{\I-\P\}f\|_{L^2_xL^2_{\gamma/2}}.
		\end{equation}
		where the last inequality follows from \eqref{98a}$_1$. 
		Then for $S_1$, we apply \eqref{3.34q} to obtain 
		\begin{align*}
			|S_1|&\le \big|\big(\{\I-\P\}\partial^\alpha f, \pa_t\Phi_a\big)_{L^2_{x,v}}\big|
			+\big|\big(\P\partial^\alpha f, \pa_t\Phi_a\big)_{L^2_{x,v}}\big|\\
			&\lesssim \|\{\I-\P\}\partial^\alpha f\|^2_{L^2_xL^2_{\gamma/2}} + \eta\|\partial^\alpha b\|_{L^2_x}^2 + C_\eta\|\pa_t\na_x\phi_a\|^2_{L^2_x}\\
			&\lesssim C_\eta\sum_{|\alpha|\le K}\|\{\I-\P\}\partial^\alpha f\|^2_{L^2_xL^2_{\gamma/2}}+ \eta\|\partial^\alpha b\|_{L^2_x}^2.
		\end{align*}	
		For $S_2$, by \eqref{3.32q}, we have 
		\begin{align*}
			|S_2|\lesssim C_\eta\|\{\I-\P\}\partial^\alpha f\|^2_{L^2_xL^2_{\gamma/2}} + \eta\|\partial^\alpha (a_+-a_-)\|^2_{L^2_x}. 
		\end{align*}
		For $S_3$, by \eqref{3.33q}, we have 
		\begin{align*}
			|S_3|\lesssim C_\eta \|\{\I-\P\}\partial^\alpha f\|^2_{L^2_xL^2_{\gamma/2}} + C_\eta \|(\partial^\alpha g,\zeta)_{L^2_v}\|^2_{L^2_x} + \eta\sum_{|\alpha|\le K-1} \|\pa^\alpha(a_+-a_-)\|_{L^2_x}^2. 
		\end{align*}
		For $S_4$, from \eqref{102} we know that $\phi_{a+} = -\phi_{a-}$.
		Thus
		\begin{align}\label{4.55}\notag
			\sum_\pm S_4 &= \sum_\pm\mp(\partial^\alpha {\nabla_x\phi}\cdot v\mu^{1/2},{\Phi_{a\pm}})_{L^2_{x,v}}\\
			&\notag= 5\sum_{\pm,j} \mp (\partial^\alpha {\partial_{x_j}\phi} ,{\partial_{x_j}\phi_{a\pm}}(t,x))_{L^2_{x}}\\
			&= -10\sum_{j=1}^3  (\partial^\alpha {\partial_{x_j}\phi} ,{\partial_{x_j}\phi_{a+}}(t,x))_{L^2_{x}}.
		\end{align}
		By using the boundary value from \eqref{102} and \eqref{4.4}, we know that if
		$\alpha_j=0$ or $2$, then $\pa_{x_j}\phi_a=0$ on $\Gamma_j$. If $\alpha_j=1$ or $3$, then $\pa^\alpha \phi=0$ on $\Gamma_j$, which is from \eqref{Neumann} and \eqref{4.16}. Thus, 
		\begin{align*}
			\sum_{j=1}^3  (\partial^\alpha {\partial_{x_j}\phi} ,{\partial_{x_j}\phi_{a+}})_{L^2_{x}} 
			&= 
			\sum_{j=1}^3  \int_{\Gamma_j}\partial^\alpha {\phi}\, {\partial_{x_j}\phi_{a+}}\,dS(x) - (\partial^\alpha {\phi} ,\Delta_x\phi_{a+})_{L^2_{x}} \\
			&=  
			(\partial^\alpha {\phi} ,\pa^\alpha(a_+-a_-))_{L^2_{x}}. 
		\end{align*}
		By \eqref{4.3}, we know that $-\Delta_x\phi = a_+-a_-$. 
		If $|\alpha|=0$, then by boundary condition \eqref{Neumann}, we have 
		\begin{align*}
			({\phi} ,(a_+-a_-))_{L^2_{x}}
			&= ({\phi} ,-\Delta_x\phi)_{L^2_{x}} = \|\na_x\phi\|_{L^2_x}^2. 
		\end{align*}
		If $\pa^\alpha=\pa_{x_i}$, then by the fact that $\partial_{x_i}\phi=0$ on $\Gamma_i$ and $\partial_{x_ix_j}\phi=0$ on $\Gamma_j$ for $j\neq i$, we have 
		\begin{align*}\notag
			(\partial^\alpha {\phi} ,\pa^\alpha(a_+-a_-))_{L^2_{x}}
			&= (\partial^\alpha {\phi} ,-\pa^\alpha\Delta_x\phi)_{L^2_{x}}=\|\partial^\alpha \na_x{\phi}\|^2_{L^2_{x}}.
		\end{align*}
		If $\pa^\alpha=\pa_{x_ix_i}$, then
		\begin{align*}
			(\partial^\alpha {\phi} ,\pa^\alpha(a_+-a_-))_{L^2_{x}}
			= \int_{\Gamma_i}\partial_{x_ix_i}\phi\,\pa_{x_i}(a_+-a_-)\,dS(x) - (\partial_{x_ix_ix_i}\phi,\pa_{x_i}(a_+-a_-))_{L^2_{x}}. 
		\end{align*}The first term on the right hand side is zero because $\partial_{x_i}a_\pm=0$ on $\Gamma_i$ from \eqref{Lem25a}. Noticing $\pa_{x_ix_j}\phi=0$ on $\Gamma_j$ and $\Gamma_i$ for $j\neq i$, we have 
		\begin{align*}
			(\partial^\alpha {\phi} ,\pa^\alpha(a_+-a_-))_{L^2_{x}} &= (\partial_{x_ix_ix_i}\phi,\pa_{x_i}\Delta_x\phi)_{L^2_{x}}\\
			&= (\partial_{x_ix_ix_i}\phi,\pa_{x_ix_ix_i}\phi)_{L^2_{x}}-\sum_{j\neq i} (\partial_{x_ix_ix_ix_j}\phi,\pa_{x_ix_j}\phi)_{L^2_{x}}\\
			&= (\partial_{x_ix_ix_i}\phi,\pa_{x_ix_ix_i}\phi)_{L^2_{x}}+\sum_{j\neq i} (\partial_{x_ix_ix_j}\phi,\pa_{x_ix_ix_j}\phi)_{L^2_{x}}
			= \|\pa^\alpha\na_x\phi\|^2_{L^2_x}.
		\end{align*}
		If $\pa^\alpha=\pa_{x_ix_jx_k}$, then we need more discussion. Note that we will apply \eqref{Neumann} and \eqref{4.16} frequently. 
		If $i,j,k$ are pairwise different, then from \eqref{4.4}, we have $\pa_{x_ix_jx_k}\phi=0$ on $\Gamma_i$, $\Gamma_j$ and $\Gamma_k$ and hence  
		\begin{align*}
			(\partial^\alpha {\phi} ,\pa^\alpha(a_+-a_-))_{L^2_{x}}
			= (\partial^\alpha {\phi} ,-\pa^\alpha\Delta_x\phi)_{L^2_{x}} = \|\pa^\alpha\na_x\phi\|^2_{L^2_x}.
		\end{align*}
		If $i=j\neq k$, then by $\pa_{x_ix_k}(a_+-a_-)=0$ on $\Gamma_i$ and $\pa_{x_ix_ix_ix_k}\phi=0$ on $\Gamma_k$, we have 
		\begin{align*}
			(\partial^\alpha {\phi} ,\pa^\alpha(a_+-a_-))_{L^2_{x}}
			&= (\partial_{x_ix_ix_ix_kx_k} {\phi} ,\pa_{x_i}(a_+-a_-))_{L^2_{x}}\\
			&= (\partial_{x_ix_ix_ix_kx_k} {\phi} ,-\pa_{x_ix_ix_i}\phi)_{L^2_{x}}\\ 
			&\qquad+\sum_{m=k} (\partial_{x_ix_ix_ix_kx_k} {\phi} ,-\pa_{x_ix_mx_m}\phi)_{L^2_{x}}\\ 
			&\qquad+\sum_{m\neq k,i} (\partial_{x_ix_ix_ix_kx_k} {\phi} ,-\pa_{x_ix_mx_m}\phi)_{L^2_{x}}\\ 
			&= \|\pa^\alpha\na_x\phi\|^2_{L^2_x}.
		\end{align*}
		The last identity follows from suitable integration by parts. 
		Note that if $m=i$, we have $\pa_{x_ix_ix_ix_k}\phi=0$ on $\Gamma_k$. If $m=k$, we have $\pa_{x_ix_mx_m}\phi=0$ on $\Gamma_i$. If $m\neq k,i$, we have $\partial_{x_ix_ix_ix_k} {\phi}=0$ on $\Gamma_k$, $\pa_{x_ix_mx_mx_k}\phi=0$ on $\Gamma_i$ and $\pa_{x_ix_ix_mx_k}\phi=0$ on $\Gamma_m$. 
		
		\smallskip
		If $i=j=k$, then by $\pa_{x_ix_ix_i}\phi=0$ on $\Gamma_i$ and $\pa_{x_ix_ix_ix_m}\phi=0$ on $\Gamma_m$ for $m\neq i$, we have 
		\begin{align*}
			(\partial^\alpha {\phi} ,\pa^\alpha(a_+-a_-))_{L^2_{x}}
			&= (\pa_{x_ix_ix_ix_i} {\phi} ,-\pa_{x_ix_i}(a_+-a_-))_{L^2_{x}}\\
			&= (\pa_{x_ix_ix_ix_i} {\phi} ,\pa_{x_ix_ix_ix_i}\phi)_{L^2_{x}}\\&\qquad+\sum_{m\neq i}(\pa_{x_ix_ix_ix_i} {\phi} ,\pa_{x_ix_ix_mx_m}\phi)_{L^2_{x}}\\
			&= \|\pa^\alpha\na_x\phi\|^2_{L^2_x}.
		\end{align*} 
		Plugging the above estimates into \eqref{4.55}, we have 
		\begin{align*}
			\sum_\pm S_4 = -10\|\pa^\alpha\na_x\phi\|^2_{L^2_x}.
		\end{align*}
		Similar to the calculation we used to derive \eqref{s5}, we know that $S_5=0$. 
		Combining the above estimates, taking summation $|\alpha|\le K$ and $\pm$ and letting $\eta>0$ small enough, we have 
		\begin{multline}\label{122e}
			\partial_t\sum_{|\alpha|\le K}(\partial^\alpha f,\tilde{\Phi}_a)_{L^2_{x,v}} + \lambda\sum_{|\alpha|\le K}\Big(\|{\partial^\alpha (a_+-a_-)}\|^2_{L^2_{x}}+\|\pa^\alpha\na_x\phi\|^2_{L^2_x}\Big)\\
			\lesssim\sum_{|\alpha|\le K}\Big( C_\eta\|\partial^\alpha\{\I-\P\}{f}\|^2_{L^2_{x}L^2_{\gamma/2}}
			+\|({\partial^\alpha g},\zeta)_{L^2_v}\|^2_{L^2_{x}}+\eta\|\partial^\alpha b\|_{L^2_x}^2\Big).
		\end{multline}

		Now we take the linear combination $\eqref{122e}+\kappa\times\eqref{122a}+\kappa^2\times\eqref{122b}+\kappa^3\times\eqref{122d}$ and let $\kappa,\eta$ sufficiently small, then 
		\begin{equation*}
			\partial_t\E_{int}(t) + \lambda\sum_{|\alpha|\le K}\|\pa^\alpha[{a_+}, a_-,{b},{c}]\|^2_{L^2_{x}}
			\lesssim \sum_{|\alpha|\le K}\Big(\|\partial^\alpha\{\I-\P\}{ f}\|^2_{L^2_{x}L^2_{\gamma/2}}+\|({\partial^\alpha g},\zeta)_{L^2_v}\|^2_{L^2_{x}}\Big)+ \|E\|_{L^2_x}^4,
		\end{equation*}
		where 
		\begin{multline*}
			\E_{int}(t) = \sum_{|\alpha|\le K}\Big((\partial^\alpha f,\tilde{\Phi}_a)_{L^2_{x,v}}+\sum_\pm\big(\kappa(\partial^\alpha f_\pm,\Phi_c)_{L^2_{x,v}} + \kappa^2(\partial^\alpha f_\pm,\Phi_b)_{L^2_{x,v}}\big)\\+\kappa^3(\partial^\alpha f,\Phi_a)_{L^2_{x,v}}\Big).
		\end{multline*}
		Note that $|a_+-a_-|^2+|a_++a_-|^2=2|a_+|^2+2|a_-|^2$. 
		Using \eqref{3.10}, \eqref{3.12}, \eqref{4.28}, \eqref{3.21}, \eqref{3.21c}, \eqref{3.27}, \eqref{4.42}, \eqref{3.33} and \eqref{3.33q}, we know that 
		\begin{align*}
			\E_{int}(t)\lesssim \sum_{|\alpha|\le K}\|\partial^\alpha f\|^2_{L^2_xL^2_v}. 
		\end{align*}This completes the Theorem \ref{Thm41}.
		\end{proof}
	
	Now we estimate $\|({\partial^\alpha g},\zeta)_{L^2_v}\|^2_{L^2_{x}}$ when $g$ is given by \eqref{g}. For $|\alpha|\le K$, by \eqref{35a}, we apply $L^3-L^6$ and $L^\infty-L^2$ H\"{o}lder's inequality to obtain  
	\begin{align*}
		\notag
		\int_{\Omega}|({\partial^\alpha\Gamma(f,f)},\zeta(v))_{L^2_{v}}|^2\,dx
		&\notag\lesssim \int_{\Omega}\sum_{\alpha_1\le\alpha}|{\partial^{\alpha-\alpha_1} f}|^2_{L^2_v}|{\partial^{\alpha_1} f}|^2_{L^2_{D}}\,dx\\
		&\notag\lesssim \sum_{2\le|\alpha_1|\le K}\|{\partial^{\alpha-\alpha_1} f}\|^2_{L^\infty_xL^2_v}\|{\partial^{\alpha_1} f}\|^2_{L^2_xL^2_{D}}\\
		&\notag\qquad+\sum_{|\alpha_1|=1}\|{\partial^{\alpha-\alpha_1} f}\|^2_{L^3_xL^2_v}\|{\partial^{\alpha_1} f}\|^2_{L^6_xL^2_{D}}\\
		&\notag\qquad+\sum_{|\alpha_1|=0}\|{\partial^{\alpha-\alpha_1} f}\|^2_{L^2_xL^2_v}\|{\partial^{\alpha_1} f}\|^2_{L^\infty_xL^2_{D}}\\
		&\lesssim \|{ f}\|^2_{H^K_xL^2_v}\|\na_x{f}\|^2_{H^{K}_xL^2_{D}},
	\end{align*}where we used embedding $\|f\|_{L^3_x(\Omega)}\lesssim\|f\|_{H^1_x(\Omega)}$, $\|f\|_{L^6_x(\Omega)}\lesssim\|\na_xf\|_{L^2_x(\Omega)}$ and $\|f\|_{L^\infty_x(\Omega)}\lesssim\|\na_xf\|_{H^1_x(\Omega)}$ from \cite[Section V and (V.21)]{Adams2003}. 
	Similarly, we have 
	\begin{align*}
		\int_{\Omega}|(\partial^\alpha(\nabla_x\phi\cdot \nabla_vf_\pm),\zeta(v))_{L^2_{v}}|^2\,dx
		&\notag\lesssim \int_{\Omega}\sum_{\alpha_1\le\alpha}|{\partial^{\alpha-\alpha_1} \nabla_x\phi}|^2|{\partial^{\alpha_1} f_\pm}|^2_{L^2_{D}}\,dx\\
		&\notag\lesssim \sum_{2\le|\alpha_1|\le K}\|{\partial^{\alpha-\alpha_1} \nabla_x\phi}\|^2_{L^\infty_x} \|{\partial^{\alpha_1} f}\|^2_{L^2_xL^2_{D}}\\
		&\notag\qquad+\sum_{|\alpha_1|=1}\|{\partial^{\alpha-\alpha_1} \nabla_x\phi}\|^2_{L^3_x}\|{\partial^{\alpha_1} f}\|^2_{L^6_xL^2_{D}}\\
		&\notag\qquad+\sum_{|\alpha_1|=0}\|{\partial^{\alpha-\alpha_1} \nabla_x\phi}\|^2_{L^2_x}\|{\partial^{\alpha_1} f}\|^2_{L^\infty_xL^2_{D}}\\
		&\lesssim \|{\na_x\phi}\|^2_{H^K_x}\|\na_x{f}\|^2_{H^{K}_xL^2_{D}},
	\end{align*}
	and 
	\begin{align*}
		\int_{\Omega}|(\partial^\alpha(\nabla_x\phi\cdot vf_\pm),\zeta(v))_{L^2_{v}}|^2\,dx
		&\notag\lesssim \int_{\Omega}\sum_{\alpha_1\le\alpha}|{\partial^{\alpha-\alpha_1} \nabla_x\phi}|^2|{\partial^{\alpha_1} f_\pm}|^2_{L^2_{D}}\,dx\\
		&\lesssim \|{\na_x\phi}\|^2_{H^K_x}\|\na_x{f}\|^2_{H^{K}_xL^2_{D}}.
	\end{align*}
Note that $\zeta$ has exponential velocity decay.
	Combining the above estimates, we obtain 
	\begin{align}\label{4.58}
		\|({\partial^\alpha g},\zeta)_{L^2_v}\|^2_{L^2_{x}}\lesssim \Big(\|{ f}\|^2_{H^K_xL^2_v}+\|{\na_x\phi}\|^2_{H^{K}_x}\Big)\|\na_x{f}\|^2_{H^{K}_xL^2_{D}}.
	\end{align}


	\section{Global Existence}\label{Sec_Main}
	
	
	In this section, we use weight function $w_{l,\nu}(\al,\beta)$ defined in \eqref{w22} and 
	and define energy functional 
	\begin{align}\label{8.2}
		X(t) = \sup_{0\le \tau\le t}e^{\delta \tau^p}\E(\tau) + \sup_{0\le \tau\le t}\E_\nu(\tau), 
	\end{align}
	where $p$ is given by \eqref{p}, $\E_\nu(t)$ is given by \eqref{E1} and $\E=\E_0$. 
	Let the {\em a priori} assumption to be 
	\begin{align}\label{priori}
		\sup_{0\le t\le T}X(t)\le \delta_0, 
	\end{align}
	for some small $\delta_0>0$. 
	Assuming \eqref{priori}, by Sobolev inequalities, 
	we have $\|\phi\|_{L^\infty_x}\lesssim \|\na_x\phi\|_{H^1}\le \delta_0$. Then 
	\begin{align*}
		|e^{\pm \phi}|\approx 1, 
	\end{align*}
	which will be frequently used later on. Next, we write the useful Lemma for estimate of trilinear terms. 
	
	\begin{Lem}\label{Lem91}
		Let $T>0$ and $l\in\R$. Let $\gamma\ge -3$ for Landau case, $(\gamma,s)\in\{-\frac{3}{2}<\gamma+2s<0 ,\ \frac{1}{2}\le s<1\}\cup \{\gamma+2s\ge 0,\ 0<s<1\}$ for Boltzmann case. 
		If $1\le|\alpha|\le 3$, then  
		\begin{align}\label{7.4a}
			\sum_{\alpha_1<\alpha}\Big|\Big(\pa^{\alpha-\alpha_1}\nabla_x\phi\cdot v\,\pa^{\alpha_1}f_\pm, e^{\pm\phi}w_{2l,2\nu}(2\alpha,0)\pa^\alpha f_\pm\Big)_{L^2_{x,v}}\Big|\lesssim \sqrt{\E_\nu(t)}\D_\nu(t), 
		\end{align}
		where $\E_\nu(t)$ and $\D_\nu(t)$ are given by \eqref{E1} and \eqref{D} respectively.
		If $1\le|\alpha|+|\beta|\le 3$, then 
		\begin{align}
			\label{7.4b}
			\sum_{\alpha_1+\beta_1<\alpha+\beta}\Big|\Big(\pa^{\alpha-\alpha_1}\nabla_x\phi\cdot \pa_{\beta-\beta_1}v\,\pa^{\alpha_1}_{\beta_1} f_\pm, e^{\pm\phi}w_{2l,2\nu}(2\alpha,2\beta)\pa^\alpha_\beta f_\pm\Big)_{L^2_{x,v}}\Big|\lesssim \sqrt{\E_\nu(t)}\D_\nu(t).
		\end{align}
		Moreover, if $|\alpha|\le 3$, we have 
		\begin{align}\label{7.5}
			\big(\pa^\alpha (\nabla_x\phi\cdot \nabla_vf_\pm),e^{\pm\phi}w_{2l,2\nu}(2\alpha,0)\pa^\alpha f_\pm\big)_{L^2_{x,v}}\lesssim \sqrt{\E_\nu(t)}\D_\nu(t)+\|\na_x\phi\|_{H^2_x}\E_\nu(t). 
		\end{align}
		If $|\alpha|+|\beta|\le 3$, we have 
		\begin{align}\label{7.5a}
			\big(\pa^\alpha_\beta (\nabla_x\phi\cdot \nabla_vf_\pm),e^{\pm\phi}w_{2l,2\nu}(2\alpha,2\beta)\pa^\alpha_\beta f_\pm\big)_{L^2_{x,v}}\lesssim \sqrt{\E_\nu(t)}\D_\nu(t)+\|\na_x\phi\|_{H^2_x}\E_\nu(t). 
		\end{align}
	\end{Lem}
	\begin{proof}
		Let $s=1$ in Landau case. Note that $\<v\>w_{l,\nu}(\alpha,0)\lesssim \<v\>^{\gamma+2s}w_{l,\nu}(\alpha-e_i,0)$ for $i=1,2,3$. Then \eqref{7.4a} can be estimated by 
		\begin{align}\notag\label{4.8}
			&\quad\,
			\sum_{\alpha_1<\alpha}\big\|\pa^{\alpha-\alpha_1}\na_x\phi\,w_{l,\nu}(\alpha-e_i,0)\<v\>^{\frac{\gamma+2s}{2}}\pa^{\alpha_1}f_\pm\big\|_{L^2_xL^2_v}\|w_{l,\nu}(\alpha,0)\<v\>^{\frac{\gamma+2s}{2}}\pa^\alpha f_\pm\|_{L^2_xL^2_v}\\
			&\notag\lesssim
			\Big(\sum_{|\alpha_1|= 0} \|\pa^{\alpha-\alpha_1}\na_x\phi\|_{L^2_x}\|w_{l,\nu}(\alpha-e_i,0)\pa^{\alpha_1}f_\pm\|_{L^\infty_xL^2_D} \\
			&\notag\qquad\quad+ \sum_{|\alpha_1|= 1} \|\pa^{\alpha-\alpha_1}\na_x\phi\|_{L^6_x}\|w_{l,\nu}(\alpha-e_i,0)\pa^{\alpha_1}f_\pm\|_{L^3_xL^2_D}\\
			&\notag\qquad\quad+ \sum_{|\alpha_1|= 2} \|\pa^{\alpha-\alpha_1}\na_x\phi\|_{L^\infty_x}\|w_{l,\nu}(\alpha-e_i,0)\pa^{\alpha_1}f_\pm\|_{L^2_xL^2_D}\Big)
			\|w_{l,\nu}(\alpha,0)\pa^\alpha f_\pm\|_{L^2_xL^2_D}\\
			&\notag\lesssim \|E\|_{H^3_x}\sum_{|\alpha_1|\le 2}\|w_{l,\nu}(\alpha_1,0)\pa^{\alpha_1}  f_\pm\|_{L^2_xL^2_D}\|w_{l,\nu}(\alpha,0)\pa^\alpha f_\pm\|_{L^2_xL^2_D}\\&\lesssim \sqrt{\E_\nu(t)}\D_\nu(t).
		\end{align}
		For \eqref{7.4b}, notice that $\pa_{\beta-\beta_1}v\,w_{l,\nu}(\alpha,\beta)\lesssim\min\{ \<v\>^{\gamma+2s}w_{l,\nu}(\alpha-e_i,\beta), \<v\>^{\gamma+2s}w_{l,\nu}(\alpha,\beta-e_i)\}$. Then similar to \eqref{4.8}, \eqref{7.4b} is bounded above by 
		\begin{align*}
			&\quad\,
			\sum_{\alpha_1+\beta_1<\alpha+\beta}\big\|\pa^{\alpha-\alpha_1}\na_x\phi\,w_{l,\nu}(\alpha-e_i,\beta)\<v\>^{\frac{\gamma+2s}{2}}\pa^{\alpha_1}_{\beta_1}f_\pm\big\|_{L^2_xL^2_v}\|w_{l,\nu}(\alpha,\beta)\<v\>^{\frac{\gamma+2s}{2}}\pa^\alpha_\beta f_\pm\|_{L^2_xL^2_v}\\
			&\lesssim \|E\|_{H^3_x}\sum_{|\alpha_1|+|\beta_1|\le 2}\|w_{l,\nu}(\alpha_1,\beta_1)\pa^{\alpha_1}_{\beta_1} f_\pm\|_{L^2_xL^2_D}\|w_{l,\nu}(\alpha,\beta)\pa^\alpha_\beta f_\pm\|_{L^2_xL^2_D}\\&\lesssim \sqrt{\E_\nu(t)}\D_\nu(t).
		\end{align*}
		For \eqref{7.5}, we have 
		\begin{multline*}
			\big|\big(\pa^\alpha (\nabla_x\phi\cdot \nabla_vf_\pm),e^{\pm\phi}w_{2l,2\nu}(2\alpha,0)\pa^\alpha f_\pm\big)_{L^2_{x,v}}\big|
			\\\lesssim \sum_{\alpha_1\le\alpha}\big|\big(\pa^{\alpha-\alpha_1}\nabla_x\phi\cdot \pa^{\alpha_1}\nabla_vf_\pm,e^{\pm\phi}w_{2l,2\nu}(2\alpha,0)\pa^\alpha f_\pm\big)_{L^2_{x,v}}\big|. 
		\end{multline*}
		When $\alpha_1=\alpha$, taking integration by parts with respect to $\na_v$, we have 
		\begin{align*}\notag
			&\quad\,\big|\big(\nabla_x\phi\cdot \pa^{\alpha}\nabla_vf_\pm,e^{\pm\phi}w_{2l,2\nu}(2\alpha,0)\pa^\alpha f_\pm\big)_{L^2_{x,v}}\big|\\
			&\notag= \|\nabla_x\phi\|_{L^\infty_x}\int_{\Omega}\int_{\R^3} |\pa^{\alpha}f_\pm|^2\big|\nabla_v(w_{l,\nu}(\alpha,0))w_{l,\nu}(\alpha,0)\big|\,dvdx\\
			&\notag\lesssim \|\na_x\phi\|_{H^2_x}\|w_{l,\nu}(\alpha,0)\pa^{\alpha}f_\pm\|^2_{L^2_{x}L^2_v}\\
			&\lesssim \|\na_x\phi\|_{H^2_x}\E_\nu(t).
		\end{align*}
		Note that $|\na_vw_{l,\nu}(\alpha,0)|\le w_{l,\nu}(\alpha,0)$.
		
		\smallskip
		When $\alpha_1<\alpha$ in Landau case, we have $w_{l,\nu}(\alpha,0)=\<v\>^{\gamma+1}w_{l,\nu}(\alpha-e_i,0)$. Regarding $\<v\>^{\gamma/2}\na_vf$ as the dissipation term, we have  
		\begin{align*}\notag
			&\quad\,\sum_{\alpha_1<\alpha}\big|\big(\pa^{\alpha-\alpha_1}\nabla_x\phi\cdot \pa^{\alpha_1}\nabla_vf_\pm,e^{\pm\phi}w_{2l,2\nu}(2\alpha,0)\pa^\alpha f_\pm\big)_{L^2_{x,v}}\big|\\
			&\lesssim \sum_{\alpha_1<\alpha}\big\|\big(\pa^{\alpha-\alpha_1}\nabla_x\phi\cdot w_{l,\nu}(\alpha-e_i,0)\<v\>^{\frac{\gamma}{2}}\pa^{\alpha_1}\nabla_vf_\pm\big\|_{L^2_xL^2_v}\|w_{l,\nu}(\alpha,0)\<v\>^{\frac{\gamma+2}{2}}\pa^\alpha f_\pm\|_{L^2_{x,v}}\\
			&\notag\lesssim 
			\Big(
			\sum_{|\alpha-\alpha_1|=3}\|\pa^{\alpha-\alpha_1}\nabla_x\phi\|_{L^2_x} \|\pa^{\alpha_1} f_\pm\|_{L^\infty_xL^2_{D,w_{l,\nu}(\alpha-e_i,0)}}\\&\notag\qquad
			+
			\sum_{|\alpha-\alpha_1|=2}\|\pa^{\alpha-\alpha_1}\nabla_x\phi\|_{L^6_x} \|\pa^{\alpha_1} f_\pm\|_{L^3_xL^2_{D,w_{l,\nu}(\alpha-e_i,0)}}\\&\notag\qquad
			+\sum_{|\alpha-\alpha_1|=1}\|\pa^{\alpha-\alpha_1}\nabla_x\phi\|_{L^\infty_x} \|\pa^{\alpha_1} f_\pm\|_{L^2_xL^2_{D,w_{l,\nu}(\alpha-e_i,0)}}\Big)\|w_{l,\nu}(\alpha,0)\pa^\alpha f_\pm\|_{L^2_{x,v}}\\
			&\notag\lesssim \|E\|_{H^3_x}\sum_{|\alpha_1|\le 2}\|\pa^{\alpha_1} f_\pm\|_{L^2_xL^2_{D,w_{l,\nu}(\alpha_1,0)}}\|w_{l,\nu}(\alpha,0)\pa^\alpha f_\pm\|_{L^2_{x}L^2_v}\\
			&\lesssim \sqrt{\E_\nu(t)}\D_\nu(t). 
		\end{align*}
		When $\alpha_1<\alpha$ in Boltzmann case for hard potential, we regard $\na_v$ as derivative and add it into $\pa^\alpha_\beta$ with $|\alpha|+|\beta|\le 3$. Note that in this case, $|\cdot|_{L^2_v}\lesssim |\cdot|_{L^2_D}$. Then 
		\begin{align*}
			&\quad\,\sum_{\alpha_1<\alpha}\big|\big(\pa^{\alpha-\alpha_1}\nabla_x\phi\cdot \pa^{\alpha_1}\nabla_vf_\pm,e^{\pm\phi}w_{2l,2\nu}(2\alpha,0)\pa^\alpha f_\pm\big)_{L^2_{x,v}}\big|\\
			&\lesssim \sum_{\alpha_1<\alpha}\big\|\big(\pa^{\alpha-\alpha_1}\nabla_x\phi\cdot w_{l,\nu}(\alpha,0)\pa^{\alpha_1}\nabla_vf_\pm\big\|_{L^2_xL^2_v}\|w_{l,\nu}(\alpha,0)\pa^\alpha f_\pm\|_{L^2_{x,v}}\\
			&\notag\lesssim 
			\Big(\sum_{|\alpha-\alpha_1|=3}\|\pa^{\alpha-\alpha_1}\nabla_x\phi\|_{L^2_x} \|w_{l,\nu}(\alpha,0)\pa^{\alpha_1}\na_v f_\pm\|_{L^\infty_xL^2_v}\\&\notag\qquad
			+\sum_{|\alpha-\alpha_1|=2}\|\pa^{\alpha-\alpha_1}\nabla_x\phi\|_{L^6_x} \|w_{l,\nu}(\alpha,0)\pa^{\alpha_1}\na_v f_\pm\|_{L^3_xL^2_v}\\&\notag\qquad
			+\sum_{|\alpha-\alpha_1|=1}\|\pa^{\alpha-\alpha_1}\nabla_x\phi\|_{L^\infty_x} \|w_{l,\nu}(\alpha,0)\pa^{\alpha_1}\na_v f_\pm\|_{L^2_xL^2_v}\Big)\|w_{l,\nu}(\alpha,0)\pa^\alpha f_\pm\|_{L^2_{x,v}}\\
			&\notag\lesssim \|E\|_{H^3_x}\sum_{|\alpha_1|+|\beta_1|\le 3}\|w_{l,\nu}(\alpha_1,\beta_1)\pa^{\alpha_1}_{\beta_1} f_\pm\|_{L^2_xL^2_D}\|w_{l,\nu}(\alpha,0)\pa^\alpha f_\pm\|_{L^2_{x}L^2_D}\\
			&\lesssim \sqrt{\E_\nu(t)}\D_\nu(t). 
		\end{align*}
		When $\alpha_1<\alpha$ in Boltzmann case for soft potential, we consider 
		\begin{align}\label{8.9}
			&\quad\,\sum_{\alpha_1<\alpha}\big|\big(\pa^{\alpha-\alpha_1}\nabla_x\phi\cdot \pa^{\alpha_1}\nabla_vf_\pm,e^{\pm\phi}w_{2l,2\nu}(2\alpha,0)\pa^\alpha f_\pm\big)_{L^2_{x,v}}\big|.
		\end{align}
%
		By Young's inequality, $\<\eta\>\lesssim \<\eta\>^s\<v\>^k + \<\eta\>^{1+s}\<v\>^{-\frac{ks}{1-s}}$ for any $k\in\R$ and hence $\<\eta\>$ is a symbol in $S(\<\eta\>^s\<v\>^k + \<\eta\>^{1+s}\<v\>^{-\frac{ks}{1-s}})$ (cf. \cite{Deng2020a}), where $\eta$ is the Fourier variable of $v$. Then by \cite[Lemma 2.3 and Corollary 2.5]{Deng2020a}, we have 
		\begin{equation}\label{410}
			|f|_{H^1_v}
			\lesssim  |f \<v\>^k|_{H^s} + |f\<v\>^{-ks/(1-s)}|_{H^{1+s}},
		\end{equation}
	for $k\in\R$. 
		From our choice \eqref{pq}, we have 
		\begin{align*}
			0\le rs + (r-q)(1-s)+\gamma, 
		\end{align*}
	and hence, 
	\begin{align*}
		w_{l,\nu}(\al,\beta)\le \<v\>^\gamma w_{l,\nu}(\al-e_i,\beta)^sw_{l,\nu}(\al-e_i,\beta+e_i)^{1-s}. 
	\end{align*}
	Applying \eqref{410} with $\<v\>^k=w_{l,\nu}(\al-e_i,\beta)^{1-s}w_{l,\nu}(\al-e_i,\beta+e_i)^{-(1-s)}$, we have
	\begin{align*}\notag
		\sum_{|\al_1|<|\al|}\|\<v\>^{-\frac{\gamma}{2}}w_{l,\nu}(\al,\beta)\pa^{\al_1}_\beta \na_vf_\pm\|_{L^2_xL^2_v}
		&\lesssim \sum_{\al_1<\al}\|\<v\>^{\frac{\gamma}{2}}w_{l,\nu}(\al-e_i,\beta)\pa^{\al_1}_\beta f_\pm\|_{L^2_xH^s_v}
		\\&\notag\quad+\sum_{\al_1<\al}\|\<v\>^{\frac{\gamma}{2}}w_{l,\nu}(\al-e_i,\beta+e_i)\pa^{\al_1}_\beta f_\pm\|_{L^2_xH^{1+s}_v}\\
		&\lesssim \sqrt{\D_{\nu}(t)}. 
	\end{align*}
		Then the first right-hand term of \eqref{8.9} can be estimated by 
		\begin{align*}
			&\quad\,\sum_{\alpha_1<\alpha}\big\|\<v\>^{\frac{\gamma}{2}}\pa^{\alpha-\alpha_1}\nabla_x\phi\cdot 
			w_{l,\nu}(\alpha-e_i,0)\pa^{\alpha_1}\na_vf_\pm)\|_{L^2_xL^2_v}\|\<v\>^{\frac{\gamma}{2}}w_{l,\nu}(\alpha,0)\pa^\alpha f_\pm\|_{L^2_{x,v}}\big|\\
			&\notag\lesssim 
			\Big(
			\sum_{|\alpha-\alpha_1|=3}\|\pa^{\alpha-\alpha_1}\nabla_x\phi\|_{L^2_x} \|\<v\>^{-\frac{\gamma}{2}}w_{l,\nu}(\al,0)\pa^{\al_1} \na_v f_\pm\|_{L^\infty_xL^2_v}\\
			&\notag\quad
			+
			\sum_{|\alpha-\alpha_1|=2}\|\pa^{\alpha-\alpha_1}\nabla_x\phi\|_{L^6_x} \|\<v\>^{-\frac{\gamma}{2}}w_{l,\nu}(\al,0)\pa^{\al_1} \na_v f_\pm\|_{L^3_xL^2_v}\\
			&\notag\quad
			+\sum_{|\alpha-\alpha_1|=1}\|\pa^{\alpha-\alpha_1}\nabla_x\phi\|_{L^\infty_x} \|\<v\>^{-\frac{\gamma}{2}}w_{l,\nu}(\al,0)\pa^{\al_1} \na_v f_\pm\|_{L^2_xL^2_v}\Big)\\&\qquad\qquad\qquad\qquad\qquad\qquad\qquad\qquad\qquad\times\|\<v\>^{\frac{\gamma}{2}}w_{l,\nu}(\alpha,0)\pa^\alpha f_\pm\|_{L^2_{x,v}}\\
			&\notag\lesssim \|E\|_{H^3_x}\sum_{|\al_1|<|\al|}\|\<v\>^{-\frac{\gamma}{2}}w_{l,\nu}(\al,\beta)\pa^{\al_1}_\beta \na_vf_\pm\|_{L^2_xL^2_v}\|w_{l,\nu}(\alpha,0)\pa^\alpha f_\pm\|_{L^2_{x}L^2_D}\\
			&\lesssim \sqrt{\E_\nu(t)}\D_\nu(t).  
		\end{align*}
The above estimates give \eqref{7.5}. 
		The proof of \eqref{7.5a} is similar to \eqref{7.5} and we omit the details for brevity.

		\end{proof}

	Now we are ready to prove our Main Theorem \ref{Theorem_bounded}. 
	
	\begin{proof}[Proof of Theorem \ref{Theorem_bounded}]
		Let $|\alpha|+|\beta|\le 3$. If $|\alpha|=2$, we restrict $\alpha_i=2$ for some $i=1,2,3$. 
		Then applying $\pa^\alpha_\beta$ to \eqref{1}, we have 
		\begin{multline}\label{71}
			\partial_t(\pa^\alpha_\beta f_\pm) 
			+ \pa_\beta(v\cdot\nabla_x\pa^\alpha f_\pm) 
			\pm \frac{1}{2}\pa^\alpha_\beta (\nabla_x\phi\cdot vf_\pm) 
			\mp \pa^\alpha_\beta (\nabla_x\phi\cdot \nabla_vf_\pm)\\
			\pm \pa^\alpha_\beta(\nabla_x\phi\cdot v\mu^{1/2})
			- \pa_\beta L_\pm \pa^\alpha f 
			= \pa^\alpha_\beta\Gamma_{\pm}(f,f).
		\end{multline}
		
		\medskip \noindent{\bf Step 1. Estimate with Spatial Derivatives.}
		We begin with the following estimate. Taking integration by parts with respect to $\na_x$, we have 
		\begin{align*}\notag
			&\quad\,\big(v\cdot\nabla_x\pa^\alpha f_\pm,e^{\pm\phi}w_{2l,2\nu}(2\alpha,0)\pa^\alpha f_\pm\big)_{L^2_{x,v}}
			\pm \frac{1}{2}\big(\nabla_x\phi\cdot v\pa^{\alpha}f_\pm,e^{\pm\phi}w_{2l,2\nu}(2\alpha,0)\pa^\alpha f_\pm\big)_{L^2_{x,v}}\\
			&\notag= \frac{1}{2}\int_{\pa\Omega}\int_{\R^3}v\cdot n(x)e^{\pm\phi}|w_{l,\nu}(\alpha,0)\pa^\alpha f_\pm(v)|^2\,dvdS(x) \\
			&\notag= \frac{1}{2}\int_{\pa\Omega}\int_{\R^3}R_xv\cdot n(x)e^{\pm\phi}|w_{l,\nu}(\alpha,0)\pa^\alpha f_\pm(R_xv)|^2\,dvdS(x)\\
			&= -\frac{1}{2}\int_{\pa\Omega}\int_{\R^3}v\cdot n(x)e^{\pm\phi}|w_{l,\nu}(\alpha,0)\pa^\alpha f_\pm(v)|^2\,dvdS(x)=0. 
		\end{align*}
		where we apply Lemma \ref{Lem24} and $R_xv\cdot n(x)=-v\cdot n(x)$. This is what $e^{\pm\phi}$ designed for; cf. \cite{Guo2012}. 
		Then letting $|\beta|=0$ in \eqref{71} and taking inner product with $e^{\pm\phi}w_{2l,2\nu}(2\alpha,0)\pa^\alpha f_\pm$ of \eqref{71} over $\Omega\times\R^3$, we have 
		\begin{multline}\label{7.4}
			\frac{1}{2}\partial_t\|w_{l,\nu}(\alpha,0)\pa^\alpha f_\pm\|_{L^2_{x,v}}^2 \mp \frac{1}{2}\big(\pa_t\phi\,w_{l,\nu}(\alpha,0)\pa^\alpha f_\pm,e^{\pm\phi}w_{l,\nu}(\alpha,0)\pa^\alpha f_\pm\phi\big)_{L^2_{x,v}}\\
			\pm \frac{1}{2}\sum_{\alpha_1<\alpha}\Big(\pa^{\alpha-\alpha_1}\nabla_x\phi\cdot v\pa^{\alpha_1}f_\pm,e^{\pm\phi}w_{2l,2\nu}(2\alpha,0)\pa^\alpha f_\pm\Big)_{L^2_{x,v}}\\
			\mp \big(\pa^\alpha (\nabla_x\phi\cdot \nabla_vf_\pm),e^{\pm\phi}w_{2l,2\nu}(2\alpha,0)\pa^\alpha f_\pm\big)_{L^2_{x,v}}
			\pm \big(\pa^\alpha\nabla_x\phi\cdot v\mu^{1/2},e^{\pm\phi}w_{2l,2\nu}(2\alpha,0)\pa^\alpha f_\pm\big)_{L^2_{x,v}}\\
			- \big(L_\pm \pa^\alpha f,w_{2l,2\nu}(2\alpha,0)\pa^\alpha f_\pm\big)_{L^2_{x,v}}
			- \big(L_\pm \pa^\alpha f,(e^{\pm\phi}-1)w_{2l,2\nu}(2\alpha,0)\pa^\alpha f_\pm\big)_{L^2_{x,v}}\\
			= \big(\pa^\alpha\Gamma_{\pm}(f,f),e^{\pm\phi}w_{2l,2\nu}(2\alpha,0)\pa^\alpha f_\pm\big)_{L^2_{x,v}}.
		\end{multline}
		We denote the second to eighth term in \eqref{7.4} to be $I_1$ to $I_7$ and estimate them term by term. 
		For $I_1$, we have 
		\begin{align}\label{I1}
			|I_1|\lesssim \|\pa_t\phi\|_{L^\infty_x}\E_\nu(t). 
		\end{align}
		By Lemma \ref{Lem91}, we know that 
		\begin{align}\label{I2}
			|I_2|+|I_3|\lesssim \sqrt{\E_\nu(t)}\D_\nu(t)+\|\na_x\phi\|_{H^2_x}\E_\nu(t). 
		\end{align} 
		%
		%
		%
		For $I_4$, when $l=p|\alpha|$ and $\nu=0$, we take summation on $\pm$ to obtain 
		\begin{align*}
			\sum_\pm I_4 &= \big(\pa^\alpha\nabla_x\phi\cdot v\mu^{1/2},\pa^\alpha \{\I-\P\}(f_+-f_-)\big)_{L^2_{x,v}}\\
			&= \big(\pa^\alpha\nabla_x\phi,\pa^\alpha G\big)_{L^2_{x}},
		\end{align*}
		where $G$ is defined by \eqref{93a}.
		Similar to the proof of \eqref{Lem25a} and \eqref{38a}, we know that $\pa_{x_ix_i} G_i(x)=G_i(x)=0$ on $x\in\Gamma_i$ for $i=1,2,3$. 
		Also, by \eqref{Neumann} and \eqref{4.16}, we know that $\pa_{x_ix_ix_i}\phi=\pa_{x_i}\phi=0$ on $\Gamma_i$ for $i=1,2,3$. 
		Then by \eqref{98a}$_1$ and integration by parts, we have 
		\begin{align*}
			\big(\pa^\alpha\nabla_x\phi,\pa^\alpha G\big)_{L^2_{x}}
			&= -\big(\pa^\alpha\phi,\pa^\alpha \nabla_xG\big)_{L^2_{x}}
			= \big(\pa^\alpha\phi,\pa^\alpha\pa_t(a_+-a_-)\big)_{L^2_{x}}. 
		\end{align*}
		If $|\alpha|\le 1$, from \eqref{Neumann} we know that $\pa^\alpha\phi=0$ or $\pa^\alpha\pa_{x_j}\phi=0$ on $\Gamma_j$ for $j=1,2,3$. Then 
		\begin{align*}
			\big(\pa^\alpha\phi,\pa^\alpha\pa_t(a_+-a_-)\big)_{L^2_{x}}
			= \big(\pa^\alpha\phi,-\pa^\alpha\pa_t\Delta_x\phi\big)_{L^2_{x}}
			= \big(\pa^\alpha\na_x\phi,\pa^\alpha\pa_t\na_x\phi\big)_{L^2_{x}}
		\end{align*}
		If $|\alpha|=2$, then $\pa^\alpha=\pa_{x_ix_j}$ for some $i,j=1,2,3$. Since $\pa_{x_i}(a_+-a_-)=0$ and $\pa_{x_i}\phi=0$ on $\Gamma_i$, by integration by parts, we know that 
		\begin{align*}
			\big(\pa^\alpha\phi,\pa^\alpha\pa_t(a_+-a_-)\big)_{L^2_{x}}
			&= \big(\pa_{x_ix_ix_j}\phi,-\pa_{x_j}\pa_t(a_+-a_-)\big)_{L^2_{x}}\\
			&= \big(\pa_{x_ix_ix_j}\phi,\pa_{x_jx_ix_i}\pa_t\phi\big)_{L^2_{x}}+\sum_{k\neq i}\big(\pa_{x_ix_ix_j}\phi,\pa_{x_jx_kx_k}\pa_t\phi\big)_{L^2_{x}}\\
			&= \frac{1}{2}\pa_t\|\pa^\alpha \na_x\phi\|_{L^2_x}^2,
		\end{align*}
	where we used the following factsfrom \eqref{Neumann} and \eqref{4.16}. $\pa_{x_ix_jx_k}\phi=0$ or $\pa_{x_jx_k}\phi=0$ on $\Gamma_i$. $\pa_{x_jx_k}\phi=0$ or $\pa_{x_ix_ix_j}\phi=0$ on $\Gamma_k$ for $k\neq i$.
		
		\smallskip
		If $|\alpha|=3$, then $\pa^\alpha=\pa_{x_ix_jx_k}$ and we need more discussion. Note that we will apply \eqref{Neumann} and \eqref{4.16} frequently. 
		If $i,j,k$ are pairwise different, then $\pa_{x_ix_jx_k}\phi=0$ on $\Gamma_i$, $\Gamma_j$ and $\Gamma_k$. Thus, 
		\begin{align*}
			(\partial^\alpha {\phi} ,\pa_t\pa^\alpha(a_+-a_-))_{L^2_{x}}
			= (\partial^\alpha {\phi} ,-\pa_t\pa^\alpha\Delta_x\phi)_{L^2_{x}} = \frac{1}{2}\pa_t\|\pa^\alpha\na_x\phi\|^2_{L^2_x}.
		\end{align*}
		If $i=j\neq k$, then by $\pa_{x_ix_k}(a_+-a_-)=0$ on $\Gamma_i$ and $\pa_{x_ix_ix_ix_k}\phi=0$ on $\Gamma_k$, we have 
		\begin{align*}
			(\partial^\alpha {\phi} ,\pa_t\pa^\alpha(a_+-a_-))_{L^2_{x}}
			&= (\partial_{x_ix_ix_ix_kx_k} {\phi} ,\pa_t\pa_{x_i}(a_+-a_-))_{L^2_{x}}\\
			&= (\partial_{x_ix_ix_ix_kx_k} {\phi} ,-\pa_t\pa_{x_ix_ix_i}\phi)_{L^2_{x}}\\ 
			&\qquad+\sum_{m=k} (\partial_{x_ix_ix_ix_kx_k} {\phi} ,-\pa_t\pa_{x_ix_mx_m}\phi)_{L^2_{x}}\\ 
			&\qquad+\sum_{m\neq k,i} (\partial_{x_ix_ix_ix_kx_k} {\phi} ,-\pa_t\pa_{x_ix_mx_m}\phi)_{L^2_{x}}\\ 
			&= \frac{1}{2}\pa_t\|\pa^\alpha\na_x\phi\|^2_{L^2_x}.
		\end{align*}
		The last identity follows from suitable integration by parts. 
		Note that if $m=i$, we have $\pa_{x_ix_ix_ix_k}\phi=0$ on $\Gamma_k$. If $m=k$, we have $\pa_{x_ix_mx_m}\phi=0$ on $\Gamma_i$. If $m\neq k,i$, we have $\partial_{x_ix_ix_ix_k} {\phi}=0$ on $\Gamma_k$, $\pa_{x_ix_mx_mx_k}\phi=0$ on $\Gamma_i$ and $\pa_{x_ix_ix_mx_k}\phi=0$ on $\Gamma_m$. 
		
		If $i=j=k$, then by $\pa_{x_ix_ix_i}\phi=0$ on $\Gamma_i$ and $\pa_{x_ix_ix_ix_m}\phi=0$ on $\Gamma_m$ for $m\neq i$, we have 
		\begin{align*}
			(\partial^\alpha {\phi} ,\pa_t\pa^\alpha(a_+-a_-))_{L^2_{x}}
			&= (\pa_{x_ix_ix_ix_i} {\phi} ,-\pa_t\pa_{x_ix_i}(a_+-a_-))_{L^2_{x}}\\
			&= (\pa_{x_ix_ix_ix_i} {\phi} ,\pa_t\pa_{x_ix_ix_ix_i}\phi)_{L^2_{x}}\\&\qquad+\sum_{m\neq i}(\pa_{x_ix_ix_ix_i} {\phi} ,\pa_t\pa_{x_ix_ix_mx_m}\phi)_{L^2_{x}}\\
			&= \frac{1}{2}\pa_t\|\pa^\alpha\na_x\phi\|^2_{L^2_x}.
		\end{align*} 
		Combining the above estimate for $I_4$, we have that when $l=p|\alpha|$ and $\nu=0$, 
		\begin{align}\label{I4}
			\sum_\pm I_4 = \frac{1}{2}\pa_t\|\pa^\alpha E\|_{L^2_x}^2. 
		\end{align}	
		When $l\neq p|\alpha|$ or $\nu\neq 0$, we write an upper bound for $I_4$:
		\begin{align}\label{I4a}
			|I_4|\lesssim C_\eta\|\pa^\alpha E\|_{L^2_x}^2 + \eta\|w_{l,\nu}(\alpha,0)\pa^\alpha f_\pm\|_{L^2_xL^2_D}^2. 
		\end{align}
		For $I_5$, when $l=p|\alpha|$ and $\nu=0$, using Lemma \ref{Lem21}, we have 
		\begin{align}\label{I5}
			\sum_\pm I_5 \gtrsim \|\{\I-\P\}\pa^\alpha f\|_{L^2_xL^2_D}^2. 
		\end{align}
		When $l\neq p|\alpha|$ or $\nu\neq 0$, by \eqref{36c}, we have 
		\begin{align}\label{I5a}
			\sum_\pm I_5 \ge c_0|\pa^\alpha f|_{L^2_{D,w_{l,\nu}(\alpha,0)}}^2-C|\pa^\alpha f|_{L^2(B_C)}^2.
		\end{align}
		For $I_6$, note that $|e^{\pm\phi}-1|\lesssim \|\phi\|_{L^\infty_x}\lesssim \|\na_x\phi\|_{H^1_x}$. Then by \eqref{35a}, 
		\begin{align}\label{I6}
			|I_6|\lesssim \|\na_x\phi\|_{H^1_x}\big(\|\pa^\alpha f_\pm\|_{L^2_xL^2_D}\|\pa^\alpha f_\pm\|_{L^2_xL^2_v}+\|\pa^\alpha f_\pm\|_{L^2_xL^2_D}^2\big)\lesssim \sqrt{\E_\nu(t)}\D_\nu(t). 
		\end{align}
		The estimate of $I_7$ can be obtained from \eqref{36a} and it follows that 
		\begin{align}\label{I7}
			|I_7|\lesssim \sqrt{\E_\nu(t)}\D_\nu(t). 
		\end{align}
		Therefore, if $l=p|\alpha|$, $\nu=0$, plugging estimate \eqref{I1}, \eqref{I2}, \eqref{I4}, \eqref{I5}, \eqref{I6} and \eqref{I7} into \eqref{7.4}, we have the energy estimate without weight:
		\begin{multline}\label{alpha1}
			\frac{1}{2}\partial_t\Big(\|\pa^\alpha f_\pm\|_{L^2_{x,v}}^2+\|\pa^\alpha E\|_{L^2_x}^2\Big) 
			+\lambda\|\{\I-\P\}\pa^\alpha f\|_{L^2_xL^2_D}^2
			\\
			\lesssim\big(\|\pa_t\phi\|_{L^\infty_x}+\|\na_x\phi\|_{H^2_x}\big)\E(t)+\sqrt{\E(t)}\D(t),
		\end{multline}	for some constant $\lambda>0$. 
		If $l\neq p|\alpha|$ or $\nu\neq 0$, plugging estimate \eqref{I1}, \eqref{I2}, \eqref{I4a}, \eqref{I5a}, \eqref{I6} and \eqref{I7} into \eqref{7.4} and letting $\eta>0$ small enough, we have the energy estimate with weight:
		\begin{multline}\label{alpha2}
			\frac{1}{2}\partial_t\|w_{l,\nu}(\alpha,0)\pa^\alpha f_\pm\|_{L^2_{x,v}}^2 
			+\lambda\|\pa^\alpha f\|_{L^2_xL^2_{D,w_{l,\nu}(\alpha,0)}}^2\\
			\lesssim\big(\|\pa_t\phi\|_{L^\infty_x}+\|\na_x\phi\|_{H^2_x}\big)\E_\nu(t)+\sqrt{\E_\nu(t)}\D_\nu(t)+\|\pa^\alpha f\|_{L^2_xL^2_{D}}^2+\|\pa^\alpha E\|_{L^2_x}^2,
		\end{multline}	for some constant $\lambda>0$.

		\medskip \noindent{\bf Step 2. Estimate with Mixed Derivatives.}
		Let $1\le|\beta|\le 3$, then $|\alpha|\le 2$. 
		From Lemma \ref{Lem24}, we have $\pa^\alpha f(R_xv)$ equal to $\pa^\alpha f(v)$ or $-\pa^\alpha f(v)$, which implies that $\pa^\alpha_\beta f(R_xv)$ equal to $\pa^\alpha_\beta f(v)$ or $-\pa^\alpha_\beta f(v)$, since $R_xv$ maps $v_i$ to $-v_i$ for some $i=1,2,3$ and derivatives on velocity variable would produce only sign $\pm$. Then we have 
		\begin{align*}
			&\quad\,	\big(v\cdot\nabla_x\pa^\alpha_{\beta} f_\pm,e^{\pm\phi}w_{2l,2\nu}(2\alpha,2\beta)\pa^\alpha_\beta f_\pm\big)_{L^2_{x,v}}
			\pm \frac{1}{2}\big(\nabla_x\phi\cdot v\pa^{\alpha}_{\beta} f_\pm,e^{\pm\phi}w_{2l,2\nu}(2\alpha,2\beta)\pa^\alpha_\beta f_\pm\big)_{L^2_{x,v}}\\
			&= \frac{1}{2}\int_{\pa\Omega}\int_{\R^3}v\cdot n(x)|w_{l,\nu}(\alpha,\beta)\pa^\alpha_\beta f_\pm(v)|^2\,dvdS(x)\\
			&\notag= \frac{1}{2}\int_{\pa\Omega}\int_{\R^3}R_xv\cdot n(x)|w_{l,\nu}(\alpha,\beta)\pa^\alpha_\beta f_\pm(R_xv)|^2\,dvdS(x)\\
			&= -\frac{1}{2}\int_{\pa\Omega}\int_{\R^3}v\cdot n(x)|w_{l,\nu}(\alpha,\beta)\pa^\alpha_\beta f_\pm(v)|^2\,dvdS(x)=0.
		\end{align*}	
		Taking inner product with $w_{2l,2\nu}(2\alpha,2\beta)\pa^\alpha_\beta f_\pm$ of \eqref{71} over $\Omega\times\R^3$, we have 
		\begin{multline}\label{7.13}
			\frac{1}{2}\partial_t\|w_{l,\nu}(\alpha,\beta)\pa^\alpha_\beta f_\pm\|_{L^2_{x,v}}^2 
			\mp\big(\pa_t\phi\, \pa^\alpha_\beta f_\pm,e^{\pm\phi}w_{2l,2\nu}(2\alpha,2\beta)\pa^\alpha_\beta f_\pm\big)_{L^2_{x,v}}\\
			+	\sum_{|\beta_1|=1}\big(\pa_{\beta_1} v\cdot\nabla_x\pa^\alpha_{\beta-\beta_1} f_\pm,e^{\pm\phi}w_{2l,2\nu}(2\alpha,2\beta)\pa^\alpha_\beta f_\pm\big)_{L^2_{x,v}}\\
			\pm \frac{1}{2}\sum_{\alpha_1+\beta_1<\alpha+\beta}\Big(\pa^{\alpha}_\beta(\nabla_x\phi\cdot vf_\pm),e^{\pm\phi}w_{2l,2\nu}(2\alpha,2\beta)\pa^\alpha_\beta f_\pm\Big)_{L^2_{x,v}}\\
			\mp \big(\pa^\alpha_\beta (\nabla_x\phi\cdot \nabla_vf_\pm),e^{\pm\phi}w_{2l,2\nu}(2\alpha,2\beta)\pa^\alpha_\beta f_\pm\big)_{L^2_{x,v}}\\
			\pm \big(\pa^\alpha_\beta(\nabla_x\phi\cdot v\mu^{1/2}),e^{\pm\phi}w_{2l,2\nu}(2\alpha,2\beta)\pa^\alpha_\beta f_\pm\big)_{L^2_{x,v}}\\
			- \big(\pa_\beta L_\pm \pa^\alpha f,w_{2l,2\nu}(2\alpha,2\beta)\pa^\alpha_\beta f_\pm\big)_{L^2_{x,v}}\\
			- \big(\pa_\beta L_\pm \pa^\alpha f,(e^{\pm\phi}-1)w_{2l,2\nu}(2\alpha,2\beta)\pa^\alpha_\beta f_\pm\big)_{L^2_{x,v}}\\
			= \big(\pa^\alpha_\beta\Gamma_{\pm}(f,f),e^{\pm\phi}w_{2l,2\nu}(2\alpha,2\beta)\pa^\alpha_\beta f_\pm\big)_{L^2_{x,v}}.
		\end{multline}
		Denote the second to ninth terms of \eqref{7.13} to be $J_1$ to $J_8$. Then for $J_1$, we have 
		\begin{align*}
			|J_1|\lesssim \|\pa_t\phi\|_{L^\infty_x}\E_\nu(t). 
		\end{align*}
		%
		For $J_2$, note that for hard potential case, we have $|\cdot|_{L^2_v}\lesssim |\cdot|_{L^2_D}$ and $w_{l,\nu}(\alpha,\beta)=w_{l,\nu}(\alpha+e_i,\beta-\beta_1)$. For soft potential case, we have $w_{l,\nu}(\alpha,\beta)\le\<v\>^{\gamma}w_{l,\nu}(\alpha+e_i,\beta-\beta_1)$ for any $|\beta_1|=1$, where we let $s=1$ in Landau case. Then
		\begin{align*}
			|J_2|&=\sum_{|\beta_1|=1}\big|\big(\pa_{\beta_1}v\cdot\nabla_x\pa^\alpha_{\beta-\beta_1}f_\pm,e^{\pm\phi}w_{2l,2\nu}(2\alpha,2\beta)\pa^\alpha_\beta f_\pm\big)_{L^2_{x,v}}\big|\\
			&\lesssim \sum_{|\beta_1|=1} \int_{\Omega\times\R^3}\<v\>^{\gamma+2s}w_{l,\nu}(\alpha+e_i,\beta-\beta_1)|\na_x\pa^\alpha_{\beta-\beta_1}f_\pm|w_{l,\nu}(\alpha,\beta)|\pa^\alpha_\beta f_\pm|\,dxdv\\
			&\lesssim C_\eta\sum_{i=1}^3\sum_{|\beta_1|=1}\|\pa^{\alpha+e_i}_{\beta-\beta_1}f_\pm\|_{L^2_xL^2_{D,w_{l,\nu}(\alpha+e_i,\beta-\beta_1)}}^2+\eta\,\|\pa^\alpha_\beta f_\pm\|_{L^2_xL^2_{D,w_{l,\nu}(\alpha,\beta)}}^2. 
		\end{align*}
		By Lemma \ref{Lem91}, we have 
		\begin{align*}
			|J_3|+|J_4|&\lesssim \sqrt{\E_\nu(t)}\D_\nu(t) +\|\na_x\phi\|_{H^2_x}\E_\nu(t). 
		\end{align*}
		%
		For $J_5$, we write an upper bound:
		\begin{align*}
			|J_5|\lesssim C_\eta\|\pa^\alpha E\|_{L^2_x}^2 + \eta\|w_{l,\nu}(\alpha,\beta)\pa^\alpha_\beta f_\pm\|_{L^2_xL^2_D}^2. 
		\end{align*}
		For $J_6$, by \eqref{36cw}, we have 		
		\begin{align*}
			\sum_\pm J_6 &\ge c_0\|\pa_\beta^\alpha  f\|_{L^2_xL^2_{D,w_{l,\nu}(\alpha,\beta)}}^2
			- C\sum_{|\beta_1|<|\beta|}\|\pa^\alpha_{\beta_1} f\|_{L^2_xL^2_{D,w_{l,\nu}(\alpha,\beta_1)}}^2 -C\|\pa^\alpha f\|_{L^2_xL^2(B_C)}^2.
		\end{align*}
		The estimate for $J_7$ is similar to \eqref{I6} and we have 
		\begin{align*}
			|J_7|\lesssim \sqrt{\E_\nu(t)}\D_\nu(t).
		\end{align*}
		For $J_8$, by \eqref{37a}, we have 
		\begin{align*}
			|J_8|\lesssim \sqrt{\E_\nu(t)}\D_\nu(t). 
		\end{align*}	
		Combining the above estimates on $J_k$ $(1\le k\le 8)$ and letting $\eta>0$ small enough, we have from \eqref{7.13} that 
		\begin{multline}\label{beta}
			\,\frac{1}{2}\partial_t\|w_{l,\nu}(\alpha,\beta)\pa^\alpha_\beta f_\pm\|_{L^2_{x,v}}^2 
			+\lambda\|\pa_\beta^\alpha  h\|_{L^2_xL^2_{D,w_{l,\nu}(\alpha,\beta)}}^2
			\lesssim \|\pa_t\phi\|_{L^\infty_x}\E_\nu(t)+\sqrt{\E_\nu(t)}\D_\nu(t)\\
			+ C\sum_{|\beta_1|<|\beta|}\|\pa^\alpha_{\beta_1} h\|_{L^2_xL^2_{D,w_{l,\nu}(\alpha,\beta_1)}}^2
			+\sum_{i=1}^3\sum_{|\beta_1|=1}\|\pa^{\alpha+e_i}_{\beta-\beta_1}f_\pm\|_{L^2_xL^2_{D,w_{l,\nu}(\alpha+e_i,\beta-\beta_1)}}^2+\|\pa^\alpha E\|^2_{L^2_x}.
		\end{multline}	
		for some $\lambda>0$. 
		
		\medskip\noindent{\bf Step 3. Energy Estimate.}
		From Theorem \ref{Thm41} with $K=3$, there exists $\E_{int}(t)$ satisfying 
		\begin{align}\label{7.19}
			\E_{int}(t)\lesssim \sum_{|\alpha|\le 3} \|\partial^\alpha f\|_{L^2_{x}},
		\end{align} such that 
		\begin{multline}\label{7.20}
			\partial_t\E_{int}(t) + \lambda \sum_{|\alpha|\le 3} \|\partial^\alpha[{a_+},{a_-},{b},{c}]\|^2_{L^2_{x}}+\lambda \sum_{|\alpha|\le 3}\|{\partial^\alpha E}\|^2_{L^2_{x}}\\
			\lesssim \sum_{|\alpha|\le 3}\|\{\I-\P\}{\partial^\alpha f}\|^2_{L^2_{x}L^2_D} + \sum_{|\alpha|\le 3}\|(\partial^\alpha g,\zeta(v))_{L^2_v}\|_{L^2_x}^2+ \|E\|_{L^2_x}^4,
		\end{multline}for some constant $\lambda>0$. 
		Also, by \eqref{4.58}, we know that 
		\begin{align}\label{7.21}
			\|({\partial^\alpha g},\zeta)_{L^2_v}\|^2_{L^2_{x}}+ \|E\|_{L^2_x}^4
			&\lesssim \E_\nu(t)\D_\nu(t). 
		\end{align}
		Taking linear combination $\sum_{|\alpha|\le 3}\eqref{alpha1}+\kappa\times\eqref{7.20}$ with $\kappa>0$ small enough and applying \eqref{7.21}, we have 
		\begin{multline}\label{7.22}
			\frac{1}{2}\partial_t\sum_{|\alpha|\le 3}\Big(\|\pa^\alpha f_\pm\|_{L^2_{x,v}}^2+\|\pa^\alpha E\|_{L^2_x}^2\Big)+\kappa\pa_t\E_{int}(t) 
			+\lambda\sum_{|\alpha|\le 3}\|\pa^\alpha f\|_{L^2_xL^2_D}^2
			+\lambda \sum_{|\alpha|\le 3}\|{\partial^\alpha E}\|^2_{L^2_{x}}\\
			\lesssim\big(\|\pa_t\phi\|_{L^\infty_x}+\|\na_x\phi\|_{H^2_x}\big)\E(t)+(\sqrt{\E_\nu(t)}+\E_\nu(t))\D_\nu(t).
		\end{multline}
		Then taking linear combination $\eqref{7.22}+\kappa\sum_{|\alpha|\le 3}\eqref{alpha2}+\kappa^2\sum_{1\le|\beta|\le 3, \,|\alpha|\le 3-|\beta|}\kappa_{|\beta|}\times\eqref{beta}$ with $0<\kappa_3\ll\kappa_2\ll\kappa_1\ll\kappa$ and $\delta$ small enough, we have 
		\begin{align}\label{8.24}
			\pa_t\E_\nu(t) + \lambda\D_\nu(t) \lesssim \big(\|\pa_t\phi\|_{L^\infty_x}+\|\na_x\phi\|_{H^2_x}\big)\E_\nu(t)+ (\sqrt{\E_\nu(t)}+\E_\nu(t))\D_\nu(t),
		\end{align}
		for any $\nu\ge 0$, where $\E_\nu(t)$ is given by 
		\begin{align*}
			\E_\nu(t) &= \frac{1}{2}\sum_{|\alpha|\le 3}\Big(\|\pa^\alpha f_\pm\|_{L^2_{x,v}}^2+\|\pa^\alpha E\|_{L^2_x}^2\Big)+\kappa\E_{int}(t)\\
			&\qquad+\frac{\kappa}{2}\sum_{|\alpha|\le 3}\Big(\|w_{l,\nu}(\alpha,0)\pa^\alpha f_\pm\|_{L^2_{x,v}}^2+\|\pa^\alpha E\|_{L^2_x}^2\Big)\\
			&\qquad+\frac{\kappa^2}{2}\sum_{1\le|\beta|\le 3, \,|\alpha|\le 3-|\beta|}\kappa_{|\beta|}\|w_{l,\nu}(\alpha,\beta)\pa^\alpha_\beta f_\pm\|_{L^2_{x,v}}^2,
		\end{align*}
		and $\D_\nu(t)$ is given by \eqref{D}. It's direct to check that $\E_\nu(t)$ satisfies \eqref{E1} by using \eqref{7.19}. 
		Thus, using the {\it a priori} assumption \eqref{priori}, \eqref{8.24} becomes 
		\begin{align}\label{8.33}
			\pa_t\E_\nu(t) + \lambda\D_\nu(t) \le \big(\|\pa_t\phi\|_{L^\infty_x}+\|\na_x\phi\|_{H^2_x}\big)\E_\nu(t). 
		\end{align}
		
		For the hard potential case, we have $\|w_{l,\nu}(\alpha,\beta)\pa^\alpha_\beta f_\pm\|^2_{L^2_{x,v}}\lesssim\|w_{l,\nu}(\alpha,\beta)\pa^\alpha_\beta f_\pm\|^2_{L^2_{x}L^2_D}$ and hence $\E(t)\lesssim \D(t)$. 
		Notice from \eqref{Neumann} that $\pa_{x_i}\phi=0$ on $\Gamma_i$. Then by Sobolev embedding \cite[Theorem 6.7-5]{Ciarlet2013}, we have 
		$\|\pa_t\pa_{x_i}\phi\|_{L^2_x}\lesssim \|\pa_t\na_x\pa_{x_i}\phi\|_{L^2_x}$
		and hence, 
		\begin{multline}\label{8.34}
			\|\pa_t\phi\|_{L^\infty_x}\lesssim \|\pa_t\na_x\phi\|_{H^1_x}\lesssim \|\pa_t\na_x^2\phi\|_{L^2_x} = \|\pa_t\Delta_x\phi\|_{L^2_x} \\=\|\pa_t(a_+-a_-)\|_{L^2_x} \lesssim \|\na_xG\|_{L^2_x}\lesssim\|\na_x\{\I-\P\}f\|_{L^2_xL^2_D}\lesssim \sqrt{\E(t)}. 
		\end{multline}
		Here, the first inequality follows from Sobolev inequality; cf. \cite{Adams2003}. The first identity follows from boundary value $\pa_{x_ix_j}\phi=0$ on $\Gamma_i$ and $\Gamma_j$ for $j\neq i$. The second identity is from \eqref{1}$_2$. The third inequality comes from \eqref{98a}$_1$. 
		Thus, when $\nu=0$, \eqref{8.33} becomes 
		\begin{align*}
			\pa_t\E(t) + \lambda\E(t) \le \sqrt{\E(t)}\E(t). 
		\end{align*}
		Under the smallness \eqref{small}, we have 
		\begin{align}\label{8.25a}
			\pa_t\E(t) + \delta\E(t) \le 0,
		\end{align}
		and hence
		\begin{align}\label{8.25b}
			\E(t)\le e^{\delta t}\E(0),
		\end{align}
		for some constant $\delta>0$. This close the {\it a priori} assumption \eqref{priori} by choosing $\ve_0$ in \eqref{small} small enough for hard potential case.

		For soft potential, we need more calculations. Recall definition \eqref{8.2} for $X(t)$ and assume \eqref{priori}. 
		Then from \eqref{8.34} and \eqref{priori}, we have $e^{\delta t^p/2}\big(\|\pa_t\phi\|_{L^\infty_x}+\|\na_x\phi\|_{H^2_x}\big)\lesssim \sqrt{X(t)}\lesssim \sqrt{\delta_0}$, for some $\delta>0$. 
		Solving \eqref{8.33}, we have
		\begin{align}\label{8.26}
			\E_\nu(t)\lesssim \E_\nu(0)e^{-\int^t_0\big(\|\pa_t\phi\|_{L^\infty_x}+\|\na_x\phi\|_{H^2_x}\big)\,d\tau}\lesssim \E_\nu(0)\lesssim \ve_0. 
		\end{align}
		Next we claim that for $T>0$, 
		\begin{align}
			\label{8.26a}
			\sup_{0\le t\le T}e^{\delta t^p}\E(T)\lesssim \ve_0+X^{3/2}(t).
		\end{align}
		Indeed, as in \cite{Strain2007, Duan2013}, for $p'>0$ to be chosen depending on $p$, we define 
		\begin{equation*}
			\mathbf{E} = \{\<v\>\le  t^{p'}\}, \quad 	\mathbf{E}^c = \{\<v\>>  t^{p'}\}.
		\end{equation*}
		Corresponding to this splitting, we define $\E^{low}(t)$ to be the restriction of $\E(t)$ to $\mathbf{E}$ and similarly $\E^{high}(t)$ to be the restriction of $\E(t)$ to $\mathbf{E}^c$. 
		We define $p' = \frac{p-1}{\gamma+2s}$ for Boltzmann case and $p'=\frac{p-1}{\gamma+2}$ for Landau case. Then on $\mathbf{E}$, we have 
		$	t^{p-1} \le \<v\>^{\frac{p-1}{p'}}$, 
		and hence $t^{p-1}\E^{low}(t)\lesssim \D(t)$. It follows from \eqref{8.33} with $\nu=0$ that 
		\begin{equation*}
			\pa_t\E(t) + \lambda t^{p-1}\E(t) \lesssim\big(\|\pa_t\phi\|_{L^\infty_x}+\|\na_x\phi\|_{H^2_x}\big)\E(t) + \lambda t^{p-1}\E^{high}(t). 
		\end{equation*}
		By solving this ODE, we have 
		\begin{equation}\label{8.36}
			\E(t) \lesssim e^{\lambda t^p}\E(0)+\int^t_0e^{-\lambda (t^p-\tau^p)}\big(\big(\|\pa_\tau\phi(\tau)\|_{L^\infty_x}+\|\na_x\phi(\tau)\|_{H^2_x}\big)\E(\tau) + \lambda t^{p-1}\E^{high}(\tau)\big)\,d\tau. 
		\end{equation}
		In what follows we estimate the terms in the time integral on the right-hand side of \eqref{8.36}. 
		Firstly, by \eqref{8.34} and \eqref{8.2}, we have 
		\begin{equation*}
			\big(\|\pa_t\phi\|_{L^\infty_x}+\|\na_x\phi\|_{H^2_x}\big)\E(t)\lesssim \E^{\frac{3}{2}}(t)\lesssim e^{-\frac{3}{2}\delta t^p}X^{\frac{3}{2}}(t). 
		\end{equation*}
		On the other hand, choose $p = p'\vt$, i.e. choose $p$ satisfying \eqref{p}. Then on $\mathbf{E}^c$, we have 
		\begin{equation*}
			 e^{-{\nu\<v\>}} \le e^{-{\nu t^{p'\vt}}} = e^{-{\nu t^{p}}}, 
		\end{equation*}
	Recalling the exponential weight in \eqref{w22}, by \eqref{8.26}, we have 
		\begin{equation*}
			\E^{high}(t)\lesssim e^{-{\nu t^{p}}}\E_{\nu}(t)\lesssim e^{-{\nu t^{p}}}\ve_0.
		\end{equation*}
		Plugging the above estimates into \eqref{8.36}, we have 
		\begin{equation*}
			\E(t)\lesssim e^{\lambda t^p}\E(0) + \int^t_0e^{-\lambda (t^p-\tau^p)}\big(e^{-\frac{3}{2}\delta \tau^p}X^{\frac{3}{2}}(\tau)+\ve_0e^{-{\nu \tau^{p}}}\big)\,d\tau\lesssim e^{\delta t^p}(\ve_0+X^{\frac{3}{2}}(t)), 
		\end{equation*}
		by choosing $\lambda< \nu$. 
		This completes the claim \eqref{8.26a}. Recalling definition \eqref{8.2} for $X(t)$ and using \eqref{8.26} and \eqref{8.26a}, we have 
		\begin{equation*}
			X(t) \lesssim \ve_0+X^{\frac{3}{2}}(t). 
		\end{equation*}
		Then choosing $\delta_0$ in \eqref{priori} sufficiently small, we have the {\em a priori} estimate:
		\begin{equation}\label{8.25c}
			X(t) \lesssim \ve_0. 
		\end{equation}

		With \eqref{8.25a}, \eqref{8.25b} and \eqref{8.25c} in hand, under the smallness of \eqref{small}, it's now standard to apply the continuity argument with local existence from Section \ref{Sec_loc_bounded} to obtain the global existence, uniqueness and large time decay for initial boundary problem \eqref{1}, \eqref{specular} and  \eqref{Neumann} in bounded domain $\Omega$. The positivity of the solutions can be obtained from \cite[Lemma 12, page 800]{Guo2012} for VPL systems and \cite[page 1121]{Guo2002} for VPB systems.
		This complete the proof of Theorem \ref{Theorem_bounded}.

		\end{proof}

	\section{Local Existence}\label{Sec_loc_bounded}
	
	In this section, we are concerned with the local-in-time existence of solutions to problem \eqref{1} in union of cubes. For brevity, we only consider the proof of Vlasov-Poisson-Landau systems when $\gamma\ge -3$, since the Vlasov-Poisson-Boltzmann case is similar. 
	\begin{Thm}\label{blocalsolution}
		Let $\gamma\ge -3$, $\Omega$ be given by \eqref{Omega} and $w_{l,\nu}(\alpha,\beta)$ be given by \eqref{w22}. Then there exists $\ve_0>0$, $T_0>0$ such that if $F_0(x,v)=\mu+\mu^{1/2}f_0(x,v)\ge 0$ and 
		\begin{align*}
			\sum_{|\alpha|+|\beta|\le 3}\big(\|{w_{l,\nu}(\alpha,\beta)\partial^\alpha_\beta f_0}\|^2_{L^2_{x,v}}+\|{\partial^\alpha E_0}\|^2_{L^2_{x}}\big)\le\ve_0,
		\end{align*}
		then the specular reflection boundary problem for VPL systems \eqref{1},  \eqref{specular} and \eqref{Neumann} admits a unique solution $f(t,x,v)$ on $t\in[0,T_0]$, $x\in\Omega$, $v\in\R^3$, satisfying the uniform estimate 
		\begin{align}
			\label{b169}\sup_{0\le t\le T_0}\E_\nu(t)+\int^{T_0}_0\D_\nu(t)\,dt\lesssim  \sum_{|\alpha|+|\beta|\le 3}\big(\|{w_{l,\nu}(\alpha,\beta)\partial^\alpha f_0}\|^2_{L^2_{x,v}}+\|{\partial^\alpha E_0}\|^2_{L^2_{x}}\big),
		\end{align}
		where $\E_\nu(t)$, $\D_\nu(t)$ are defined by \eqref{E1} and \eqref{D} respectively. 
	\end{Thm}
	
	We begin with the following linear inhomogeneous problem on the union of cubes:
	\begin{equation}\label{b154}\left\{
		\begin{aligned}
			&\partial_tf_\pm + v\cdot \nabla_xf_\pm  \pm \frac{1}{2}\nabla_x\psi\cdot vf_\pm  \mp\nabla_x\psi\cdot\nabla_vf_\pm \pm \nabla_x\phi\cdot v\mu^{1/2} - A_\pm f \\
			&\qquad\qquad\qquad\qquad\qquad\qquad\qquad\qquad\qquad= \Gamma_{\pm}(g,h)+Kh,\\
			&-\Delta_x\phi = \int_{\R^3}(f_+-f_-)\mu^{1/2}\,dv,\\ 
			&f(0,x,v) = f_0(x,v),\quad E(0,x)=E_0(x),\\ 
			&{f}(t,x,R_xv) = {f}(t,x,v),\text{ on }\gamma_-,\\
			&\partial_{n}\phi = 0,\ \text{ on } x\in\pa\Omega, 
		\end{aligned}\right.
	\end{equation}
	for a given $h=h(t,x,v)$ and $\psi=\psi(t,x)$.

	\begin{Lem}\label{bLem82}
		Let the same assumption in \eqref{blocalsolution} be satisfied. 
		There exists $\varepsilon_0>0$, $T_0>0$ such that if 
		\begin{multline}\label{b156a}
			\sum_{|\alpha|+|\beta|\le 3}\Big\{\|{w_{l,\nu}(\alpha,\beta)\partial^\alpha_\beta f_0}\|_{L^2_{x,v}}+\|{\partial^\alpha_\beta h}\|_{L^2_{T_0}L^2_{x}L^2_{D,w_{l,\nu}(\alpha,\beta)}}\\+\|{w_{l,\nu}(\alpha,\beta)\partial^\alpha_\beta h}\|_{L^\infty_{T_0}L^2_{x,v}}+\|{\partial^\alpha\nabla_x\psi}\|_{L^\infty_{T_0}L^2_{x}}\Big\}+\|\pa_t{\psi}\|_{L^\infty_{T_0}L^\infty_{x}}\le \varepsilon_0,
		\end{multline}
		then the initial boundary value problem \eqref{b154} admits a unique solution $f=f(t,x,v)$ on $\Omega\times\R^3$ satisfying 
		\begin{multline}\label{b158a}
			\sup_{0\le t\le T_0}\E_\nu(t)+\int^{T_0}_0\D_\nu(t)\,dt+\|\pa_t{\phi}\|_{L^\infty_{T_0}L^\infty_{x}}\lesssim \sum_{|\alpha|+|\beta|\le  3}\big(\|{w_{l,\nu}(\alpha,\beta)\partial^\alpha_\beta f_0}\|^2_{L^2_{x,v}}+\|{\partial^\alpha E_0}\|^2_{L^2_{x}}\big)\\+T_0^{1/2}\sum_{|\alpha|+|\beta|\le 3}\|{w_{l,\nu}(\alpha,\beta)\partial^\alpha_\beta h}\|^2_{L^\infty_{T_0}L^2_{x,v}}, 
		\end{multline}
		where $\E_\nu(t)$ and $\D_\nu(t)$ are defined by \eqref{E1} and \eqref{D} respectively.
	\end{Lem}
	\begin{proof}

		We consider equation \eqref{b154}$_1$ with initial data $(f_0,E_0)$. 
		Similar to \eqref{7.4}, applying $\pa^\alpha$ to \eqref{b154}$_1$ and taking inner product of the resultant equation with $w_{2l,2\nu}(2\alpha,0)e^{\pm\psi}\pa^\alpha f_\pm$, we have 
		\begin{multline*}
			\frac{1}{2}\partial_t\|w_{l,\nu}(\alpha,0)\pa^\alpha f_\pm\|_{L^2_{x,v}}^2 \mp \frac{1}{2}\big(\pa_t\psi\,w_{l,\nu}(\alpha,0)\pa^\alpha f_\pm,e^{\pm\psi}w_{l,\nu}(\alpha,0)\pa^\alpha f_\pm\psi\big)_{L^2_{x,v}}\\
			\pm \frac{1}{2}\sum_{\alpha_1<\alpha}\Big(\pa^{\alpha-\alpha_1}\nabla_x\psi\cdot v\pa^{\alpha_1}f_\pm,e^{\pm\psi}w_{2l,2\nu}(2\alpha,0)\pa^\alpha f_\pm\Big)_{L^2_{x,v}}\\
			\mp \big(\pa^\alpha (\nabla_x\psi\cdot \nabla_vf_\pm),e^{\pm\psi}w_{2l,2\nu}(2\alpha,0)\pa^\alpha f_\pm\big)_{L^2_{x,v}}
			\pm \big(\pa^\alpha\nabla_x\psi\cdot v\mu^{1/2},e^{\pm\psi}w_{2l,2\nu}(2\alpha,0)\pa^\alpha f_\pm\big)_{L^2_{x,v}}\\
			- \big(L_\pm \pa^\alpha f,w_{2l,2\nu}(2\alpha,0)\pa^\alpha f_\pm\big)_{L^2_{x,v}}
			- \big(L_\pm \pa^\alpha f,(e^{\pm\psi}-1)w_{2l,2\nu}(2\alpha,0)\pa^\alpha f_\pm\big)_{L^2_{x,v}}\\
			= \big(\pa^\alpha\Gamma_{\pm}(f,f),e^{\pm\psi}w_{2l,2\nu}(2\alpha,0)\pa^\alpha f_\pm\big)_{L^2_{x,v}}+\big(\pa^\alpha h,e^{\pm\psi}w_{2l,2\nu}(2\alpha,0)\pa^\alpha f_\pm\big)_{L^2_{x,v}}.
		\end{multline*}
		Similar to \eqref{7.13}, applying $\pa^\alpha_\beta$ to \eqref{b154}$_1$ and taking inner product of the resultant equation with $w_{2l,2\nu}(2\alpha,2\beta)e^{\pm\psi}\pa^\alpha_\beta f_\pm$, we have 
		\begin{multline*}
			\frac{1}{2}\partial_t\|w_{l,\nu}(\alpha,\beta)\pa^\alpha_\beta f_\pm\|_{L^2_{x,v}}^2 
			\mp\big(\pa_t\psi\, \pa^\alpha_\beta f_\pm,e^{\pm\psi}w_{2l,2\nu}(2\alpha,2\beta)\pa^\alpha_\beta f_\pm\big)_{L^2_{x,v}}\\
			+	\sum_{|\beta_1|=1}\big(\pa_{\beta_1} v\cdot\nabla_x\pa^\alpha_{\beta-\beta_1} f_\pm,e^{\pm\psi}w_{2l,2\nu}(2\alpha,2\beta)\pa^\alpha_\beta f_\pm\big)_{L^2_{x,v}}\\
			\pm \frac{1}{2}\sum_{\alpha_1+\beta_1<\alpha+\beta}\Big(\pa^{\alpha}_\beta(\nabla_x\psi\cdot vf_\pm),e^{\pm\psi}w_{2l,2\nu}(2\alpha,2\beta)\pa^\alpha_\beta f_\pm\Big)_{L^2_{x,v}}\\
			\mp \big(\pa^\alpha_\beta (\nabla_x\psi\cdot \nabla_vf_\pm),e^{\pm\psi}w_{2l,2\nu}(2\alpha,2\beta)\pa^\alpha_\beta f_\pm\big)_{L^2_{x,v}}\\
			\pm \big(\pa^\alpha_\beta(\nabla_x\psi\cdot v\mu^{1/2}),e^{\pm\psi}w_{2l,2\nu}(2\alpha,2\beta)\pa^\alpha_\beta f_\pm\big)_{L^2_{x,v}}\\
			- \big(\pa_\beta L_\pm \pa^\alpha f,w_{2l,2\nu}(2\alpha,2\beta)\pa^\alpha_\beta f_\pm\big)_{L^2_{x,v}}\\
			- \big(\pa_\beta L_\pm \pa^\alpha f,(e^{\pm\psi}-1)w_{2l,2\nu}(2\alpha,2\beta)\pa^\alpha_\beta f_\pm\big)_{L^2_{x,v}}\\
			= \big(\pa^\alpha_\beta\Gamma_{\pm}(f,f),e^{\pm\psi}w_{2l,2\nu}(2\alpha,2\beta)\pa^\alpha_\beta f_\pm\big)_{L^2_{x,v}}
			+\big(\pa^\alpha_\beta h,e^{\pm\psi}w_{2l,2\nu}(2\alpha,2\beta)\pa^\alpha_\beta f_\pm\big)_{L^2_{x,v}}.
		\end{multline*}
		Following the similar argument from \eqref{71} to \eqref{8.24}, applying smallness \eqref{b156a}, using the estimate \eqref{36a} for $A_\pm$ to replace the estimates on $L_\pm$ and macroscopic estimates, we have 
		\begin{multline}\label{10.3}
			\pa_t\E_\nu(t) + \lambda \D_\nu(t) \lesssim \|\pa_t\psi\|_{L^\infty_x}\E_\nu(t) + \sum_{|\alpha|+|\beta|\le 3} \|w_{l,\nu}(\alpha,\beta)\pa^\alpha_\beta h\|^2_{L^2_{x,v}} + \E_\nu(t)\\ + \sum_{|\alpha|+|\beta|\le 3} \|\pa^\alpha_\beta h\|^2_{L^2_{x}L^2_{D,w_{l,\nu}(\alpha,\beta)}}\sqrt{\E_\nu(t)\D_\nu(t)}. 
		\end{multline}
		By using \eqref{b156a}, we have 
		$
		\|\pa_t\psi\|_{L^\infty_x}\lesssim \ve_0. 
		$
		Then solving \eqref{10.3}, we obtain 
		\begin{align}\label{10.4}
			\sup_{0\le t\le T}\big(e^{Ct}\E_\nu(t)\big) +\lambda\int^T_0\D_\nu(t)\,dt \lesssim  T\sum_{|\alpha|+|\beta|\le 3} \sup_{0\le t\le T}\|w_{l,\nu}(\alpha,\beta)\pa^\alpha_\beta h\|^2_{L^2_{x,v}},
		\end{align}
		for some large constant $C>0$. Since \eqref{b154} is a linear equation, with \eqref{10.4} in hand, it's standard to apply the theory for linear evolution equation to find the local-in-time existence for \eqref{b154}. In fact, we can obtain the local-in-time solution $(f,\phi)$ to \eqref{b154} with estimate: for some $T_0>0$, 
		\begin{align*}
			\sup_{0\le t\le T_0}\big(e^{Ct}\E_\nu(t)\big) +\lambda\int^{T_0}_0\D_\nu(t)\,dt \lesssim T_0\sum_{|\alpha|+|\beta|\le 3} \sup_{0\le t\le T_0}\|w_{l,\nu}(\alpha,\beta)\pa^\alpha_\beta h\|^2_{L^2_{x,v}}. 
		\end{align*}
		Similar to \eqref{8.34}, we can obtain 
		\begin{align*}
			\sup_{0\le t\le T_0}\|\pa_t\phi\|_{L^\infty_x}\lesssim \sup_{0\le t\le T_0}\sqrt{\E(t)} \lesssim  T^{1/2}_0\sup_{0\le t\le T_0}\|w_{l,\nu}(\alpha,\beta)\pa^\alpha_\beta h\|^2_{L^2_{x,v}}.
		\end{align*}
		The above two estimates implies \eqref{b158a}. This completes Lemma \ref{blocalsolution}.

		\end{proof}
	
	\begin{proof}[Proof of Theorem \ref{blocalsolution}]
		Write $(f_0^\ve,E_0^\ve)$ to be the mollification of $(f_0,E_0)$ as the following. 
				Let $\eta_v$ and $\eta_x$ be the standard mollifier in $\R^3$ and $\Omega$: $\eta_v,\eta_x\in C^\infty_c$, $0\le \eta_v,\eta_x\le 1$, $\int\zeta_vdv=\int\zeta_xdx=1$. For $\varepsilon>0$, let $\eta^\varepsilon_v(v) = \ve^{-3}\zeta_v{(\ve^{-1}v)}$ and $\eta_x^\ve(x)=\ve^{-3}\zeta_x{(\ve^{-1}x)}$. Then we mollify the initial data as $f_0^\ve=f_0*\eta^\ve_v*\eta^\ve_x$, $E^\ve_0=E_0*\eta_x^\ve$. Then 
				\begin{align*}
					\|\pa^\alpha_\beta{f_0^\ve}\|_{L^2_{x,v}}\le \|\pa^\alpha_\beta f_0*\eta^\ve_v*\eta^\ve_x\|_{L^2_{x_1,v}}\lesssim \|\eta_v\|_{L^1_{v}}\|{\eta_x}\|_{L^1_{x}}\|\pa^\alpha_\beta{f_0}\|_{L^2_{x,v}}\le \|\pa^\alpha_\beta{f_0}\|_{L^2_{x,v}}, 
				\end{align*}
				and similarly,
				\begin{align*}
					\|\pa^\alpha{E_0^\ve}\|_{L^2_{x_1}}\le \|\pa^\alpha{ E_0}\|_{L^2_{x}}.
				\end{align*}
			Also, $f^\ve_0\to f^\ve_0$ and $E^\ve_0\to E_0$ in $L^2_{x,v}$ and $L^2_x$ respectively as $\ve\to 0$. 
		We now construct the approximation solution sequence as
		 \begin{align*}
			\{(f^n(t,x,v),\phi^n(t,x))\}^\infty_{n=0}
		\end{align*} 
	by using the following iterative scheme: 
		\begin{equation*}\left\{
			\begin{aligned}
				&\partial_tf^{n+1}_\pm + v\cdot \nabla_xf^{n+1}_\pm  \pm \frac{1}{2}\nabla_x\phi^n\cdot vf^{n+1}_\pm  \mp\nabla_x\phi^n\cdot\nabla_vf^{n+1}_\pm \\&\qquad\qquad\pm \nabla_x\phi^{n+1}\cdot v\mu^{1/2} - A_\pm f^{n+1} = \Gamma_{\pm}(f^n,f^{n+1})+Kf^n,\\
				&-\Delta_x\phi^{n+1} = \int_{\R^3}(f^{n+1}_+-f^{n+1}_-)\mu^{1/2}\,dv,\\ 
				&f^{n+1}(0,x,v) = f^{\frac{1}{n+1}}_0(x,v),\quad E^{n+1}(0,x)=E^{\frac{1}{n+1}}_0(x),\\ 
				&{f}^{n+1}(t,x,R_xv) = {f}^{n+1}(t,x,v),\text{ on }\gamma_-,\\
				&\partial_{n}\phi^{n+1} = 0,\ \text{ on } x\in\pa\Omega, 
			\end{aligned}\right.
		\end{equation*}for $n=0,1,2,\cdots$, where we set $f^0(t,x,v)=f_0(x,v)$ and $\phi^0$ given by $-\Delta_x\phi^0=\int_{\R^3}(f^0_+-f^0_-)\mu^{1/2}\,dv$ and $\pa_n\phi^0=0$ on $\pa\Omega$. 
		With Lemma \ref{bLem82}, it is a standard procedure to apply the induction argument to show that there exists $\ve_0>0$ and $T_0>0$ such that if 
		\begin{align*}
			\sum_{|\alpha|\le 3}\big(\|{w\partial^\alpha f_0}\|^2_{L^2_{x,v}}+\|{\partial^\alpha E_0}\|^2_{L^2_{x}}\big)\le \ve_0,
		\end{align*}
		then the approximate solution sequence $\{f^n\}$ is well-defined with estimate
		\begin{align}\label{5.8}\notag
			&\sup_{0\le t\le T_0}\E_\nu(f^n,t)+\int^{T_0}_0\D_\nu(f^n,t)\,dt+\|\pa_t{\phi^n}\|_{L^\infty_{T_0}L^\infty_{x}}\\
			&\notag\lesssim \sum_{|\alpha|+|\beta|\le  3}\big(\|{w_{l,\nu}(\alpha,\beta)\partial^\alpha_\beta f_0}\|^2_{L^2_{x,v}}+\|{\partial^\alpha E_0}\|^2_{L^2_{x}}\big)+T_0^{1/2}\sum_{|\alpha|+|\beta|\le 3}\|{w_{l,\nu}(\alpha,\beta)\partial^\alpha_\beta f^{n-1}}\|^2_{L^\infty_{T_0}L^2_{x,v}}\\
			&\lesssim \sum_{k=0}^nT_0^{k/2}\sum_{|\alpha|+|\beta|\le  3}\big(\|{w_{l,\nu}(\alpha,\beta)\partial^\alpha_\beta f_0}\|^2_{L^2_{x,v}}+\|{\partial^\alpha E_0}\|^2_{L^2_{x}}\big) \lesssim \ve_0^2, 
		\end{align}
		by choosing $T_0>0$ small enough, where we write $\E_\nu(f^n,t)$ and $\D_\nu(f^n,t)$ to show the dependence on $f^n$. 
		Notice that $f^{n+1}-f^n$ solves 
		\begin{align*}
			&\partial_t(f^{n+1}_\pm-f^n_\pm) + v\cdot \nabla_x(f^{n+1}_\pm-f^n_\pm)  \pm \frac{1}{2}\nabla_x\phi^n\cdot v(f^{n+1}_\pm -f^n_\pm) \pm \frac{1}{2}(\nabla_x\phi^n-\nabla_x\phi^{n-1})\cdot vf^n_\pm\\
			&\quad \mp\nabla_x\phi^n\cdot\nabla_v(f^{n+1}_\pm -f^n_\pm)
			\mp  (\nabla_x\phi^n-\nabla_x\phi^{n-1})\cdot\nabla_vf^n_\pm
			\pm (\nabla_x\phi^{n+1}-\nabla_x\phi^n)\cdot v\mu^{1/2}\\
			&\quad- A_\pm (f^{n+1}-f^n) = \Gamma_{\pm}(f^n,f^{n+1}-f^n)+\Gamma_{\pm}(f^n-f^{n-1},f^n)+K(f^n-f^{n-1}),
		\end{align*}
		for $n=1,2,3,\cdots$. Using the method for deriving \eqref{b158a} and \eqref{5.8}; see also \cite{Guo2012}, we know that $f^{n+1}-f^n$ is Cauchy sequence with estimate
		\begin{multline*}
			\sup_{0\le t\le T_0}\E_\nu(f^{n+1}-f^n,t)+\int^{T_0}_0\D_\nu(f^{n+1}-f^n,t)\,dt \\\lesssim \sum_{|\alpha|+|\beta|\le  3}\big(\|{w_{l,\nu}(\alpha,\beta)\partial^\alpha_\beta (f^{\frac{1}{n+1}}_0-f^{\frac{1}{n}}_0)}\|^2_{L^2_{x,v}}+\|{\partial^\alpha (E^{\frac{1}{n+1}}_0-E^{\frac{1}{n}}_0)}\|^2_{L^2_{x}}\to 0,\text{ as } n\to\infty.
		\end{multline*}
		Then the limit function $f(t,x,v)$ is indeed a unique local-in-time solution to \eqref{1}, \eqref{specular} and \eqref{Neumann} satisfying estimate \eqref{b169}. 
		For the positivity, we refer to the argument from \cite[Lemma 12, page 800]{Guo2012}; the details are omitted for brevity. The proof of Theorem \ref{blocalsolution} is complete. 	
		\end{proof}
	
	\medskip
	\noindent {\bf Acknowledgements.} 
	D.-Q Deng was supported by Direct Grant from BIMSA. 
	
	\providecommand{\bysame}{\leavevmode\hbox to3em{\hrulefill}\thinspace}
	\providecommand{\MR}{\relax\ifhmode\unskip\space\fi MR }
	\providecommand{\MRhref}[2]{%
		\href{http://www.ams.org/mathscinet-getitem?mr=#1}{#2}
	}
	\providecommand{\href}[2]{#2}

\end{document}